\documentclass[12pt,a4paper]{amsart}
\usepackage{enumerate}

\usepackage{graphicx}
\usepackage{yhmath}
\usepackage{amsmath,verbatim}
\usepackage{amssymb}
\usepackage{pgf,tikz,pgfplots}
\usetikzlibrary{arrows}
\usepackage[utf8]{inputenc}
\usepackage{mathrsfs, hyperref}
\pgfplotsset{compat=1.15}

\theoremstyle{plain}
\newtheorem{thm}{Theorem}[section]
\newtheorem{lemma}[thm]{Lemma}
\newtheorem{corollary}[thm]{Corollary}
\newtheorem{prop}[thm]{Proposition}
\newtoks\prt

\theoremstyle{definition}

\newtheorem{remark}[thm]{Remark}

\newtheorem{definition}[thm]{Definition}

\def\eqn#1$$#2$${\begin{equation}\label#1#2
\end{equation}}
\newcommand{\abs}[1]{\lvert#1\rvert}

\numberwithin{equation}{section}

\def\diam{\operatorname{diam}}
\def\dist{\operatorname{dist}}
\def\loc{\operatorname{loc}}
\def\co{\operatorname{co}}
\def\osc{\operatorname{osc}}

\def\BV{\operatorname{BV}}
\def\L{\mathcal{L}}

\def\P{\mathcal{P}}
\def\G{\mathcal{G}}
\def\H{\mathcal{H}}
\def\Q{\mathcal{Q}}
\def\J{\mathbf{J}}
\def\phi{\varphi}
\def\epsilon{\varepsilon}
\def\eps{\varepsilon}
\def\en{\mathbb N}
\def\wstar{\overset{\ast}{\rightharpoonup}} 

\def\er{\mathbb R}
\def\R{\mathbb R}
\def\bA{\textbf{A}}
\def\bB{\textbf{B}}

\def\oint{-\hskip -13pt \int}

\def\id{\operatorname{id}}

\def\rn{\mathbb R^n}

\newtoks\by
\newtoks\paper
\newtoks\book
\newtoks\jour
\newtoks\yr
\newtoks\pages
\newtoks\vol
\newtoks\publ
\def\ota{{\hbox\vol{???}}}
\def\cLear{\by=\ota\paper=\ota\book=\ota\jour=\ota\yr=\ota
\pages=\ota\vol=\ota\publ=\ota}
\def\endpaper{\the\by, {\the\paper},
\textit{\the\jour} \textbf{\the\vol} (\the\yr), \the\pages.\cLear}
\def\endbook{\the\by, \textit{\the\book}, \the\publ.\cLear}
\def\endprep{\the\by, \textit{\the\paper}, \the\jour.\cLear}
\def\endyearprep{\the\by, \textit{\the\paper}, \the\jour, (\the\yr).\cLear}
\def\name#1#2{#2 #1}

\definecolor{ubqqys}{rgb}{0.29411764705882354,0.,0.5098039215686274}
\definecolor{wwqqzz}{rgb}{0.4,0.,0.6}
\definecolor{qqttzz}{rgb}{0.,0.2,0.6}
\definecolor{ttzzqq}{rgb}{0.2,0.6,0.}
\definecolor{yqqqqq}{rgb}{0.5019607843137255,0.,0.}
\definecolor{uuuuuu}{rgb}{0.26666666666666666,0.26666666666666666,0.26666666666666666}
\definecolor{ttttff}{rgb}{0.2,0.2,1.}
\definecolor{ffqqqq}{rgb}{1.,0.,0.}
\definecolor{qqccqq}{rgb}{0.,0.8,0.}
\definecolor{qqwuqq}{rgb}{0.,0.39215686274509803,0.}
\definecolor{qqzzcc}{rgb}{0.,0.6,0.8}
\definecolor{xfqqff}{rgb}{0.4980392156862745,0.,1.}
\definecolor{qqqqff}{rgb}{0.,0.,1.}
\definecolor{qqffqq}{rgb}{0.,1.,0.}

\def\de0#1{\rule[3pt]{#1}{0.4pt} \hspace{-0.1pt} \rule[3.05pt]{0.05pt}{0.4pt} \hspace{-0.1pt} \rule[3.1pt]{0.05pt}{0.4pt} \hspace{-0.1pt} \rule[3.15pt]{0.05pt}{0.4pt} \hspace{-0.1pt} \rule[3.2pt]{0.05pt}{0.4pt} \hspace{-0.1pt} \rule[3.25pt]{0.05pt}{0.4pt} \hspace{-0.1pt} \rule[3.3pt]{0.05pt}{0.4pt} \hspace{-0.1pt} \rule[3.35pt]{0.05pt}{0.4pt} \hspace{-0.1pt} \rule[3.4pt]{0.05pt}{0.4pt} \hspace{-0.1pt} \rule[3.45pt]{0.05pt}{0.4pt} \hspace{-0.1pt} \rule[3.5pt]{0.05pt}{0.4pt} \hspace{-0.1pt} \rule[3.55pt]{0.05pt}{0.4pt} \hspace{-0.1pt} \rule[3.6pt]{0.05pt}{0.4pt} \hspace{-0.1pt} \rule[3.65pt]{0.05pt}{0.4pt} \hspace{-0.1pt} \rule[3.7pt]{0.05pt}{0.4pt} \hspace{-0.1pt} \rule[3.75pt]{0.05pt}{0.4pt} \hspace{-0.1pt} \rule[3.8pt]{0.05pt}{0.4pt} \hspace{-0.1pt} \rule[3.85pt]{0.05pt}{0.4pt} \hspace{-0.1pt} \rule[3.9pt]{0.05pt}{0.4pt} \hspace{-0.1pt} \rule[3.95pt]{0.05pt}{0.4pt} \hspace{-0.1pt} \rule[4.0pt]{0.05pt}{0.4pt} \hspace{-0.1pt} \rule[4.05pt]{0.05pt}{0.4pt} \hspace{-0.1pt} \rule[4.1pt]{0.05pt}{0.4pt} \hspace{-0.1pt} \rule[4.15pt]{0.05pt}{0.4pt} \hspace{-0.1pt} \rule[4.2pt]{0.05pt}{0.4pt} \hspace{-0.1pt} \rule[4.25pt]{0.05pt}{0.4pt} \hspace{-0.1pt} \rule[4.3pt]{0.05pt}{0.4pt} \hspace{-0.1pt} \rule[4.35pt]{0.05pt}{0.4pt} \hspace{-0.1pt} \rule[4.4pt]{0.05pt}{0.4pt} \hspace{-0.1pt} \rule[4.45pt]{0.05pt}{0.4pt} \hspace{-0.1pt} \rule[4.5pt]{0.05pt}{0.4pt} \hspace{-0.1pt} \rule[4.55pt]{0.05pt}{0.4pt} \hspace{-0.1pt} \rule[4.6pt]{0.05pt}{0.4pt} \hspace{-0.1pt} \rule[4.65pt]{0.05pt}{0.4pt} \hspace{-0.1pt} \rule[4.7pt]{0.05pt}{0.4pt} \hspace{-0.1pt} \rule[4.75pt]{0.05pt}{0.4pt} \hspace{-0.1pt} \rule[4.8pt]{0.05pt}{0.4pt} \hspace{-0.1pt} \rule[4.85pt]{0.05pt}{0.4pt} \hspace{-0.1pt} \rule[4.9pt]{0.05pt}{0.4pt} \hspace{-0.1pt} \rule[4.95pt]{0.05pt}{0.4pt} \hspace{-0.1pt} \rule[5.0pt]{0.05pt}{0.4pt} \hspace{-0.1pt} \rule[5.05pt]{0.05pt}{0.4pt} \hspace{-0.1pt} \rule[5.1pt]{0.05pt}{0.4pt} \hspace{-0.1pt} \rule[5.15pt]{0.05pt}{0.4pt} \hspace{-0.1pt} \rule[5.2pt]{0.05pt}{0.4pt} \hspace{-0.1pt} \rule[5.25pt]{0.05pt}{0.4pt} \hspace{-0.1pt} \rule[5.3pt]{0.05pt}{0.4pt} \hspace{-0.1pt} \rule[5.35pt]{0.05pt}{0.4pt} \hspace{-0.1pt} \rule[5.4pt]{0.05pt}{0.4pt} \hspace{-0.1pt} \rule[5.45pt]{0.05pt}{0.4pt} \hspace{-0.1pt} \rule[5.5pt]{0.05pt}{0.4pt} \hspace{-0.1pt} \rule[5.55pt]{0.05pt}{0.4pt} \hspace{-0.1pt} \rule[5.6pt]{0.05pt}{0.4pt} \hspace{-0.1pt} \rule[5.65pt]{0.05pt}{0.4pt} \hspace{-0.1pt} \rule[5.7pt]{0.05pt}{0.4pt} \hspace{-0.1pt} \rule[5.75pt]{0.05pt}{0.4pt} \hspace{-0.1pt} \rule[5.8pt]{0.05pt}{0.4pt} \hspace{-0.1pt} \rule[5.85pt]{0.05pt}{0.4pt} \hspace{-0.1pt} \rule[5.9pt]{0.05pt}{0.4pt} \hspace{-0.1pt} \rule[5.95pt]{0.05pt}{0.4pt} \hspace{-0.1pt} \rule[6.0pt]{0.05pt}{0.4pt}}	
\def\debst{\mathop{\hspace{5.2pt} ^* \hspace{-11.2pt} \de0{16pt}~}\limits}	

\MakeRobust{\ref}

\makeatletter
\newcommand{\labeltext}[2]{%
	\@bsphack
	\def\@currentlabel{#1}{\label{#2}}%
	\@esphack
}
\makeatother

\def\step#1#2#3{\par \noindent{{\\ \bf Step~\labeltext{#1}{#3}#1. }{\bf #2. }}}

\begin{document}

\title[Classification of area-strict limits of planar BV homeomorphisms]{Classification of area-strict limits of planar BV homeomorphisms}

\author[D. Campbell]{Daniel Campbell}
\address{D.~Campbell: Department of Mathematics, University of Hradec Kr\' alov\' e, Rokitansk\'eho 62, 500 03 Hradec Kr\'alov\'e, Czech Republic} 
\email{daniel.campbell@uhk.cz}

\author[A. Kauranen]{Aapo Kauranen}
\address{A.~Kauranen: Department of Mathematics and Statistics, University of Jyv\"askyl\"a, PL 35, 40014 Jyv\"askly\"an yliopisto, Finland}
\email{aapo.p.kauranen@jyu.fi}

\author[E. Radici]{Emanuela Radici}
\address{E.~Radici: DISIM - Department of Information Engineering, Computer Science and Mathematics, University of L’Aquila, Via Vetoio 1 (Coppito), 67100 L’Aquila (AQ), Italy}
\email{emanuela.radici@univaq.it}

\thanks{The first author was supported by the grant GACR 20-19018Y}


\begin{abstract}
	We present a classification of area-strict limits of planar $BV$ homeomorphisms. This class of mappings allows for cavitations and fractures but fulfil a suitable generalization of the INV condition. As pointed out by J. Ball \cite{B}, these features are expected in limit configurations of elastic deformations. In \cite{PP}, De Philippis and Pratelli introduced the \emph{no-crossing} condition which characterizes the $W^{1,p}$ closure of planar homeomorphisms. In the current paper we show that a suitable version of this concept is equivalent with a map, $f$, being the area-strict limit of BV homeomorphisms. This extends our results from \cite{CKR}, where we proved that the \emph{no-crossing BV} condition for a BV map was equivalent with the map being the m-strict limit of homeomorphisms (i.e. $f_k \debst f$ and $|D_1f_k|(\Omega)+|D_2f_k|(\Omega) \to |D_1f|(\Omega)+|D_2f|(\Omega)$). Further we show that the \emph{no-crossing BV} condition is equivalent with a seemingly stronger version of the same condition.
\end{abstract} 

\maketitle
\section{Introduction}

	Over the past few years the classification of weak and strong limits of Sobolev diffeomorphisms has attracted a lot of interest for its relevance to variational models of nonlinear elasticity and geometric function theory. 
	The pioneering work in the area was by Iwaniec and Onninen \cite{IO} followed by the more recent result of De Philippis and Pratelli \cite{PP}. Thanks to these results, the classification of weak and strong Sobolev limits of Sobolev homeomorphisms in the planar setting is now well understood. 
	
	In the case that $p\geq 2$ the authors of \cite{IO} utilise the approximation techniques of \cite{IKO1,IKO2} to prove that it is exactly monotone maps which characterize the weak closure of $W^{1,p}$ homeomorphisms.
	
	The authors of \cite{PP} however, approach the problem using a different technique introduced in \cite{HP} for the diffeomorphic approximation of $W^{1,1}$ homeomorphisms. They prove that the weak and strong closures of $W^{1,p}$ homeomorphisms coincide for all $1 \leq p <\infty$ (assuming uniform integrability in the case of $p=1$). This result was examined in the case that the mappings in question equal identity on the boundary of a square. When $p<2$ the closure contains mappings with discontinuities and so monotonicity is too restrictive to characterize the class. On the other hand, they, quite surprisingly, proved that the INV condition of M\" uller and Spector \cite{MS} is satisfied by maps that cannot be the weak Sobolev limits of homeomorphisms (see \cite[Section 5.2]{PP}) and the situation is not saved even by the restriction of $J_f>0$-a.e. for the limit maps (see \cite[Section 5.3]{PP}).
	
	Therefore, it was necessary to introduce a new condition which the authors called the \emph{no-crossing} condition. 
	In essence a Sobolev mapping belongs to the (weak or strong) closure of homeomorphisms if and only if its restriction to `almost any' grid of horizontal and vertical lines can be uniformly approximated by continuous injective maps. Let us note it was proved in \cite{CPR} that the condition that the injective approximation of the map on a grid cannot be weakened to the condition where the map can be uniformly approximated by injective maps on a single injective Lipschitz curve. One advantage of the no-crossing condition is that it does not require the map to be defined everywhere but only up to a $\mathcal{H}^1$-negligible set. The necessity of this was demonstrated in \cite{PP} where they construct limits of planar Sobolev homeomorphisms presenting cavitations. Therefore, studying the closure of planar homeomorphisms in the $BV$ setting opens itself as a natural question. As was shown in \cite{CHKR} there are maps in the limit class which exhibit more complicated discontinuities, like fractures.
	
	The continuity and invertibility properties of candidates for energy minimizing deformations in elasticity theory were studied in the pioneering works of Ball \cite{B1,B2}. Thanks to the concept of non-interpenetration of matter, it is natural to minimise in classes of homeomorphisms which satisfy certain boundary conditions. In some models, however, the energy functional does not guarantee the existence of a homeomorphic solution and in these cases one is led to find a larger class which contains the limit maps. While doing so, however, one wants the class to be restrictive enough to demand its maps exhibit required behaviour of elastic deformations; for example, the non-interpenetration of the material. One such condition is the INV condition of M\" uller and Spector from \cite{MS}. This is a kind of monotonicity condition which allows for cavitation, a phenomenon which has been observed experimentally in deformed elastic materials (see \cite[Figure 4]{GL}). Ball proposed generalizing the mathematical model so as to allow for both cavitations and fracture-type singularities. This is motivated by several experimental observations, e.g. on ductile fracture of titanium alloys \cite{Petrinic1,Petrinic2} or in \cite{WS,GW} on elastomers, where the experiments suggest that the strains at the cavity surface produced during cavitation are so large that fracture occurs at the same time.

	One approach to the question of modelling deformations allowing for discontinuities was studied by Henao and Mora Corral by introducing a term in the energy functional which penalizes new surface created by the deformations. In the series of results \cite{HMC1, HMC2, HMC3, HMC4} they prove that the minimizers are one-to-one almost everywhere and can exhibit fractures.
	
	Another approach was proposed by the authors and Hencl in \cite{CHKR}. Our motivation was fuelled by interesting $BV$ relaxation results obtained by Kristensen and Rindler \cite{KR} and Rindler and Shaw \cite{RS} for Dirchlet-type boundary conditions, and by Ba\' ia, Kr\" omer and Kru\v z\' ik \cite{BKK} for Neumann-type conditions. Further motivation came from recent results in \cite{PR1} and \cite{PR2} by the third author and Pratelli on the \emph{strict} and \emph{area-strict} approximation of BV homeomorphisms by diffeomorphisms. We remind the reader that a sequence $f_k: \Omega \to \er^n$ of $BV$ functions converges \emph{strictly} to $f \in BV(\Omega, \er^n)$ if $f_k \to f$ in $L^1(\Omega, \er^n)$ and $|Df_k|(\Omega) \to |Df|(\Omega)$. The sequence converges \emph{area-strictly} if it converges strictly and it is possible to decompose $Df_k$ as the sum of two measures $\mu_k + \nu_k$ such that $|\mu_k - D^a f|(\Omega) \to 0$ and $|\nu_k|(\Omega) \to |D^s f|(\Omega)$, where $D^a f$ and $D^s f$ denote the absolutely and the singular part of $Df$ respectively.

	In this paper we study classes of area-strict limits and the so-dubbed `m-strict' limits of $BV$ homeomorphisms. We say that a sequence of $BV$ maps $f_k$ converges m-strictly to $f$ on $\Omega$ if $f_k \to f$ in $L^1(\Omega)$ and
	\begin{equation}\label{MStrictDef}
		|D_1f_k|(\Omega) + |D_2f_k|(\Omega) \to |D_1f|(\Omega)+ |D_2 f|(\Omega).
	\end{equation}
	It is obvious that area-strict convergence implies strict convergence which in turn implies m-strict convergence. It is not hard to construct examples which show that each convergence is sharply weaker than the previous.
	
	In \cite{CHKR} it was shown that strict limits of planar $BV$ homeomorphisms can exhibit cavities and fractures but still preserve a sort of monotonicity property. Although there is not enough topological information to meaningfully generalize the INV condition, a type of topological image can be defined and it was shown in \cite{CHKR} that the intersection of this topological image of two disjoint sets has zero measure. As remarked in \cite{CKR}, the proof of this fact in \cite{CHKR} used only the m-strict convergence (not strict convergence).
	
	It was proved in \cite{CKR} that the m-strict limits of BV homeomorphisms satisfy the $NCBV$ condition (see Definition~\ref{DefNCBV} for the precise definition and the following paragraph for an intuition, it is a generalization of the NC condition of \cite{PP}) and any map satisfying the $NCBV$ condition can be approximated m-strictly by BV homeomorphisms. Given these facts, there were grounds to claim in \cite{CKR} that the class of m-strict limits of homeomorphisms is an appropriate class within which to conduct $BV$ relaxations which are physically relevant in elasticity. On the other hand, as in the Sobolev case strong and weak closure of diffeomorphisms coincide, also in $BV$ the counterpart of strong and weak closure should behave the same way. The strong closure of diffeomorphisms does not say anything meaningful for elasticity, hence the strong counterpart has to be understood in this weaker sense of area-strict convergence. Roughly speaking, area-strict convergence is equivalent to strong convergence in those portions of the domain where the singular part of the derivative is small, while is equivalent to strict convergence where the singular part is concentrated. These classes are also an attractive option since they are much more widely used and studied. We prove that in fact, somewhat surprisingly the classes coincide (although we do not claim that an m-strict converging sequence must also converge strictly). An extra advantage of this result is that being an m-strict limit may be an easier condition to check since one can restrict oneself to behaviour on almost all lines reducing the dimension of the problem.
	
	We shall now endeavour to give the reader a rough idea of what the $NCBV$ condition is. We consider a BV map $f$ defined on the unit square. On almost every horizontal and vertical line the restriction of a $BV$ map is a $BV$ map from the line. It is not hard to prove that for $\L^2$-almost every point $X$ of the domain the restriction of the BV map to the horizontal and vertical line intersecting $X$ is continuous at  $X$. Call $\Gamma$ the union of a finite number of horizontal and vertical lines on each of which the restriction of $f$ is BV and $f_{\rceil \Gamma}$ is continuous at the intersection point of any two lines. There is a countable number of points outside of which $f_{\rceil \Gamma}$ is continuous and each of these points is a jump and lies on exactly one horizontal or vertical segment. We declare the topological image of each point $X\in \Gamma$ to be the segment connecting the one-sided limits of $f_{\rceil\Gamma}$ at $X$. Then the topological image of each of the horizontal and vertical lines is now a Lipschitz curve. We parametrize this Lipschitz curve by some Lipschitz mapping from the grid $\Gamma$ and call this map the geometric representative of $f$. We say that $f$ satisfies the $NCBV$ condition if for every such $\Gamma$ the geometric representative of $f_{\rceil\Gamma}$ can be uniformly approximated by a continuous injective map. See section~\ref{BVonGrid} for our concept of restricting a $BV$ map onto a grid and the precise definition of the NCBV condition. The $NCBV^+$ condition is essentially the same as the $NCBV$ condition but the lines do not have to be only vertical and horizontal, instead, $\Gamma$ is the union of piecewise linear injective paths that intersect each other at most once and the intersection point of any pair is distinct from another pair.
	
	In the following we denote $Q(c,r)$ the square the square centered at $c$ with side $2r$, in particular $Q(0,1)$ denotes $(-1,1)^2$. Our main result is the following:

	\begin{thm}\label{main}
		Let $f\in BV(Q(0,1); Q(0,1))$ and $f(x,y) = (x,y)$ for every $(x,y)\in \partial Q(0,1)$. Then the following conditions are equivalent
		\begin{enumerate}
			\item $f$ satisfies the $NCBV$ condition,
			\item $f$ satisfies the $NCBV^+$ condition,
			\item there exists a sequence $f_k\in BV(Q(0,1), Q(0,1))$ of diffeomorphisms with $f_k = \id$ on $\partial Q(0,1)$ converging to $f$ area-strictly,
			\item there exists a sequence $f_k\in BV(Q(0,1), Q(0,1))$ of diffeomorphisms with $f_k = \id$ on $\partial Q(0,1)$ converging to $f$ weakly in BV and $|D_1 f_k|(Q(0,1)) + |D_2 f_k|(Q(0,1)) \to |D_1 f|(Q(0,1)) + |D_2 f|(Q(0,1))$.
		\end{enumerate}
	\end{thm}

	\subsection{Overview of the proof}
	Since (3) implies (4) is obvious and since (4) implies (1) has been proved in \cite{CKR}, there are two implications to prove, i.e. (1) implies (2) and (2) implies (3).
	
	To prove (1) implies (2) we assume that we have a good non-straight grid for $f$ called $\Gamma$. Our aim is to create a system of horizontal and vertical segments on which $f$ remains close to $f$ on $\Gamma$. The non-straight grid $\Gamma$ is the union of $\gamma_i([0,1])$, where $\gamma_i$ are piecewise linear, injective and continuous. We call $X_{i,j}$ the unique element of $\gamma_i([0,1])\cap \gamma_{j}([0,1])$ whenever the intersection is non-empty.
	
	Our first step is to eliminate the intersection points $X_{i,j}$ of $\Gamma$. These points are chosen so that $\lim_{r\to 0}r^{-1}|Df|(Q(X_{i,j},r)) = 0$. We choose a small $r>0$ and 4 disjoint spirals (see Figure~\ref{Fig:Spiral}) with which we replace part of each segment of $\Gamma$ ending at $X_{i,j}$. By choosing the lines of the spiral carefully and by choosing the $r$ small enough we guarantee that the oscillation of $f_{\rceil \Gamma\cap Q(X_{i,j},r)}$ and the oscillation of $f$ on the spiral are both much smaller than some $\sigma$. Then any injective approximation of $f$ on the spiral with error $\sigma /2$ is also an injective approximation of $f$ on $\Gamma \cap Q(X_{i,j},r)$ with error bounded by $\sigma$. This step is Proposition~\ref{Crossroads}.
	
	For all $X\in \Gamma$ except for a finite number of points where there are `large' jumps of $f_{\rceil \Gamma}$ we have $\limsup_{r\to 0}r^{-1}|Df|(Q(X,r))\ll \sigma$. We can then find a rectangle $R_X$ such that the oscillation of $f$ on $\Gamma \cap R_X$ and the oscillation of $f$ on $\partial R_X$ are both much smaller than $\sigma$. Then any injective approximation of $f$ on an appropriate part of $\partial R_X$ with error $\sigma /2$ is also an injective approximation of $f$ on $\Gamma \cap R_X$ with error bounded by $\sigma$. This step is conducted in the proof of Theorem~\ref{ItsAllLies} by applying Lemma~\ref{NoName}.
	
	It then remains to deal with the finite number of `large' (say larger than $\frac{\sigma}{40}$) jumps of $f_{\rceil \Gamma}$. But these jump points are Lebesgue points of the polar decomposition of $D^jf$ and so it is not hard to find a point very close by which has a very similar jump and construct a piece-wise horizontal and vertical path with endpoints on $\Gamma$ and doing basically just the same jump as $f$. See Figure~\ref{Fig:JumpPoints} for a depiction of how we do this. This step is conducted in the proof of Theorem~\ref{ItsAllLies}. After having done the above 3 steps we have an admissible set of horizontal and vertical segments from which we generate a good straight grid $\tilde{\Gamma}$.
	
	At this point it is easy to generate an injective approximation of the geometric representative of $f$ with error $\sigma$ using the $NCBV$ condition with error sufficiently smaller than $\sigma$ and finding an appropriate correspondence of a subset of $\tilde{\Gamma}$ and $\Gamma$.
	
	The proof that (2) implies (3) is quite involved. Some techniques developed in \cite{CKR} are also implemented here. In comparison with the approximation in \cite{CKR} we have to approximate $D^af$ in $L^1$ and in order to do so we have to use the techniques developed in \cite{HP} (and following papers like \cite{C}). Further the m-strict convergence in \cite{CKR} is strictly weaker than even strict convergence and therefore it is necessary to find a better way to approximate the singular part of the derivative. The new extension result needed for this approximation has been developed in \cite{CKRExt}. It is a rotated version of the main result in \cite{PR1}, not only for rectangles but also for convex polygons. Let us now give a short sketch of how we use these results for the proof.
	
	First we find a small set supporting the vast majority of $|D^sf|$. By dividing this set into very small squares and using a Lebesgue point-type argument for the polar decomposition of $D^sf$ we can apply Theorem~\ref{Componentwise Extension} which is taken from \cite{CKRExt}. Thanks to Alberti's rank-one theorem and the Lebesgue point-type estimate, for each square $Q_i$ we have a unit vector $v_i$ such that $D^sf_{\rceil Q_i}$ is very close to $u_i\otimes v_i |D^sf_{\rceil Q_i}|$ for some unit vector $u_i$. This enables us to make estimates like $ |Df|(Q_i)\leq(1+\epsilon) |D^sf|(Q_i) \leq (1+2\epsilon)|\langle D^sf, v_i\rangle|(Q_i)$ which means that Theorem~\ref{Componentwise Extension} gives an estimate of the energy of our approximation on $Q_i$ by $(1+C\epsilon)|D^sf|(Q_i)$.
	
	From now on we work on the rest of $Q(0,1)$ where the energy of $|D^sf|$ is already very small. We separate this set into very small squares and use a Lebesgue point-type argument for the derivatives to find 4 sets, call them $A_1,A_2,A_3, A_4$. We have that $|Df|(A_1)$ is very small because on every square in $A_1$ we have that $|Df|(Q_i) \leq \epsilon \L^2(Q_i)$. The set $A_2$ is very small and so $|Df|(A_2)$ is also very small. The set $A_3$ is made of squares where $f$ is very close to a nice affine mapping and $A_4$ is made of squares where $f$ is very close to an affine mapping with zero Jacobian. We slightly shift the vertexes of the squares so that the behaviour of $f$ on the boundary of the resulting convex quadrilaterals is similar to the behaviour of $f$ on the quadrilateral (by the BV on lines characterization). This allows us to use the extension theorems from \cite{HP} on quadrilaterals in $A_1$ and $A_2$ (Theorem~\ref{HPExt}), and $A_3$ (Theorem~\ref{NullExt}) while in $A_4$ it is enough to do a straight forward triangularisation. There is then no problem in proving that the derivative of our approximation is close to $D^af$ in $L^1$.

\section{Preliminaries}

\subsection{Extension Theorems}
The following is from \cite[Theorem 2.1]{HP}.
\begin{thm}\label{HPExt}
	Let $\phi : \partial Q \to \er^2$ be a piecewise linear and one-to-one function. There is
	a finitely piecewise affine homeomorphism $g : Q \to \er^2$ such that $g = \phi$ on $\partial Q$, and
	$$
		\int_{Q}|Dg| \leq C\diam Q \int_{\partial Q}|D_{\tau}\phi|.
	$$
\end{thm}

The following theorem is \cite[Theorem 3.7]{C}. The question is how to approximate a map which is close to a degenerate linear map $\Phi$ (up to a rotation in the pre-image we may assume that $\Phi = \left(\begin{matrix}
d \; 0\\
0 \; 0\\
\end{matrix}\right)$).

\begin{thm}\label{NullExt}
	Let $d> \delta >0$, let $r_0 \in(0,1)$ and let $\Q$ be a convex set and the image of $[0,r_0]^2$ in a $2$-bi-Lipschitz mapping which is equal to an affine mapping on $\co\{(0,0), (0,r_0),(r_0,0)\}$ and $\co\{(r_0,r_0), (0,r_0),(r_0,0)\}$. Then for every $\varphi:\partial Q\to\er^2$ finitely piece-wise linear and one-to-one mapping with
	\begin{equation}\label{bound1}
	\int_{\partial \Q} \Big|D_{\tau} \varphi(t)-\left(\begin{matrix}
	d \; 0\\
	0 \; 0\\
	\end{matrix}\right)\tau\Big| \; d\H^1(t)< \delta r_0 ,
	\end{equation}
	and $\|D_{\tau}\varphi\|_{L^{\infty}(\partial Q)}\leq d+2\delta$,
	there exists a finitely piece-wise affine homeomorphism $g:Q\to \er^2$ such that $g=\varphi$ on $\partial \Q$ and
	\begin{equation}\label{CoolAssEquation}
	\Big\| Dg(x)-\left(\begin{matrix}
	d \; 0\\
	0 \; 0\\
	\end{matrix}\right)\Big\|_{L^1(\Q)} < C\delta r_0^2 .
	\end{equation}
\end{thm}
	Let $\Q_i$ be a convex quadrilateral. For a unit vector $v_i$ we denote $\pi_{v_i}(x) = v_i^{\bot}\langle x,v_i^{\bot}\rangle = x - v_i\langle x,v_i\rangle $ and $\pi_{v_i^{\bot}}(x) = v_i\langle x,v_i\rangle = x - v_i^{\bot}\langle x,v_i^{\bot}\rangle $. For each $X\in \pi_{v_i}(\operatorname{int}\Q_i)$ there are exactly two distinct points $X_*, X^*$ in $\partial \Q_i$ such that $\pi_{v_i(X_*)} =  \pi_{v_i(X^*)} =  X$ and similarly $\pi_{v_i^{\bot}}(Z_*) = \pi_{v_i^{\bot}}(Z^*) = Z \in \pi_{v_i^{\bot}}(\Q_i)$. To be specific we denote $X_*, X^*$ so that $\langle X_*, v_i\rangle <\langle X^* , v_i\rangle$ and $Z_*, Z^*$ so that $\langle Z_*, v_i^{\bot}\rangle <\langle Z^* , v_i^{\bot}\rangle$.
	Given $\phi: \partial \Q \to \R^2$ a continuous injective map, we call $\P$ the bounded connected component of $\R^2 \setminus \phi(\partial \Q)$ identified by the Jordan curve $\phi(\partial \Q)$. For any pair of points $\bA,\bB \in \P$ we denote by
	\begin{equation}\label{dDef}
	d_{\P}(\bA,\bB) \text{ the geodesic distance between $\bA$ and $\bB$ inside $\P$. }
	\end{equation}
	The following is \cite[Theorem 1.2]{CKRExt}.

	\begin{thm}\label{Componentwise Extension}
		Let $\theta \in [0,2\pi)$ be fixed and let $v_{\theta} = (\cos\theta, \sin\theta)$, $\Q \subset \R^2$ be a convex polygon and $\phi: \partial \Q \to \R^2$ be a continuous piecewise linear injective map.  Then for every $\eps >0$ there exists a finitely piecewise affine homeomorphism $g: \Q \to \R^2$ extending $\phi$, such that 
		\begin{equation}\label{eqComponentwise extension}
			\begin{aligned}
				| \langle Dg , v_{\theta}\rangle|(\Q) &\leq\int_{\pi_{v_i}(\Q)} d_\P (\phi(X^*), \phi(X_*)) d\H^1(X)+ \epsilon \text{ and }\\
				| \langle Dg , v_{\theta}^{\bot}\rangle|(\Q) &\leq  \int_{\pi_{v_i^{\bot}}(\Q)} d_\P (\phi(Z^*), \phi(Z_*)) d\H^1(Z)+ \epsilon . 
			\end{aligned}
		\end{equation}
	\end{thm}

\subsection{Properties of BV maps}

	In the following we repeat some necessary results from the structure theory of BV mappings.
	
	The following definition is \cite[Definition~2.40]{AFP}
	\begin{definition}[Tangent measures]
		We define the set of tangent measures to a $\er^m$ valued Radon measure $\mu$ on $\rn$ at $x\in \rn$ as the set of all finite Radon measures on $B(0,1)$, which are weak* limits of
		$$
		\frac{\mu(B(x,r))}{|\mu(B(x,r))|}
		$$
		as $r\to 0$.
	\end{definition}
	The following definition is \cite[Definition~2.79]{AFP}. 
	\begin{definition}[Approximate tangent spaces]
		Let $x\in G\subset \rn$ be open and let $\mu\in \mathcal{M}(G, \er^m)$. We say that $V$, a $k$-dimensional vector subspace of $\rn$, is the approximate tangent space of $\mu$ with multiplicity $\lambda \in \er$ if
		$$
			\frac{\mu(B(x,r))}{r^k} \wstar \lambda \H^k_{\rceil V}
		$$
		as $r\to 0$ and we denote this as $\operatorname{Tan}^k(\mu , x):=  \lambda \H^k_{\rceil V}$.
	\end{definition}

    Let us emphasize that through out the paper we use the standard notation $\H^k$ to denote the $k$-dimensional Hausdorff measure. 
	The following is \cite[Theorem~2.81, (b)]{AFP}
	\begin{thm}[Strict convergence of approximate tangent spaces]\label{StrictForTanMeas}
		Let $\mu$ be a $\er^m$ valued Radon measure on $\rn$ and assume that $\limsup_{r\to 0^+}r^{-k}|\mu|(B(x,r))<\infty$. Call $f$ the function of the polar decomposition of $\mu$ i.e. $\mu = f|\mu|$ and assume further that $x$ is a Lebesgue point of $f$ with respect to $|\mu|$. Then
		$$
			\nu = \operatorname{Tan}^k(\mu , x)  \ \text{ if and only if } \ |\nu| = \operatorname{Tan}^k(|\mu| , x).
		$$
	\end{thm}

	The following definition is \cite[Definition~3.67]{AFP}
	\begin{definition}[Approximate jump points]
		Let $f\in L^1_{\loc}(G, \er^m)$ and let $x\in G\subset \rn$ be open. We say that $x$ is an approximate jump point of $f$ if there exist $a\neq b \in \er^m$ and a $v\in \rn$, $|v|=1$  such that
		$$
		 \lim_{r\to0} \oint_{B(x,r,v,+)}|f(y) - a| d\L^n(y) = 0 \quad \text{ and } \quad \lim_{r\to0} \oint_{B(x,r,v,-)}|f(y) - b| d\L^n(y)= 0
		$$
		where $B(x,r,v,+) : = \{y \in B(x,r); \langle y-x,v\rangle>0\}$ and $B(x,r,v,-) : = \{y \in B(x,r); \langle y-x,v\rangle<0\}$. 
	\end{definition}
	We denote $f^+(x) := a$ and $f^-(x) := b$ and up to fixing the orientation of the vector $v$ the notation is unique. We denote the set of all jump points of $f$ as $J_f$.
	
	The following is the Federer-Vol'pert theorem as in \cite[Theorem~3.78]{AFP}
	\begin{thm}[Federer-Vol'pert]\label{FV}
		For any $f\in \BV(\Omega, \er^m)$ the discontinuity set $S_f$ is countably $\H^{n-1}$-rectifiable and $\H^{n-1}(S_f\setminus J_f) = 0$. Moreover $Df_{\rceil J_f} = (f^+ - f^-)\otimes v\H^{n-1}_{\rceil J_f}$ and
		$$
			\begin{aligned}
				\operatorname{Tan}^{n-1}(J_f , x) &=  (v(x))^{\bot}\\
				\operatorname{Tan}^{n-1}(|Df|_{\rceil J_f} , x) &= |f^+(x)-f^-(x)| \H^{n-1}_{\rceil (v(x))^{\bot}}
			\end{aligned}
		$$
		for $\H^{n-1}$-almost every $x \in J_f$.
	\end{thm}

	The following is Alberti's Rank one theorem from \cite{A} as in \cite[Theorem~3.94]{AFP}
	\begin{thm}[Rank one]\label{ARO}
		Let $f\in \BV(\Omega, \er^m)$ and $g$ is the function such that $Df = g|Df|$ then there exist unit vectors $u,v$ such that $g(x) = u(x)\otimes v(x)$ for $|D^sf|$-almost every $x\in \Omega$.
	\end{thm}

	\begin{lemma}\label{Stupido}
		Let $f\in \BV(\Omega)$ and let $A\subset \Omega$ be open then $|\langle Df, v\rangle|(A) \leq |Df|(A)$ for any $|v|=1$.
	\end{lemma}
	\begin{proof}
		By \cite[Theorem 5.2]{EG} we have a sequence of smooth functions $f_k$ strictly converging to $f$ in $BV(\Omega)$. For every unit vector $v$ it is obvious that $|\langle Df_k(x), v\rangle| \leq |Df_k(x)|$ everywhere in $\Omega$. Now integrating over $A$ and using the lower semi continuity of the variation we conclude.
	\end{proof}

\section{The NCBV property for $BV$ maps}

\subsection{$BV$ on grids}\label{BVonGrid}

In \cite{PP} the authors introduced a property called the $NC$ condition that characterises the limits of $W^{1,p}$ homeomorphisms from $Q(0,1)$ onto $Q(0,1)$ equalling the identity on the boundary. In \cite{CKR} we generalized this condition for BV maps, which can even fail to be continuous on a 1-rectifiable set. The essence of these conditions is that a certain map can be approximated uniformly by a continuous injective map with error arbitrarily small. Our approach requires choosing a finite number of points where $f$ is continuous which is in the following proposition.

\begin{prop}\label{DebileDebileDebileDot}
	Let $f\in BV(Q(0,1),\er^2)$ and let $v_1, v_2, v_3, v_4 \in \er^2$ with $|v_i| =1$. Then for almost every $(x,y) \in Q(0,1)$ it holds that 
	\begin{equation}\label{Yesterday}
		f_{\rceil (x,y) + \bigcup_{i} v_i\er} \text{ is continuous at } (x,y)
	\end{equation}
	and 
	\begin{equation}\label{Rainbow}
		\lim_{r\to 0} r^{-1}|Df|\big(Q((x,y),r)\big) =0.
	\end{equation}
\end{prop}
\begin{proof}
	The claim \eqref{Yesterday} follows from \cite[Theorem 3.107]{AFP} and \eqref{Rainbow} follows from \cite[Proposition 3.92]{AFP}.
\end{proof}

\begin{corollary}\label{Repeat}
	Let $f\in BV(Q(0,1),\er^2)$, There exists a set $N_1, N_2 \subset[-1,1]$, $\L^1(N_1) = \L^1(N_2) = 0$ such that
	\begin{enumerate}
		\item $f_{\rceil \{x\}\times [-1,1]}$ is BV on $\{x\}\times [-1,1]$,
		\item $f_{\rceil [-1,1]\times\{y\}}$ is BV on $[-1,1]\times\{y\}$ and
		\item$f_{\rceil \{x\}\times [-1,1] \cup [-1,1]\times\{y\}}$ is continuous at $(x,y)$
	\end{enumerate}
	for $\L^1$-almost every $y\in [-1,1]^2$ if $x\in [-1,1]\setminus N_1$ and for $\L^1$-almost every $x\in [-1,1]^2$ if $y\in [-1,1]\setminus N_2$.
\end{corollary}
\begin{proof}
	The claim follows from Proposition~\ref{DebileDebileDebileDot} and the Fubini theorem.
\end{proof}

Let
\begin{equation}\label{BiPolar}
\big(f^+(\cdot)-f^-(\cdot)\big)\otimes v(\cdot)\H^1_{\rceil J_f} = D^jf
\end{equation}
be the standard decomposition of $D^jf$ mentioned in Theorem \ref{FV}. Also by Theorem~\ref{FV} for $\H^1$ almost every $x\in J_f$ it holds that
\begin{equation}\label{StealAndBorrow}
\frac{1}{r}\int_{B(x,r)\cap J_f} |(f^+(y)-f^-(y)\big)\otimes v(y) - (f^+(x)-f^-(x)\big)\otimes v(x)| \, d\H^1_{\rceil J_f}(y) \to 0
\end{equation}
and we call these points Lebesgue points of $(f^+(\cdot)-f^-(\cdot)\big)\otimes v(\cdot)$.

\begin{definition}[Admissible curves for $f$]\label{Admissible}
	Let $f \in BV(Q(0,1), Q(0,1))$. Let $\gamma_{i}:[0,1]\to Q(0,1)\subset \er^2$, $i=1,\dots, K$ be finitely piecewise linear mappings with $\gamma_{i}([0,1])\cap \gamma_{j}([0,1])$ contains at most one point $\{X_{i,j}\}$ for all $i\neq j$. Let $\Gamma = \bigcup_{i=1}^{K} \gamma_{i}([0,1])\subset Q(0,1)$. We call $\Gamma$ admissible for $f$ if
	\begin{enumerate}
		\item $f\circ\gamma_i \in BV((0,1), \er^2))$ for all $i$,
		\item $f_{\rceil \Gamma}$ is continuous at $\gamma_{i}(s_{i,k})$ at every point $s_{i,k}$, where $\{s_{i,k}\}$ is the finite set of endpoints of intervals on which $\gamma_i$ is linear,
		\item $f_{\rceil \Gamma}$ is continuous at each $X_{i,j}$ for every $i,j$ such that $\gamma_{i}([0,1])\cap \gamma_{j}([0,1]) \neq \emptyset$,
		\item$\lim_{r\to 0} r^{-1} |Df|\big(Q(\gamma_i(s_{i,k}), r)\big) = 0$ for each $s_{i,k}$,
		\item$\lim_{r\to 0} r^{-1} |Df|\big(Q(X_{i,j}, r)\big) = 0$ for each $X_{i,j}$.
	\end{enumerate}
\end{definition}

Note that by Proposition~\ref{DebileDebileDebileDot} `almost every' polyline is admissible.

\begin{definition}[Good straight grid]\label{GSGDef}
	Let $f \in BV(Q(0,1), Q(0,1))$ and let $\Gamma = (\bigcup_{i=1}^K\{x_i\}\times[-1,1])\cup(\bigcup_{j=1}^K[-1,1]\times\{y_j\})$ be admissible for $f$ in the sense of Definition~\ref{Admissible}. We call $\Gamma$ a good straight grid for $f$ if
	\begin{enumerate}
		\item every point of $\Gamma \cap J_f$ is a Lebesgue point of $\big(f^+(\cdot)-f^-(\cdot)\big)\otimes v(\cdot)$ in the sense of \eqref{StealAndBorrow},
		\item $\langle(1,0), v(x,y_j)\rangle\neq 0$, $\langle(0,1), v(x_i,y)\rangle\neq 0$ for all $x,y\in [-1,1]$ such that $(x_i,y), (x,y_j) \in \Gamma\cap J_f$, where $v(x,y)$ is the vector from Theorem~\ref{ARO}.
	\end{enumerate}
\end{definition}

\begin{definition}
	Let $f \in BV(Q(0,1), Q(0,1))$ and let $\Gamma \subset Q(0,1)$ be the union of a finite number of horizontal and vertical segments. Let $\tilde{\Gamma}$ be the smallest straight grid such that $\tilde{\Gamma} \supset \Gamma$. Then we say that $\tilde{\Gamma}$ is the straight grid generated by $\Gamma$.
\end{definition}

Now we define the object $\Gamma$ which we refer to as a `non-straight grid'.
\begin{definition}[Good non-straight grid]\label{GNSGDef}
	Let $f \in BV(Q(0,1), Q(0,1))$. Let $\gamma_{i}:[0,1]\to Q(0,1)\subset \er^2$, $i=1,\dots, K$ be finitely piecewise linear mappings with $\gamma_{i}([0,1])\cap \gamma_{j}([0,1])$ contains at most one point $\{X_{i,j}\}$ if $i\neq j$. Let $\Gamma = \bigcup_{i=1}^{K} \gamma_{i}([0,1])$ be admissible for $f$. We call $\Gamma$ a good non-straight grid for $f$ if
	\begin{enumerate}
		\item every point of $\Gamma \cap J_f$ is a Lebesgue point of $\big(f^+(\cdot)-f^+(\cdot)\big)\otimes v(\cdot)$ with respect to $\H^1_{\rceil J_f}$,
		\item the derivative $\gamma_i'(t)$ exists and $\langle\gamma_i'(t), v(\gamma_i(t))\rangle\neq 0$ whenever $\gamma_i(t) \in J_f$ where $v(x,y)$ is the vector from Theorem~\ref{ARO}.
	\end{enumerate}
\end{definition}

	\begin{remark}\label{GGG}
		A simple affine change of variables together with Corollary~\ref{Repeat} proves the following. For every pair of distinct directions $v_1, v_2\in \er^2$, $|v_j|=1$ almost every line $L_1$ parallel to $v_1$ yields an admissible $\Gamma = L_1\cup L_2$ for almost every line $L_2$ parallel to $v_2$. Informally we can say that `almost every' non-straight grid is admissible for $f$. In Theorem~\ref{KulovyBlesk} we show that `almost every' non-straight grid is good for $f$.
	\end{remark}

	In this paper we often restrict a planar $BV$ map onto lines or good grids (both straight and non-straight). Since, for almost every line, the restriction of a $BV$ map is one-dimensional $BV$ on the line, it stands to reason, for a good choice of grid $\Gamma$, that $f_{\rceil \Gamma}$ is also in $BV$. On the other hand, the space $BV(\Gamma)$ is not very standard so we now explain what we mean by this.

	Let $f:Q(0,1) \to \er^2$ and let $\Gamma$ be a good non-straight grid. We say that $f_{\rceil\Gamma}$ is $BV$ on $\Gamma$ if
	\begin{enumerate}
		\item $f\circ\gamma_i \in BV([0,1],\er^2)$
		\item $f_{\rceil\Gamma}$ is continuous at each $X_{i,j}$.
	\end{enumerate}
	By $D_{\tau}f$ we denote the measure on $\Gamma$ given by $\sum_{i=1}^KD_{\tfrac{\gamma_i'}{|\gamma_i'|}}f_{\rceil \gamma_i([0,1])}$. The condition (2) above ensures that $\{X_{i,j}: 1\leq i<j\leq K\}$ is a negligible set in $|D_{\tau} f_{\rceil\Gamma}|$. This means that we can interpret $D_{\tau}$ as the distributional derivative tangential to $\Gamma$ outside the set $\{X_{i,j}: 1\leq i<j\leq K\} \cup \{\gamma_i(s_{i,m}): 1\leq i\leq K\} $, where the points $s_{i,m}$ are those defined in Deifnition \ref{Admissible}.

	If $f\in BV(\Gamma)$ and $|D_{\tau}f_{\rceil\Gamma}|$ is absolutely continuous with respect to $\H^1$ then we say that $f_{\rceil \Gamma}$ is in $W^{1,1}(\Gamma)$.

	As a matter of convention, when we write $\partial_{\tau} f_{\rceil \Gamma}(X)$ we refer to the classical partial derivative of the mapping $f$ in the direction tangential to $\Gamma$ at $X$. We avoid writing this at points $X_{i,j}$ to avoid ambiguity. It is known for $f_{\rceil\Gamma}$ in $BV$ on $\Gamma$ (in the sense described above) that $\partial_{\tau}f_{\rceil \Gamma}$ exists almost everywhere and the absolutely continuous part of $D_{\tau}f_{\rceil \Gamma}$ can be represented by $\partial_{\tau}f_{\rceil \Gamma}\H_{\rceil \Gamma}^1$.
	
	Throughout the paper $\nabla f(x)$ denotes the approximative derivative of $f$ at $x$ which for $f\in  BV(Q(0,1),\er^2)$ exists $\L^2$ almost everywhere for the correct representative.

	Let $\gamma:[0,1] \to Q(0,1)$ be a bi-Lipschitz parametrisation of a curve such that $f\circ\gamma \in BV((0,1),\er^2)$. We define $h:\gamma([0,1])\to \er^2$, the geometric representative of $f$ on $\gamma([0,1])$ as explained below. For each $t$ we define
	$$
		\begin{aligned}
			Y(t) &= \lim_{s\to t^-} f\circ\gamma(s),
			\quad &Z(t) = \lim_{s\to t^+} f\circ\gamma(s),&\\
			l(t) &= |Df\circ\gamma|([0,t)),
			\quad &L(t) = |Df\circ\gamma|([0,t]).&
		\end{aligned}
	$$
	We define the a curve $\tilde{h}:[0,1+L(1)] \to \er^2$ as the constant speed parametrization of the segment $[Y(t)Z(t)]$ from $[l(t)+t, L(t)+t]$. Finally, we define
	\begin{equation}\label{defh}
		h(x) = \tilde{h}\big([1+L(1)]\gamma^{-1}(x)\big)
	\end{equation}
	on $\gamma([0,1])$.

	Let $\Gamma$ be a good non-straight grid for $f$ then we define $h$, the geometric representative of $f$ on $\Gamma$, as follows. For each $1\leq i\leq K$ we find finitely many intervals $[s_{i,m}, s_{i,m+1}]$ covering $[0,1]$ such that $\gamma_i$ is linear on each $[s_{i,m}, s_{i,m+1}]$, $f_{\rceil \Gamma}$ is continuous at each $\gamma_i(s_{i,m})$ and for each $X_{i,j}$ there is an $m$ such that $\gamma(s_{i,m}) = X_{i,j}$. We define the map $h$ on each segment $[\gamma_i(s_{i,m})\gamma_i(s_{i,m+1})]$ as in \eqref{defh}. This gives us a map $h$ defined on the whole of $\Gamma$. When we refer to a segment of $\Gamma$ we mean one of the segments $[\gamma_i(s_{i,m})\gamma_i(s_{i,m+1})]$.

	\begin{definition}[No-Crossing $BV$ condition]\label{DefNCBV}
		Let $f \in BV(Q(0,1), \er^2)$ with $f(x) = x$ on $\partial Q(0,1)$. We say that $f$ satisfies the $NCBV$ condition if for every $\sigma >0$ and for every good straight grid $\Gamma$ for $f$, there exists a continuous injective map $H_{\sigma}:\Gamma \to \er^2$ such that $|h(x) - H_{\sigma}(x)| < \sigma$ for all $x\in \Gamma$, where $h$ is the geometric representative\footnote{See the previous paragraphs for the definition.} of $f$ on $\Gamma$.
	\end{definition}

	\begin{definition}[No-Crossing $BV$+ condition]\label{DefNCBVPlus}
		Let $f \in BV(Q(0,1), \er^2)$ with $f(x) = x$ on $\partial Q(0,1)$. We say that $f$ satisfies the $NCBV^+$ condition if for every $\sigma >0$ and for every good non-straight grid $\Gamma$ for $f$, there exists a continuous injective map $H_{\sigma}:\Gamma \to \er^2$ such that $|h(x) - H_{\sigma}(x)| < \sigma$ for all $x\in \Gamma$, where $h$ is the geometric representative\footnote{See the previous paragraphs for the definition.} of $f$ on $\Gamma$.
	\end{definition}

	The following lemma was published in \cite{CKR}. We include its proof here for the convenience of the reader.
	\begin{lemma}\label{Repetition}	
		Let $X,Y\in \R^2$, $\epsilon, \delta\in [0,1]$ and $L:=\abs{X-Y}>0$. Let $C\in B(X,\delta L)$ and $D\in B(Y,\delta L).$
		Let $\eta:[0,1]\to \R^2$ be a path (with constant speed parametrisation) joining points $C$ and $D$ with arc length
		$l(\eta)\leq (1+\epsilon)L.$ Let $\gamma:[0,1]\to \R^2$ be the constant speed parametrisation of the line segment joining $X$ and $Y$. Then for every $t\in [0,1]$ $\abs{\eta(t)-\gamma(t)}\leq \sqrt{3\epsilon+12\delta} L.$
	\end{lemma}
	\begin{proof}
		We assume without loss of generality that $X$ is the origin and $Y=(L,0).$ 
		Fix a point $t\in [0,1].$  We use the following notation $l_1=l(\eta_{\rceil[0,t]})$ and $l_2=l(\eta_{\rceil[t,1]})$ and thus $l_1+l_2=l(\eta)\leq (1+\epsilon)L.$
		At least one of the angles $\sphericalangle(\eta(t),(tL,0),(L,0))$ and $\sphericalangle((0,0),(tL,0),\eta(t))$ is clearly in the range $[\pi/2,\pi].$ We give the proof in the case where $\sphericalangle(\eta(t),(tL,0),(L,0)) \in [\pi/2,\pi].$ The proof in the other case is essentially the same. 
		
		We denote the distance between $\eta(t)$ and $(tL,0)$ by $E$ and the distance between $(0,0)$ and $\eta(t)$ by $D.$ By the assumption $\sphericalangle(\eta(t),(tL,0),(L,0)) \in [\pi/2,\pi]$ we obtain from law of cosines that
		\begin{equation}
		\label{kosinilause}
		l_1+\delta\geq D \geq \sqrt{L^2t^2+E^2}.
		\end{equation}
		
		On the other hand, we have $l_1=l(\eta)t\leq (1+\epsilon)Lt+\delta L.$ Combining this with \eqref{kosinilause} and simplifying results in the estimate
		$$
		E^2\leq (3\epsilon+12\delta)L^2
		$$
	\end{proof}

	\begin{definition}[Good arrival grids]\label{good arrival grid}
		Let the mapping $f\in BV(Q(0,1); Q(0,1))$, let $\Gamma$ be a good starting grid for $f$ and let $\gamma$ be the geometrical representative of $f$ on $\Gamma$. Let $\kappa>0$, let the numbers $-1=w_0<w_1<w_2\, \cdots\, ,\,< w_{N+1}=1$ and $-1=z_0<z_1<z_2\, \cdots < z_{M+1}=1$ satisfy $w_{n+1}-w_n<\kappa$ and $z_{m+1}-z_m<\kappa$ for every $0\leq n\leq N$ and $0\leq m\leq M$. We say that
		\begin{equation}\label{defgag}
		\G = \bigcup_{n=0}^{N+1} \{w_n\}\times [-1,1]\ \cup\ \bigcup_{m=0}^{M+1} [-1,1] \times \{z_m\}\subseteq Q(0,1)
		\end{equation}
		is a \emph{good arrival grid for $f$ associated with $\Gamma$ and with side-length $\kappa$} if $P:=\gamma^{-1}(\G)\cap \Gamma$ is a finite set and for every $p\in P$ it holds that
		\begin{itemize}
			\item[$\cdot$] $p$ is not a cross of the grid $\Gamma$ (i.e. a point $(x_i, y_j)$),
			\item[$\cdot$] $\gamma(p)$ is not a cross of the grid $\G$ (i.e. a point $(w_n, z_m)$),
			\item[$\cdot$] $p$ is a point where the derivative $\partial_{\tau}\gamma(p)$ exists and $\partial_{\tau}\gamma(p) \neq 0$,
			\item[$\cdot$] $\partial_\tau\gamma(p)$ is not parallel to the side of $\G$ containing $\gamma(p)$.
		\end{itemize}
	\end{definition}

	An important fact is that good arrival grids always exist. More precisely, we have the following property, whose proof is a simple variant of the proof of~\cite[Lemma~3.6]{PP} and can be found in~\cite[Lemma~4.4]{CPR}. 
	
	\begin{lemma}\label{ArrivalGrid}
		Let $f\in BV(Q(0,1); Q(0,1))$ and let $\Gamma$ be a good starting grid for $f$ in $Q(0,1)$. Then the geometrical representative $\gamma$ of $f$ on $\Gamma$ is in $W^{1,1}(\Gamma, Q(0,1))$. Moreover, there exists $\bar{\kappa} = \bar{\kappa}(L) >0$ such that for any $0 < \kappa < \bar \kappa$ and any $\Sigma \subset \Gamma$ $\H^1$-negligible set, there exists a good arrival grid $\G$ for $f$ associated with $\Gamma$, with side-length $\kappa$, and such that $\gamma^{-1}(\G)\cap \Sigma=\emptyset$.
	\end{lemma}
	
	We define the concept of the generalized segment, already introduced in \cite{PP}, which will be useful throughout the proof of Theorem \ref{main}.
	
	\begin{definition}[generalized segments]\label{generalized segments}
		Let $\mathcal{G} \subset Q(0,1)\subset \er^2$ be a grid (the finite union of horizontal and vertical lines). Let $R$ be a rectangle of the grid $\Gamma$ (the closure of a component of $Q(0,1) \setminus \Gamma$). Let $X\neq Y$ and $X,Y\in\partial R \subset \G$. Given $\xi >0$ a small parameter, the \emph{generalized segment} $[XY]$ between $X$ and $Y$ in $R$ is defined as the standard segment $[XY]$ if the two points are not in the same side of $\partial R$; otherwise, $[XY]$ is the union of two segments of the form $[XM]$ and $[MB]$ where $M$ is the point inside $R$ whose distance from the side containing $X$ and $Y$ is $\xi|X-Y|/2$ and the projection of $M$ on the segment $[XY]$ is the mid-point of $[XY]$.
	\end{definition}
	
	The following claim about generalized segments holds.
	
	\begin{prop}\label{EverythingIDo}
		Let $R\subset \er^2$ be a rectangle and let $a,b \in \partial R$. Let $S$ be a generalized segment from $a$ to $b$ in $R$ with parameter $\xi >0$ and let $\tilde{S} \subset S$ be a closed and connected subset of $S$. Then
		$$
		\H^1(\tilde S) \leq (1+\xi)\diam(\tilde{S})
		$$
	\end{prop}
	\begin{proof}
		If $S$ is a segment the claim is immediate. In fact, $\H^1(\tilde S) \leq \diam(\tilde{S})$. If $S$ is the union of 2 segments, then after rotation and translation we can interpret $S$ as the graph of the function $\xi|x|$. We have $\diam (\tilde{S}) \geq \diam (\pi(\tilde{S}))$, where $\pi$ is the projection onto the horizontal axis. Using the area formula
		$$
		\H^1(\tilde{S}) = \sqrt{1+\xi^2}\diam(\pi(\tilde{S})) \leq (1+\xi)\diam(\tilde{S}).
		$$ 
	\end{proof}

\subsection{Equivalence of $NCBV$ and $NCBV^+$ conditions}

\begin{prop}\label{Obvious}
	Let $f\in BV(Q(0,1))$ be an $NCBV+$ map, then $f$ is an $NCBV$ map.
\end{prop}
\begin{proof}
	It suffices to see that every good straight grid for $f$ is also a good non-straight grid for $f$, which is evident.
\end{proof}

In the following lemma we refer by $N_1,N_2\subset [-1,1]$ to the null sets from Corollary~\ref{Repeat}. Further, for all $(y_1,y_2,x_1,x_2)\in[-1,1]^4$ we define the (boundary of a) rectangle $R_{(y_1,y_2,x_1,x_2)} = [(x_1,y_1)(x_2,y_1)]\cup[(x_2,y_1)(x_2,y_2)]\cup[(x_2,y_2)(x_1,y_2)]\cup[(x_1,y_2)(x_1,y_1)]$.

\begin{lemma}\label{NoName}
	Let $f\in BV(Q(0,1))$ be an NCBV map and let $\sigma > 0$. Then for every point $(x_0,y_0) \in Q(0,1)$ where
	\begin{equation}\label{Lock}
	\liminf_{r\to 0} r^{-1}|Df|(Q((x_0,y_0),r)) \leq \tfrac{1}{34}\sigma
	\end{equation}
	there exists a sequence of positive numbers $r_n \to 0$ and a sequence of sets $A_n \subset ((y_0-r_n,y_0 -\tfrac{1}{2}r_n)\setminus N_1)\times ((y_0+\tfrac{1}{2}r_n,y_0 +r_n)\setminus N_1)\times((x_0-r_n,x_0 -\tfrac{1}{2}r_n)\setminus N_2)\times ((x_0+\tfrac{1}{2}r_n,x_0 +r_n)\setminus N_2)$ with $\L^4(A_n)>0$ and for any $(y_1,y_2,x_1,x_2) \in A_n$ it holds that the straight grid generated by $R_{(y_1,y_2,x_1,x_2)}$ is good for $f$ and
	\begin{equation}\label{Nuts}
		|D_{\tau}f_{\rceil R_{(y_1,y_2,x_1,x_2)}}|(R_{(y_1,y_2,x_1,x_2)}) < \tfrac{1}{4}\sigma.
	\end{equation}
\end{lemma}
\begin{proof}
	By \eqref{Lock}, we find a sequence $r_n \to 0$ so that $|Df|(Q((x_0,y_0),r_n))< \tfrac{1}{33}\sigma r_n$. We denote $C_{y,n} = [x_0-r_n, x_0+r_n] \times \{y\}$ for any $y\in[y_0-r_n, y_0+r_n]$. By a standard disintegration argument (see \cite[Theorem 2.28]{AFP} and \cite[Theorem 3.107]{AFP}) and Lemma~\ref{Stupido} we have
	$$
	\begin{aligned}
	\frac{r_n\sigma}{33}
	&> |Df|(Q((x_0,y_0),r_n))\\
	&\geq |\langle Df, (1,0)\rangle|(Q((x_0,y_0),r_n))\\
	& = \int_{y_0-r_n}^{y_0+r_n} |D_{\tau}f_{\rceil C_{y,n}}| (C_{y,n}) dy.
	\end{aligned}
	$$
	We use the argument
	$$
	\int_A |D_{\tau}f_{\rceil C_{y,n}}| (C_{y,n})\, dy \geq \lambda \L^1(\{y \in A; |D_{\tau}f_{\rceil C_{y,n}}| (C_{y,n}) \geq \lambda\})
	$$
	for both $A = (y_0 -r_n, y_0-\tfrac{1}{2}r_n)$ and for $A  =(y_0+ \tfrac{1}{2} r_n, y_0+r_n)$ to get
	$$
	\begin{aligned}
	\L^1(\{t \in A; |D_{\tau}f_{\rceil C_{y,n}}| (C_{y,n}) < \lambda\}) &=\frac{1}{2} r_n - \L^1(\{t \in A; |D_{\tau}f_{\rceil C_{y,n}}| (C_{y,n}) \geq \lambda\})\\
	&\geq \tfrac{1}{2}r_n -  \frac{1}{\lambda} \int_A |D_{\tau}f_{\rceil C_{y,n}}| (C_{y,n})\, dy \\
	&> \tfrac{1}{2}r_n - \frac{r_n\sigma}{33\lambda}.
	\end{aligned}
	$$
	Choose $\lambda = \tfrac{1}{16}\sigma$, then
	\begin{equation}\label{Skiing}
	\L^1(\{t \in A; |D_{\tau}f_{\rceil C_{y,n}}| (C_{y,n}) <\tfrac{1}{16}\sigma\}) >0.
	\end{equation}
	The same estimates hold also for vertical lines.
	
	Call $A_n$ the subset of $((y_0-r_n,y_0 -\tfrac{1}{2}r_n)\setminus N_1)\times ((y_0+\tfrac{1}{2}r_n,y_0 +r_n)\setminus N_1)\times((x_0-r_n,x_0 -\tfrac{1}{2}r_n)\setminus N_2)\times ((x_0+\tfrac{1}{2}r_n,x_0 +r_n)\setminus N_2)$ whose each component satisfies \eqref{Skiing}. Then $A_n$ has positive $\L^4$ measure. Without loss of generality we may assume that each point of $A_n$ is a Lebesgue point of the set (with respect to $\L^4$). By Proposition~\ref{Repeat} have that (almost every) $[-1,1]\times\{x_1\}\cup[-1,1]\times\{x_2\}\cup\{y_1\}\times[-1,1]\cup\{y_2\}\times[-1,1]$ is admissible for $f$. 
\end{proof}

\begin{prop}\label{Crossroads}
	Let $f\in BV(Q(0,1),\er^2)$ be an NCBV map, let $\sigma > 0$ and let $(x_0,y_0)\in Q(0,1)$ satisfy
	\begin{equation}\label{Twisted}
	\lim_{r\to 0} r^{-1}|Df|\big(Q((x_0,y_0),r)\big) =0.
	\end{equation}
	Let $v_1, \dots, v_4\in \er^2$ be distinct vectors with $|v_i|  = 1$ and call the rays $\gamma_i = (x_0,y_0) + v_i[0,\infty)$. Then there exists a constant $\delta >0$ depending on the vectors $v_i$ and an $r_0 = r_0(\sigma)>0$ such that for any $0<r\leq r_0$ there exists a set $A_r\subset Q((x_0,y_0),\delta r)$ with $\L^2(A_r) >0$ and for every $(\tilde{x},\tilde{y})\in A_r$ there exist four polylines $\tilde{\gamma}_1, \tilde{\gamma}_2, \tilde{\gamma}_3,\tilde{\gamma}_4 \subset Q((x_0,y_0), r)$ with the following properties
	\begin{enumerate}
		\item the point $(\tilde{x}, \tilde{y})$ is an endpoint of all $\tilde{\gamma}_i$, $i=1,\dots,4$ and the second endpoint of $\tilde{\gamma}_i$ (we call it $(x_i,y_i)$) lies on $\gamma_i \setminus \{(\tilde{x},\tilde{y})\}$,
		\item the point $(\tilde{x}, \tilde{y})$ is the only common point of the curves $\tilde{\gamma}_i\cup (x_i,y_i)+(0,\infty)v_i$ i.e.$\big[\tilde{\gamma}_i\cup \big((x_i,y_i)+(0,\infty)v_i \big)\big]\cap \big[\tilde{\gamma}_j\cup \big((x_j,y_j)+(0,\infty)v_j\big) \big] = \{(\tilde{x}, \tilde{y})\}$ for $1\leq i<j\leq 4$,
		\item the collection $\{\tilde{\gamma}_i \cup (x_i,y_i)+(0,\infty)v_i \}_{i=1}^4$ is admissible for $f$,
		\item all segments of $\tilde{\gamma}_i$ are parallel to either $(1,0)$ or $(0,1)$,
		\item $\osc(f,\tilde{\gamma}_i) \leq \tfrac{1}{4}\sigma$.
	\end{enumerate}
\end{prop}
\begin{proof}
	We separate the plane into 4 quadrants $P_1 = \{(x,y) \in \er^2; x>x_0, y \geq y_0 \}, P_2 = \{(x,y) \in \er^2; x\leq x_0, y > y_0 \},P_3 = \{(x,y) \in \er^2; x<x_0, y \leq y_0 \}$ and $P_4 = \{(x,y) \in \er^2; x\geq x_0, y < y_0 \}$. Notice that each $\gamma_i$ is contained in exactly one of the quadrants. Further observe that there exists a $0<\delta<\tfrac{1}{8}$ such that $(\gamma_i - u)\setminus Q((x_0,y_0),\tfrac{1}{2})$ is contained in exactly one quadrant for any $|u|<\delta$ (in the case of $v_i=(\pm 1,0)$ or $(0,\pm1)$ this may not be the same as the original quadrant). In every case having chosen $|u|<\delta$ we have a uniquely determined quadrant $P_i + u$ containing $\gamma_i$.
	
	We define the polyline `spiral'
	$$
	\begin{aligned}
	\gamma_{1}^*=
	& [(0,0),(1,0)] \cup [(1,0), (1,2)] \cup [(1,2), (-3,2)] \cup [(-3,2),(-3,-4)]\\
	& \cup [(-3,-4)(5,-4)] \cup [(5,-4),(5,6)] \cup  [(5,6), (-7, 6)]
	\end{aligned}
	$$
	and 
	$\gamma_{2}^*, \gamma_{3}^*, \gamma_{4}^*$ are rotations of $\gamma_{1}^*$ by $90, 180$ and $270$ degrees clockwise. See Figure~\ref{Fig:Spiral} for an illustration. We use these polylines as a starting point. In the following we construct $\tilde{\gamma}_i$ from $\gamma_{i}^*$ by translating, rescaling and truncating them.
	
	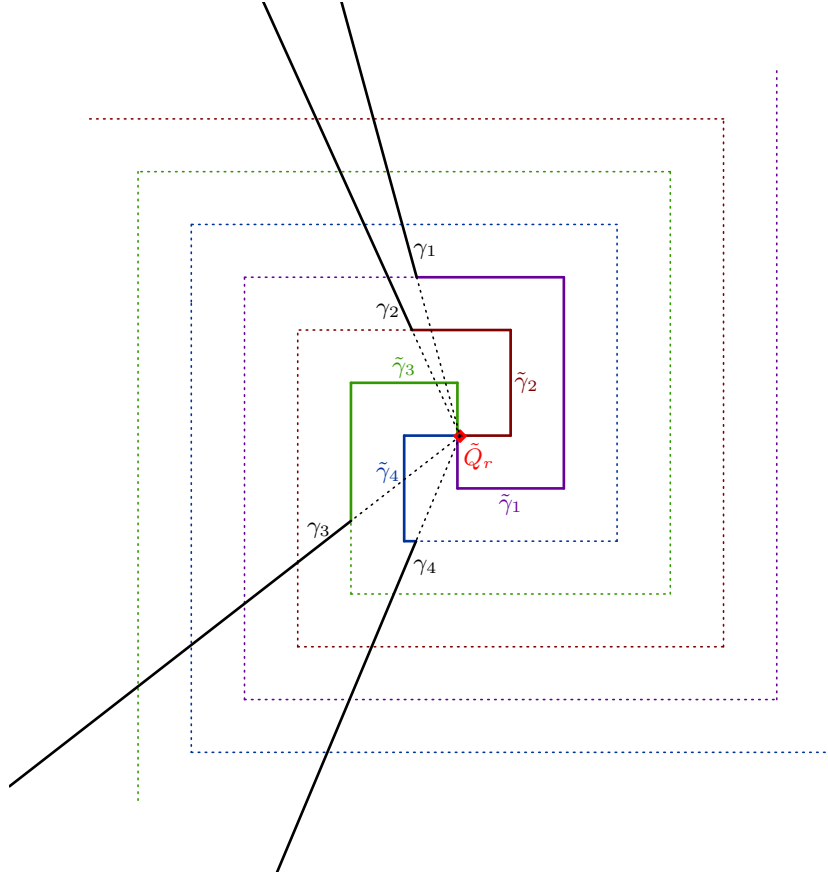
\begin{figure}
		\begin{tikzpicture}[line cap=round,line join=round,>=triangle 45,x=0.7cm,y=0.7cm]
		\clip(-8.415073427508649,-8.24827821812365) rectangle (8.487257935467955,8.200575961943699);
		\fill[line width=1.0pt,color=ffqqqq,fill=ffqqqq,fill opacity=1.0] (0.05,0.09) -- (-0.05,-0.01) -- (0.05,-0.11) -- (0.15,-0.01) -- cycle;
		\draw [line width=1.0pt,color=yqqqqq] (0.,0.)-- (1.,0.);
		\draw [line width=1.0pt,color=yqqqqq] (1.,0.)-- (1.,2.);
		\draw [line width=0.6pt,dotted,color=yqqqqq] (-3.,-4.)-- (5.,-4.);
		\draw [line width=0.6pt,dotted,color=yqqqqq] (5.,-4.)-- (5.,6.);
		\draw [line width=1.0pt,color=ttzzqq] (0.,0.)-- (0.,1.);
		\draw [line width=1.0pt,color=ttzzqq] (0.,1.)-- (-2.,1.);
		\draw [line width=0.6pt,dotted,color=ttzzqq] (-2.,-3.)-- (4.,-3.);
		\draw [line width=0.6pt,dotted,color=ttzzqq] (4.,-3.)-- (4.,5.);
		\draw [line width=1.0pt,color=qqttzz] (0.,0.)-- (-1.,0.);
		\draw [line width=1.0pt,color=qqttzz] (-1.,0.)-- (-1.,-2.);
		\draw [line width=0.6pt,dotted,color=qqttzz] (3.,-2.)-- (3.,4.);
		\draw [line width=0.6pt,dotted,color=qqttzz] (3.,4.)-- (-5.,4.);
		\draw [line width=1.0pt,color=wwqqzz] (0.,0.)-- (0.,-1.);
		\draw [line width=1.0pt,color=wwqqzz] (0.,-1.)-- (2.,-1.);
		\draw [line width=1.0pt,color=wwqqzz] (2.,-1.)-- (2.,3.);
		\draw [line width=0.6pt,dotted,color=wwqqzz] (-4.,-5.)-- (6.,-5.);
		\draw [line width=1.0pt,color=yqqqqq] (1.,2.)-- (-0.8530289115897572,2.);
		\draw [line width=0.6pt,dotted,color=yqqqqq] (-0.8530289115897572,2.)-- (-3.,2.);
		\draw [line width=0.6pt,dotted,color=yqqqqq] (-3.,-4.)-- (-3.,-2.4);
		\draw [line width=0.6pt,dotted,color=yqqqqq] (-3.,-2.4)-- (-3.,2.);
		\draw [line width=0.6pt,dotted,color=qqttzz] (-5.,4.)-- (-5.,2.5);
		\draw [line width=0.6pt,dotted,color=qqttzz] (-5.,2.5)-- (-5.,-6.);
		\draw [line width=0.6pt,dotted,color=ttzzqq] (-6.,5.)-- (-6.,-7.);
		\draw [line width=0.6pt,dotted,color=ttzzqq] (4.,5.)-- (-2.132572278974393,5.);
		\draw [line width=0.6pt,dotted,color=ttzzqq] (-2.132572278974393,5.)-- (-6.,5.);
		\draw [line width=0.6pt,dotted,color=qqttzz] (-5.,-6.)-- (7.,-6.);
		\draw [line width=0.6pt,dotted,color=wwqqzz] (6.,-5.)-- (6.,7.);
		\draw [line width=0.6pt,dotted,color=yqqqqq] (5.,6.)-- (-7.,6.);
		\draw [line width=1.0pt,color=ffqqqq] (0.05,0.09)-- (-0.05,-0.01);
		\draw [line width=1.0pt,color=ffqqqq] (-0.05,-0.01)-- (0.05,-0.11);
		\draw [line width=1.0pt,color=ffqqqq] (0.05,-0.11)-- (0.15,-0.01);
		\draw [line width=1.0pt,color=ffqqqq] (0.15,-0.01)-- (0.05,0.09);
		\draw [line width=0.6pt,dotted,color=wwqqzz] (-4.,-5.)-- (-4.,3.);
		\draw [line width=0.6pt,dotted,color=wwqqzz] (-4.,3.)-- (-0.7650880926885759,3.);
		\draw [line width=1.0pt,color=wwqqzz] (2.,3.)-- (-0.7650880926885759,3.);
		\draw [line width=1.0pt,color=ttzzqq] (-2.,1.)-- (-2.,-1.6163934426229511);
		\draw [line width=0.6pt,dotted,color=ttzzqq] (-2.,-3.)-- (-2.,-1.6163934426229511);
		\draw [line width=1.0pt,color=qqttzz] (-1.,-2.)-- (-0.7782615538627972,-2.);
		\draw [line width=0.6pt,dotted,color=qqttzz] (3.,-2.)-- (-0.7782615538627972,-2.);
		\draw [line width=0.6pt,dotted] (0.05,-0.01)-- (-2.,-1.6163934426229511);
		\draw [line width=1.0pt,domain=-8.415073427508649:-2.0] plot(\x,{(--0.049180327868853624-0.7836065573770488*\x)/-1.});
		\draw [line width=0.6pt,dotted] (0.05,-0.01)-- (-0.7782615538627972,-2.);
		\draw [line width=1.0pt,domain=-8.415073427508649:-0.7782615538627972] plot(\x,{(--0.05416211836111984-1.*\x)/-0.41621183611195856});
		\draw [line width=0.6pt,dotted] (0.05,-0.01)-- (-0.8530289115897572,2.);
		\draw [line width=1.0pt,domain=-8.415073427508649:-0.8530289115897572] plot(\x,{(-0.18202927539124936--4.*\x)/-1.797072460875139});
		\draw [line width=0.6pt,dotted] (-0.7650880926885759,3.)-- (0.05,-0.01);
		\draw [line width=1.0pt,domain=-8.415073427508649:-0.7650880926885759] plot(\x,{(-0.14187619841174293--3.*\x)/-0.8123801588258235});
		\begin{scriptsize}
		\draw[color=yqqqqq] (1.3,1) node {$\tilde{\gamma}_2$};
		\draw[color=ttzzqq] (-1,1.3) node {$\tilde{\gamma}_3$};
		\draw[color=qqttzz] (-1.3,-0.7) node {$\tilde{\gamma}_4$};
		\draw[color=wwqqzz] (1.0,-1.3) node {$\tilde{\gamma}_1$};
		\draw[color=black] (-2.6,-1.8) node {$\gamma_3$};
		\draw[color=black] (-0.6,-2.5) node {$\gamma_4$};
		\draw[color=black] (-1.3,2.35) node {$\gamma_2$};
		\draw[color=black] (-0.6,3.5) node {$\gamma_1$};
		\draw[color=red] (0.4,-0.4) node {$\tilde{Q}_r$};
		\end{scriptsize}
		\end{tikzpicture}
		\caption{The approach to replacing four segments with a common endpoint by four polylines ``spirals'' parallel to coordinate axes. Here the paths have already been renumbered so that the first quadrant is the successor of the second of the two quadrants containing two segments which we have also renumbered.}\label{Fig:Spiral}
	\end{figure}
	
	We call
	$$
	\gamma_{i,u,r}^* = (x_0,y_0) + u + r  \gamma_i^*  \quad \text{ and } \quad \Gamma(u,r) =\bigcup_{i=1}^4 \gamma_{i,u,r}^*
	$$
	for each $u \in \{(x,y) \in \er^2; |x|+ |y|\leq \delta r \} =:\tilde{Q}_r$ and $r>0$. Then $\Gamma(u,r) \subset Q((x_0,y_0), 8r)$ for each $u\in \tilde{Q}_r$. For every $u\in \tilde{Q}_r$ we decompose $u = s(\tfrac{1}{\sqrt{2}}, \tfrac{1}{\sqrt{2}}) + t(\tfrac{1}{\sqrt{2}}, -\tfrac{1}{\sqrt{2}})$ for  $s,t \in [-\tfrac{1}{\sqrt{2}}\delta r,\tfrac{1}{\sqrt{2}}\delta r]$. There exists a number $N$ such that for any fixed $s$ there is at most $N$ points of the set $\Gamma(0,r)$ on the corresponding line $s(\tfrac{1}{\sqrt{2}}, \tfrac{1}{\sqrt{2}}) + \er (\tfrac{1}{\sqrt{2}}, -\tfrac{1}{\sqrt{2}})$.

	As explained above, by $|D_{\tau}f_{\rceil \Gamma(u,r)}|$ we denote the measure on $\Gamma(u,r)$ given by the one-dimensional variation of $f$ on the segments of $\Gamma(u,r)$. Let us calculate
	\begin{equation}\label{Radiation}
	\begin{aligned}
	&\int_{\tilde{Q}_r} |D_{\tau}f_{\rceil \Gamma(u,r)}|(\Gamma(u,r)) \, d\L^2(u) \\
	&= \int_{\tfrac{-\delta r}{\sqrt{2}}}^{\tfrac{\delta r}{\sqrt{2}}}\int_{\tfrac{-\delta r}{\sqrt{2}}}^{\tfrac{\delta r}{\sqrt{2}}} |D_{\tau}f_{\rceil \Gamma(s(\tfrac{1}{\sqrt{2}}, \tfrac{1}{\sqrt{2}}) + t(\tfrac{1}{\sqrt{2}}, -\tfrac{1}{\sqrt{2}}),r)}|(\Gamma(s(\tfrac{1}{\sqrt{2}}, \tfrac{1}{\sqrt{2}}) + t(\tfrac{1}{\sqrt{2}}, -\tfrac{1}{\sqrt{2}}),r)) \, dt \, ds \\
	&\leq  \int_{\tfrac{-\delta r}{\sqrt{2}}}^{\tfrac{\delta r}{\sqrt{2}}} N\Big(|\langle Df,(0,1) \rangle|(Q((x_0,y_0),2r)) + |\langle Df,(1,0) \rangle|(Q((x_0,y_0),2r))\Big) \, ds\\
	&\leq CN\delta r|Df|\big(Q(x_0,y_0),8r\big).
	\end{aligned}
	\end{equation}
	Using \eqref{Twisted} we get that
	$$
	\oint_{\tilde{Q}_r} |Df_{|\Gamma(u,r)}|(\Gamma(u,r)) \, d\L^2(u)  \leq C\delta^{-1}r^{-1}|Df|\big(Q(x_0,y_0), 7r\big) \to 0
	$$
	as $r\to 0$. Using the standard average values argument (as in the proof of \eqref{Skiing}) we find an $r_0 >0$ such that for any $0<r<r_0$ there exists an $\tilde{A}_r \subset \tilde{Q}_r$ of positive $\L^2$ measure such that
	\begin{equation}\label{LivingBoy}
	|D_{\tau}f_{\rceil\Gamma(u,r)}|(\Gamma(u,r)) \leq\tfrac{1}{4} \sigma \text{ for any } u \in \tilde{A}_r.
	\end{equation}
	We may assume that all points of $\tilde{A}_r$ are Lebesgue points for $f$.
	
	Now we truncate the curves $\tilde{\gamma}_{i,u,r}^*$ for $u\in \tilde{A}_r$ so that they satisfy (1) and (2). The set $\Gamma(u, r)$ is the union of four polylines, $\gamma_{i,u, r}^*$, each of which is a `square spiral'-type curve anti-clockwise around $(x_0,y_0)+u$. Recall that for each $u \in A_r$ and for each $\gamma_i$ we have a uniquely determined quadrant $P_j$ such that $\gamma_i \setminus Q((x_0,y_0),\tfrac{1}{2}r) \subset P_j + u$. We call $P_{j+1}$ the successor of $P_j$ and $P_{j-1}$ its predecessor calculating the indexes $\mod 4$.
	
	Assume that we have $u\in A_r$ fixed. We start working with the quadrant whose predecessor contains the greatest number of segments $\gamma_i$. In the case where each quadrant contains exactly one segment we may start with any quadrant. In the case when there are two quadrants each containing two segments but neither is the successor of the other we start with the successor of either quadrant containing segments. If there are two successive quadrants both containing two segments each, we start with the quadrant succeeding the second of these two quadrants. Thus we have determined a starting quadrant (and without loss of generality) we assume it is  $P_1$. 
	
	We define $\tilde{\gamma}_1$ as the part of $\gamma_{1,u, r}^*$ which goes from $(\tilde{x}, \tilde{y}) = (x_0,y_0) + u$ to its first intersection with the first segment it meets, (if necessary re-number it to be) $\gamma_1$. Then we define each curve $\tilde{\gamma}_i$ similarly; $\tilde{\gamma}_i$ is the part of $\gamma_{i,u, r}^*$ which goes from $(\tilde{x}, \tilde{y}) = (x_0,y_0) + u$ to its first intersection with the segment (assumed after re-numbering to be) $\gamma_i$ i.e. the first segment which has not yet been taken by a previous curve. Since there is a finite number of scenarios it is not difficult to check that the algorithm ensures points (1) and (2) hold by checking on a case by case basis. An illustration of this process is in Figure~\ref{Fig:Spiral}.
	
	Point (4) is obvious. Point (5) is an obvious result of \eqref{LivingBoy}. Then, by Proposition~\ref{DebileDebileDebileDot} $\{\tilde{\gamma}_i\cup (x_i,y_i)+(0,\infty)v_i\}$ is admissible for $f$ for almost every $u\in \tilde{A}_r$, which is point (3) where $A_r = (x_0,y_0)+\tilde{A}_r$.	
\end{proof}

\begin{thm}\label{ItsAllLies}
	Let $f\in BV(Q(0,1), \er^2)$ and $f = \id$ on $\partial Q(0,1)$. It holds that $f$ is an $NCBV^+$ map if and only if it is an $NCBV$ map.
\end{thm}
\begin{proof}
	Thanks to Proposition~\ref{Obvious} it suffices to show that if $f$ satisfies $NCBV$ then $f$ satisfies $NCBV^+$. Therefore we take a non-straight grid $\Gamma$ good for $f$ and, for any any fixed $\sigma$, we construct $\tilde{\Gamma}$ a straight grid good for $f$ and an injective continuous map $g$ from $\hat{\Gamma} \subset \tilde{\Gamma}$ onto $\Gamma$ with the property that given any injective approximation $h$ of $f$ on $\tilde{\Gamma}$, with $\|h - f\|_{\infty, \tilde{\Gamma}}<\sigma/4$ it holds that $\|h\circ g^{-1}-f\|_{\infty, \Gamma} < \sigma$.

	\begin{figure}	
		\begin{tikzpicture}[line cap=round,line join=round,>=triangle 45,x=1.0cm,y=1.0cm]
		\clip(-8,-4.2) rectangle (8,4.3);
		\fill[line width=1.pt,color=ttttff,fill=ttttff,fill opacity=0.10000000149011612] (-2.82842712474619,2.82842712474619) -- (-2.82842712474619,-2.82842712474619) -- (2.82842712474619,-2.82842712474619) -- (2.82842712474619,2.82842712474619) -- cycle;
		\draw [line width=1.pt] (0.,0.) circle (4.cm);
		\draw [line width=0.5pt,dash pattern=on 1pt off 1pt,color=ffqqqq,domain=-4.438870273717386:5.161783365469267] plot(\x,{(-0.--2.2924995422892533*\x)/3.2778721525714762});
		\draw [line width=0.5pt,domain=-4.438870273717386:5.161783365469267] plot(\x,{(-0.--1.5060397079474188*\x)/3.7056503340285163});
		\draw [line width=1.pt] (-1.8631606742593614,-0.357901749239813)-- (2.01110114697385,-0.357901749239813);
		\draw [line width=1.pt] (2.01110114697385,-0.357901749239813)-- (2.0096752396517337,0.8167664075589935);
		\draw [line width=1.pt] (-1.8631606742593614,-0.357901749239813)-- (-1.8625801834827382,-0.7569844595972629);
		\draw [line width=1.pt,color=ttttff] (-2.82842712474619,2.82842712474619)-- (-2.82842712474619,-2.82842712474619);
		\draw [line width=1.pt,color=ttttff] (-2.82842712474619,-2.82842712474619)-- (2.82842712474619,-2.82842712474619);
		\draw [line width=1.pt,color=ttttff] (2.82842712474619,-2.82842712474619)-- (2.82842712474619,2.82842712474619);
		\draw [line width=1.pt,color=ttttff] (2.82842712474619,2.82842712474619)-- (-2.82842712474619,2.82842712474619);
		\draw [color=ffqqqq](3.7,2.9) node[anchor=north west] {$v(x_0,y_0)^{\bot}$};
		\draw (3.9,1.7) node[anchor=north west] {$[(x_1,y_1)(x_2,y_2)]$};
		\draw (0.5,-0.3) node[anchor=north west] {$P_{(x_0,y_0)}$};
		\draw (-1.3,0.7) node[anchor=north west] {$(x_0,y_0)$};
		\begin{scriptsize}
		\draw [fill=uuuuuu] (0.,0.) circle (2 pt);
		\draw [fill=ttttff] (2.0096752396517337,0.8167664075589935) circle (1pt);
		\draw [fill=ttttff] (-1.8625801834827382,-0.7569844595972629) circle (1pt);
		\end{scriptsize}
		\end{tikzpicture}
		\caption{The situation close to the points of $J_f \cap \Gamma$. The jump set `approaches' the approximate tangent space. Because $(x_0,y_0)$ is a `Lebesgue' jump point there are many horizontal segments with jumps very close to the jump at $(x_0,y_0)$. Such a horizontal segment can easily be connected to $[(x_1,y_1)(x_2,y_2)]$ by segments parallel to coordinate axes.}\label{Fig:JumpPoints}
	\end{figure}
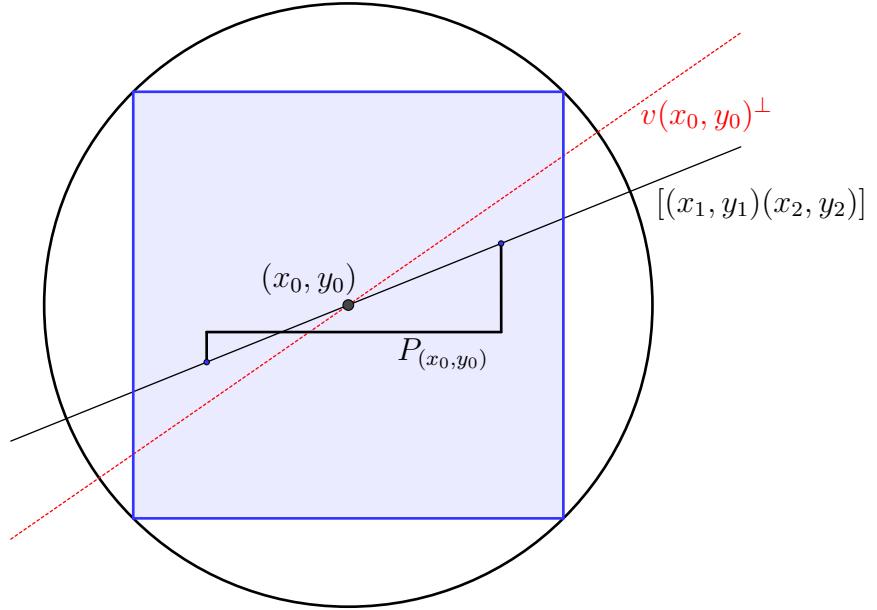

	We call $F_{\sigma}$ the set of $(x_0,y_0) \in \Gamma \cap J_f$ Lebesgue points for $\big(f^+(\cdot)-f^+(\cdot)\big)\otimes v(\cdot)$ in the sense of \eqref{StealAndBorrow} such that $\big|f^+(\cdot)-f^+(\cdot)\big|\geq \tfrac{1}{40}\sigma$. Then $F_{\sigma}$ is a finite set and for each $(x_0,y_0)\in F_{\sigma}$ there exists exactly one segment $[(x_1,y_1)(x_2,y_2)]$ of $\Gamma$ containing $(x_0,y_0)$. Without loss of generality we assume that $[(x_1,y_1)(x_2,y_2)] \cap F_{\sigma} = \{(x_0,y_0)\}$ and that $\big | \big \langle(1,0), v(x_0,y_0) \big \rangle \big |\geq \tfrac{1}{2}$.

	By the Federer-Vol'pert theorem, Theorem~\ref{FV}  (and Theorem~\ref{StrictForTanMeas}) we have the strict convergence of $r^{-1} |Df|\big(r [\cdot + (x_0,y_0)]\big)$ to $|f^+(x_0,y_0) - f^-(x_0,y_0)| \H^1_{\rceil v(x_0, y_0)^{\bot}}$ on $B(0,1)$ as $r\to 0$ for each point of $(x_0,y_0) \in \Gamma \cap F_{\sigma}$. Therefore we get
\begin{equation}\label{to be used for Close2}
r^{-1} |\langle Df, (1,0)\rangle|\big((x_0-r,x_0+r)\times(y_0-\tfrac{1}{2}r,y_0+\tfrac{1}{2}r)\big) \to |f^+(x_0,y_0) - f^-(x_0,y_0)|.
\end{equation}	
	Further we have that  $[(x_1,y_1)(x_2,y_2)]$ is not parallel to the tangent space $v(x_0, y_0)^{\bot}$. Therefore we observe the existence of a polyline $P_{(x_0,y_0)} \subset Q((x_0,y_0), r)$ with segments parallel to coordinate axes with one endpoint $(\tilde{x}_{(x_0,y_0)},\tilde{y}_{(x_0,y_0)}) \in [(x_1,y_1)(x_0,y_0)]\setminus \{(x_0,y_0)\}$ and the other on $(\hat{x}_{(x_0,y_0)},\hat{y}_{(x_0,y_0)}) \in[(x_0,y_0)(x_2,y_2)]\setminus \{(x_0,y_0)\}$ such that the straight grid generated by $P_{(x_0,y_0)}$ (call it $\tilde{\Gamma}_1$) is a good straight grid for $f$ and
	\begin{equation}\label{Close1}
		\Big| | D_{\tau}f_{\rceil \tilde{\Gamma}_1}|(P_{(x_0,y_0)}) - |f^+(x_0,y_0) - f^-(x_0,y_0)| \Big| \leq \big(\frac{\sigma}{100}\big)^2.
	\end{equation}
 	To ease our notation we denote $S_{(x_0,y_0), \sigma} = [(\tilde{x}_{(x_0,y_0)},\tilde{y}_{(x_0,y_0)}) (\hat{x}_{(x_0,y_0)},\hat{y}_{(x_0,y_0)})]$. Further, thanks to \eqref{to be used for Close2}, we assume that $r>0$ is chosen so small that
	\begin{equation}\label{Close2}
		| D_{\tau}f_{\rceil \Gamma}|\big( S_{(x_0,y_0), \sigma}\big) \leq |f^+(x_0,y_0) - f^-(x_0,y_0)| + \big(\frac{\sigma}{100}\big)^2.
	\end{equation}

	For a depiction of the process see Figure~\ref{Fig:JumpPoints}. Since we are in position to choose the $x$-coordinate of vertical segments and the $y$-coordinate for horizontal segments for the polylines $P_{(x_0,y_0)}$ from sets of positive measure. Therefore we may assume that the horizontal segments do not have $y$-coordinate in the set $N_1$ from Corollary~\ref{Repeat} and the vertical segments do not have $x$-coordinate in the set $N_2$ from Corollary~\ref{Repeat} and that the straight grid generated by $\bigcup_{(x_0,y_0)\in \Gamma\cap F_{\sigma}} P_{(x_0,y_0)}$ is good for $f$. Further we may assume that we choose $r$ so small that the squares $Q\big((x_0,y_0),r\big)$ for $(x_0, y_0) \in F_{\sigma}\cap \Gamma$ are pairwise disjoint and do not contain any of the intersection points $X_{i,j}$ of $\Gamma$. 
	
	Let $r>0$ be chosen so small that $Q(X_{i,j}, r) \cap \Gamma$ is exactly the union of four segments, $Q(X_{i,j}, r)$ are pairwise disjoint and are also disjoint with each $Q\big((x_0,y_0),r\big)$ for each $(x_0, y_0) \in F_{\sigma}\cap \Gamma$. For each point $X_{i,j}$ intersection point of $\gamma_i([0,1]) \neq \gamma_j([0,1]) \subset \Gamma$ we find polylines (contained in $Q(X_{i,j}, r)$) from Proposition~\ref{Crossroads}, where the four segments from the claim of Proposition~\ref{Crossroads} (there called $\gamma_1, \dots, \gamma_4$) refer to the four segments of $\Gamma$, which intersect at $X_{i,j}$.  We will refer to these polylines as $\tilde{\gamma}^{X_{i,j}}_k$, $k=1,2,3,4$. By the fact that $A_r$ has positive measure we can garantee that the straight grid generated by $\{\tilde{\gamma}^{X_{i,j}}_k: i,j,k \}$ and by the polylines $\{P_{(x_0,y_0)}: (x_0,y_0)\in F_{\sigma} \}$ is a good straight grid for $f$. We call this grid $\tilde{\Gamma}_2$.
	
	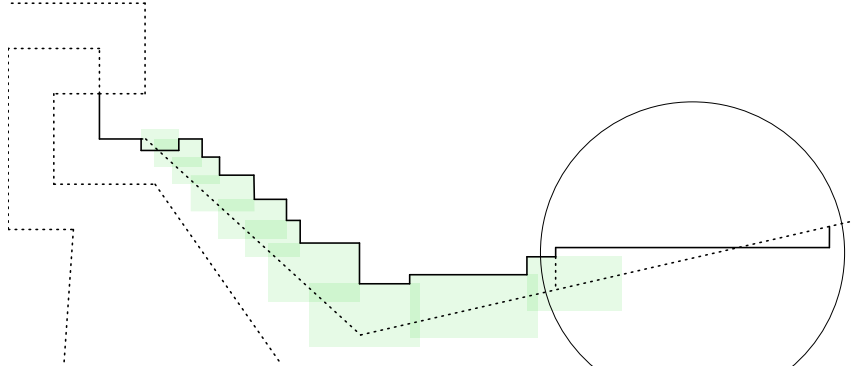
\begin{figure}
		\begin{tikzpicture}[line cap=round,line join=round,>=triangle 45,x=0.6 cm,y=0.6 cm]
		\clip(-2,-6) rectangle (16.5,4);
		\fill[line width=1.pt,color=qqccqq,fill=qqccqq,fill opacity=0.10000000149011612] (0.9134081329050975,-0.7768586616667901) -- (0.9134081329050975,-1.2538268537651163) -- (1.7395412752136707,-1.2538268537651163) -- (1.7395412752136707,-0.7768586616667901) -- cycle;
		\fill[line width=1.pt,color=qqccqq,fill=qqccqq,fill opacity=0.10000000149011612] (1.2,-1.) -- (1.2,-1.6069579560255214) -- (2.251282017894359,-1.6069579560255214) -- (2.251282017894359,-1.) -- cycle;
		\fill[line width=1.pt,color=qqccqq,fill=qqccqq,fill opacity=0.10000000149011612] (1.6,-1.4) -- (1.6,-1.9961001309902808) -- (2.63247571327363,-1.9961001309902806) -- (2.63247571327363,-1.4) -- cycle;
		\fill[line width=1.pt,color=qqccqq,fill=qqccqq,fill opacity=0.10000000149011612] (2.,-2.6) -- (2.,-1.8) -- (3.385640646055102,-1.8) -- (3.4,-2.6) -- cycle;
		\fill[line width=1.pt,color=qqccqq,fill=qqccqq,fill opacity=0.10000000149011612] (2.6,-3.2) -- (2.6,-2.3340145526058156) -- (4.099930793501993,-2.3340145526058156) -- (4.099930793501993,-3.2) -- cycle;
		\fill[line width=1.pt,color=qqccqq,fill=qqccqq,fill opacity=0.10000000149011612] (3.2,-3.6) -- (3.2,-2.8) -- (4.4,-2.8) -- (4.4,-3.6) -- cycle;
		\fill[line width=1.pt,color=qqccqq,fill=qqccqq,fill opacity=0.10000000149011612] (3.7,-4.6) -- (3.7,-3.3) -- (5.7,-3.3) -- (5.7,-4.6) -- cycle;
		\fill[line width=1.pt,color=qqccqq,fill=qqccqq,fill opacity=0.10000000149011612] (4.6,-5.6) -- (4.6,-4.2) -- (7.024871130596427,-4.2) -- (7.024871130596427,-5.6) -- cycle;
		\fill[line width=1.pt,color=qqccqq,fill=qqccqq,fill opacity=0.10000000149011612] (6.8,-4.) -- (6.8,-5.4) -- (9.6,-5.4) -- (9.6,-4.) -- cycle;
		\fill[line width=1.pt,color=qqccqq,fill=qqccqq,fill opacity=0.10000000149011612] (9.369923719539479,-3.6036865257614235) -- (9.369923719539479,-4.800378011187435) -- (11.442654173282403,-4.800378011187435) -- (11.442654173282403,-3.6036865257614235) -- cycle;
		\draw [line width=0.6pt,dotted] (0.,0.)-- (0.,1.);
		\draw [line width=0.6pt,dotted] (0.,1.)-- (-2.,1.);
		\draw [line width=0.6pt,dotted] (-2.,1.)-- (-2.,-3.);
		\draw [line width=0.6pt,dotted] (0.,0.)-- (-1.,0.);
		\draw [line width=0.6pt,dotted] (-1.,0.)-- (-1.,-2.);
		\draw [line width=0.6pt] (0.,0.)-- (0.,-1.);
		\draw [line width=0.6pt,dotted] (0.,0.)-- (1.,0.);
		\draw [line width=0.6pt,dotted] (1.,0.)-- (1.,2.);
		\draw [line width=0.6pt,dotted] (1.,2.)-- (-3.,2.);
		\draw [line width=0.6pt,dotted] (-1.,-2.)-- (1.219116183301343,-2.);
		\draw [line width=0.6pt,dotted] (-2.,-3.)-- (-0.5707741478766021,-3.);
		\draw [line width=0.6pt,dotted] (-3.,2.)-- (-3.,-0.8228886400887924);
		\draw [line width=0.6pt,dotted] (-0.5707741478766021,-3.)-- (-0.853699921814361,-7.062814113746213);
		\draw [line width=0.6pt,dotted] (1.219116183301343,-2.)-- (4.,-6.);
		\draw [line width=0.6pt,dotted] (1.0140092826374982,-1.)-- (5.724995960862981,-5.331578355146914);
		\draw [line width=0.6pt,dotted] (-3.,-0.8228886400887924)-- (-8.095258480143789,-2.8798052362057787);
		\draw [line width=0.6pt,dotted] (5.724995960862981,-5.331578355146914)-- (20.,-2.);
		\draw [line width=0.6pt] (10.,-3.4)-- (16.,-3.4);
		\draw [line width=0.6pt] (16.,-3.4)-- (16.00188432764547,-2.933102057195302);
		\draw [line width=0.2pt] (13.,-3.5131326540808905) circle (2cm);
		\draw [line width=0.6pt] (0.9134081329050977,-1.)-- (0.9134081329050975,-1.2538268537651163);
		\draw [line width=0.6pt] (0.9134081329050975,-1.2538268537651163)-- (1.7395412752136707,-1.2538268537651163);
		\draw [line width=0.6pt] (1.7395412752136707,-1.2538268537651163)-- (1.7395412752136712,-1.);
		\draw [line width=0.6pt] (1.7395412752136712,-1.)-- (2.251282017894359,-1.);
		\draw [line width=0.6pt] (2.251282017894359,-1.)-- (2.251282017894359,-1.4);
		\draw [line width=0.6pt] (2.251282017894359,-1.4)-- (2.63247571327363,-1.4);
		\draw [line width=0.6pt] (2.63247571327363,-1.4)-- (2.6324757132736303,-1.8);
		\draw [line width=0.6pt] (2.6324757132736303,-1.8)-- (3.385640646055102,-1.8);
		\draw [line width=0.6pt] (3.385640646055102,-1.8)-- (3.3952257760208435,-2.334014552605816);
		\draw [line width=0.6pt] (3.3952257760208435,-2.334014552605816)-- (4.099930793501993,-2.3340145526058156);
		\draw [line width=0.6pt] (4.099930793501993,-2.3340145526058156)-- (4.099930793501993,-2.8);
		\draw [line width=0.6pt] (4.099930793501993,-2.8)-- (4.4,-2.8);
		\draw [line width=0.6pt] (4.4,-2.8)-- (4.4,-3.3);
		\draw [line width=0.6pt] (4.4,-3.3)-- (5.7,-3.3);
		\draw [line width=0.6pt] (5.7,-3.3)-- (5.7,-4.2);
		\draw [line width=0.6pt] (5.7,-4.2)-- (6.8,-4.2);
		\draw [line width=0.6pt] (6.8,-4.2)-- (6.8,-4.);
		\draw [line width=0.6pt] (6.8,-4.)-- (9.369923719539475,-4.);
		\draw [line width=0.6pt] (9.369923719539475,-4.)-- (9.369923719539479,-3.6036865257614235);
		\draw [line width=0.6pt] (9.369923719539479,-3.6036865257614235)-- (9.999833074567189,-3.6036865257614235);
		\draw [line width=0.6pt] (9.999833074567189,-3.6036865257614235)-- (10.,-3.4);
		\draw [line width=0.6pt,dotted] (9.999833074567189,-3.6036865257614235)-- (9.999234539937886,-4.33403322691165);
		\draw [line width=0.6pt] (0.,-1.)-- (0.9134081329050977,-1.);
		\draw [line width=0.6pt,dotted] (0.9134081329050977,-1.)-- (1.0140092826374982,-1.);
		\end{tikzpicture}
		\caption{The above figure depicts how we choose a polyline parallel to coordinate axes along the boundary of chosen rectangles (shown in green) between (in this case) an intersection point and a jump point}\label{Fig:Rectangles}
	\end{figure}

	Recall the notation; for every point $(x_0,y_0)\in F_{\sigma}$ we have a segment (above denoted by $S_{x_0,y_0\sigma}$). Using an obvious adaptation of the notation of Proposition~\ref{Crossroads}, with $[(x_k(X_{i,j}),y_k(X_{i,j})) X_{i,j}]$ we denote the  segment of endpoints $(x_k(X_{i,j}),y_k(X_{i,j}))$ and $X_{i,j}$ and we define the set
	$$
		\kappa = \overline{ \Gamma \setminus \Big(\bigcup_{(x_0,y_0)\in F_{\sigma}}P_{(x_0,y_0)} \cup \bigcup_{X_{i,j}}\bigcup_{k=1}^4 \big[(x_k(X_{i,j}),y_k(X_{i,j})) X_{i,j}\big]\Big)}.
	$$
	Obviously $\kappa$ is a compact set. All points of $\kappa$ satisfy \eqref{Lock}. Thus by Lemma~\ref{NoName} we have a fine covering of $\kappa$ with rectangles each of which satisfies \eqref{Nuts} and each rectangle is associated to a particular point on $\kappa$. We choose a finite covering of $\kappa$ of the rectangles from Lemma~\ref{NoName} with the rectangles chosen so small that whenever a pair of rectangles intersect then either
	\begin{itemize}
		\item $R_1$, $R_2$ are associated with points on the same segment of $\Gamma$ and the associated points are not disconnected by some $(x_0,y_0)\in F_{\sigma}$
		\item $R_1$, $R_2$ are associated with points on neighbouring segments of $\Gamma$ whose common endpoint is some $\gamma_{i}(s_{i,k})$ (where $s_i,k$ has been defined in Definition \ref{Admissible}), not a point $X_{i,j}$.
	\end{itemize}
	Each rectangle can be made slightly bigger allowing us to choose horizontal lines with $y$-coordinate in $[-1,1]\setminus N_1$ (similarly for $x$-coordinate) since the set $A_n$ has positive measure. Therefore we may assume that the grid generated by the sides of all the rectangles and $\tilde{\Gamma}_2$ is a good straight grid for $f$. We may also assume that the covering we have chosen is minimal in the sense that it does not contain a strict sub-covering of $\kappa$. This grid is exactly $\tilde{\Gamma}$ mentioned at the start of the proof.
	
	Now we want to choose the set $\hat{\Gamma} \subset \tilde{\Gamma}$. Each of the polylines $P_{(x_0,y_0)}$ intersect exactly two of the rectangles $R_1, R_2$ chosen in the previous paragraph using Lemma~\ref{NoName} (each containing exactly one endpoint of $P_{(x_0,y_0)}$). The part of $P_{(x_0,y_0)}$ included in $\hat{\Gamma}$ is $\overline{P_{(x_0,y_0)}\setminus(R_1\cup R_2)}$.
	
	For each $X_{i,j}$, we defined four polylines $\tilde{\gamma}_1^{X_{i,j}}, \dots, \tilde{\gamma}_4^{X_{i,j}}$, each of which intersects exactly one (and distinct) rectangle $R_1, \dots, R_4$ chosen using Lemma~\ref{NoName}. We include each $\overline{\tilde{\gamma}_k^{X_{i,j}} \setminus R_k}$ in $\hat{\Gamma}$ for $k=1,2,3,4$ and do this at each $X_{i,j}$.
	
	Let $R$ be a rectangle chosen using Lemma~\ref{NoName} such that we already have exactly one point in $\hat{\Gamma}\cap \partial R$. Then $R$ has a neighbouring $R'$ (also chosen using Lemma~\ref{NoName}) and there exists a path on $\partial R$ from the previously chosen point in $\hat{\Gamma}\cap \partial R$ to a point in $\partial R'$. We include this path in $\hat{\Gamma}$.
	
	If we have a rectangle $R$ chosen using Lemma~\ref{NoName} such that $\hat{\Gamma}\cap \partial R$ contains exactly two points. Then there exists a path on $\partial R$ between the two points of $\hat{\Gamma}\cap \partial R$. We include this path in $\hat{\Gamma}$. After dealing with the finite number of rectangles we have defined $\hat{\Gamma}$.

	Now we describe how to define $g$ on $\hat{\Gamma}$. We define $g$ on $\hat{\Gamma} \cap P_{(x_0,y_0)}$ as the constant speed map from $\hat{\Gamma} \cap P_{(x_0,y_0)}$ onto the segment $S_{x_0,y_0, \sigma}$. Then $g$ is injective and continuous on $\hat{\Gamma} \cap P_{(x_0,y_0)}$. On $\hat{\Gamma} \cap \tilde{\gamma}_k^{X_{i,j}}$ we define $g$ as the constant speed map onto the segment $\big[\big(x_k(X_{i,j}),y_k(X_{i,j})\big) X_{i,j}\big]$. For any pair of rectangles $R_1,R_2$ chosen using Lemma~\ref{NoName} we choose a point $Z_{R_1,R_2} \in \Gamma \cap R_1 \cap R_2$. Then for each rectangle $R$ chosen using Lemma~\ref{NoName} we define $g$ as the constant speed map from $\hat{\Gamma} \cap \partial R$ onto the part of $\Gamma$ between the two points $Z_{R,R'}$ and $Z_{R, R''}$, where $R'$ and $R''$ are the two neighbours of $R$. Alternatively if $R$ neighbours $R'$ and $P_{(x_0,y_0)}$ we define $g$ as constant speed onto the part of $\Gamma$ between $Z_{R, R'}$ and the corresponding endpoint of the polyline $P_{(x_0,y_0)}$. Finally, if $R$ neighbours $R'$ and $\tilde{\gamma}_k^{X_{i,j}}$ we define $g$ as constant speed onto the part of $\Gamma$ between $Z_{R, R'}$ and the corresponding endpoint of the polyline $\tilde{\gamma}_k^{X_{i,j}}$.
	
	We separate $\hat{\Gamma}$ into three pieces; parts of $\partial R$, the polylines $\tilde{\gamma}_k^{X_{i,j}}\cap \hat{\Gamma}$ and the polylines $P_{(x_0,y_0)} \cap \hat{\Gamma}$. In each case we refer to the corresponding polyline as $E$. In the first two cases we have that $|Df_{\tilde{\Gamma}}|(E) \leq \tfrac{1}{4}\sigma$ by Lemma~\ref{NoName} and Proposition~\ref{Crossroads}. Then an injective $H_{\sigma}$ with $\|H_{\sigma} - f\|_{\infty, \tilde{\Gamma}}<\sigma/4$ has oscillation bounded by $\tfrac{1}{2}\sigma$ on $E$. The same oscillation estimates holds on the segment ${\Gamma} \cap R$ and $f$  is continuous at the intersection points of $\Gamma \cap \partial R$. Therefore it holds that $\|H_{\sigma}\circ g^{-1}-f\|_{\infty, E} < \sigma$. The only case that needs closer consideration is the case of large jumps, i.e. $E = P_{(x_0,y_0)} \cap \hat{\Gamma}$. But in this case we have chosen the endpoints of $P_{(x_0,y_0)} \cap \hat{\Gamma}$ so that their image is very close to the corresponding one sided limits at the jump point $(x_0,y_0)$ (by \eqref{Close1} and \eqref{Close2} the error is bounded by $4(\tfrac{1}{100}\sigma)^2$). Then by Lemma~\ref{Repetition} we have that $\|H_{\sigma}\circ g^{-1}-f\|_{\infty, E} < \tfrac{1}{4}\sigma$ and so we get our required result.
	
\end{proof}

\section{Construction of a sequence converging area-strictly to $f$ if $f$ satisfies the $NCBV^+$ condition}

To approximate $f$ area-strictly we need to isolate the majority of the singular part of $Df$ from the majority of the absolutely continuous part of $Df$. That is the goal of the following lemma.

\begin{lemma}\label{Isolationism}
	Let $f\in \BV(Q(0,1))$ and let $\epsilon >0$ then there exists a finite number of squares $\{Q_i\}_{i=1}^{N_0}$ (whose union we denote as $\tilde{F}_{\epsilon} = \bigcup_{i=1}^{N_0} Q_i$) such that 
	$$
		|D^af|(\tilde{F}_{\epsilon}) \leq \epsilon |D^af|(Q(0,1))
	$$
	and
	$$
		|D^sf|\big(Q(0,1)\setminus \tilde{F}_{\epsilon} \big)| \leq \epsilon |D^sf|\big(Q(0,1)\big)
	$$
	Further, if $|D^af|\big(Q(0,1)\big) > 0$, then
	$$
		|D^sf|\big(Q(0,1)\setminus \tilde{F}_{\epsilon} \big)| \leq \epsilon^2 |D^af|\big(Q(0,1)\big).
	$$
\end{lemma}
\begin{proof}
	We assume that $|D^af|\big(Q(0,1)\big) > 0$ because the opposite case easily follows from the proof for this case. We may assume that $|D^sf|(Q(0,1))>0$ because otherwise the claim is nothing but the absolute continuity of the integral. Let $\delta > 0$ be a number chosen small enough that
	$$
		|D^af|(A) \leq \epsilon \min\big\{1, |D^af|\big(Q(0,1)\big), |D^sf|\big(Q(0,1)\big)\big\}
	$$
	for any $A\subset Q(0,1)$ such that $\L^2(A)< \delta$. We call the set
	$$
		S = \Big\{x\in Q(0,1);  \ \lim_{r\to 0}\frac{|Df|(Q(x,r))}{r^2} = \infty\Big\}
	$$
	and similarly we call
	$$
		A_M^R = \Big\{x\in Q(0,1)\colon  \ \frac{|Df|(Q(x,r))}{r^2} > 4M \textrm{ for every }   0 < r < R \Big\}.
	$$
	Recall that by \cite[Proposition 3.92]{AFP}
	\begin{equation}\label{SingOnS}
		|D^sf|(Q(0,1)\setminus S) = 0.
	\end{equation}
	By the Vitali-Besicovitch theorem \cite[Theorem 2.19]{AFP} we have a countable collection of pairwise disjoint squares $Q(x_i,r_i)$ with $x_i \in A_M^R$ and $0<r_i <R$ covering $A_M^R$ (up to a set of $\L^n+|Df|$-measure 0). Using $\frac{|Df|(Q(x_i,r_i))}{r_i^2} > 4M$ we get
	$$
		\L^2(A_M^R) \leq \sum_{i=1}^\infty 4r^2_i < \frac{|Df|\big(Q(0,1)\big)}{M}.
	$$
	and by choosing $M = \delta^{-1} 2|Df|(Q(0,1))$ we have 
	\begin{equation}\label{TestyEsti}
		\L^2(A_M^R) < \frac{\delta}{2}.
	\end{equation}
	Obviously for all $M>0$ and all $R_0 >0$ we have
	$$
		S \subset \bigcup_{0<R<R_0} A_M^R,
	$$
	especially the inclusion $A_M^{R_1} \subset  A_M^{R_2}$ holds for any $R_1 > R_2> 0$.
	This shows that $\L^2 (S) = 0$ but also, using \eqref{SingOnS}, that $|D^sf|(S \setminus A_M^R) \to 0$ as $R\to 0^+$ (the limit makes sense thanks to the previous inclusion). Find then an $R_0$ such that
	\begin{equation}\label{Optics}
\begin{split}
	&|D^sf|(S \setminus A_M^{R_0}) < \frac{\epsilon}{2} |D^s f|(Q(0,1)), \qquad |D^sf|(S \setminus A_M^{R_0}) < \frac{\epsilon^2}{2} |D^a f|(Q(0,1)), \\
	& \L^2\big(A_M^{R_0} + B_{R_0}(0)\big) \leq \delta,
	\end{split}
	\end{equation}
	where with $A_M^{R_0} + B_{R_0}(0)$ we mean the set $\{ x \in \R^2 :\, \dist(x,A_M^{R_0}) < R_0 \}$.
	
	The Vitali-Besicovitch theorem gives a countable number of pairwise disjoint squares $Q_i = Q(x_i, R_i)$, $x_i \in A_M^{R_0}$ with $R_i \leq R_0$. Since $\lim_{N\to \infty}|D^sf|(A_M^{R_0} \setminus \bigcup_{i=1}^{N} Q_i) = 0$ we can find a $N_0$ such that
	$$
		|D^sf|\Big(A_M^{R_0} \setminus \bigcup_{i=1}^{N_0} Q_i\Big) < \frac{\epsilon}{2} |D^s f|\big(Q(0,1)\big)
	$$
	and similarly
	$$
		|D^sf|\Big(A_M^{R_0} \setminus \bigcup_{i=1}^{N_0} Q_i\Big) < \frac{\epsilon^2}{2} |D^a f|\big(Q(0,1)\big).
	$$
	Combining the last two estimates with \eqref{SingOnS} and \eqref{Optics} we have
	$$
		\begin{aligned}
			|D^sf|\Big(Q(0,1) \setminus \bigcup_{i=1}^{N_0} Q_i\Big) &< \epsilon |D^s f|(Q(0,1)) \text{ and }\\
			|D^sf|\Big(Q(0,1) \setminus \bigcup_{i=1}^{N_0} Q_i\Big) &< \epsilon^2 |D^a f|(Q(0,1)) .
		\end{aligned}
	$$
	On the other hand, by the third of \eqref{Optics}, \eqref{TestyEsti} and the choice of $\delta$, we have
	$$
		|D^af|\Big(\bigcup_{i=1}^{N_0} Q_i \Big) \leq \epsilon \min\{1, |D^af|(Q(0,1)), |D^sf|(Q(0,1))\}.
	$$
\end{proof}

The following theorem is used to divide $Q(0,1)$ up into small squares (side length is $2^{1-K}$) which have different properties. In the squares of type $E_{\epsilon, K}$ the majority of the behaviour comes from the singular part of the derivative and the direction map of the polar decomposition of the derivative satisfies a Lebesgue-point-type estimate based on the choice of $\epsilon$. Further we categorise the other squares, which are used to approximate $D^af$ in $L^1$. The categories are $G_{\epsilon, \alpha, K}$ (where $f$ is very close to a nice affine map), $T_{\epsilon, \alpha, K}$ (where $f$ is very close to a non-constant affine map with zero Jacobian) and $W_{\epsilon, \alpha, K}$ (where $f$ is nearly constant or does not behave `$\epsilon$-similarly' to any affine map). In the following theorem we refer to the set $\tilde{F}_{\epsilon}$ defined in Lemma~\ref{Isolationism}. 
\begin{thm}\label{CEZ}
	Let $f\in BV(Q(0,1))$ be an NCBV map and let $\epsilon > 0$. There exists an $0<\alpha_0 < \epsilon$ such that for any $0<\alpha\leq \alpha_0$ the following holds. There exists a $K = K(\epsilon, \alpha)\in \en$ such that the division of $Q(0,1)$ into $K$-dyadic squares $\{Q_i\}_{i = 1}^{2^{2K}}$ has the following properties
	\begin{enumerate}
		\item calling $F_{\epsilon, K}$ the set of $K$-dyadic squares which themselves intersect $\tilde{F}_{\epsilon}$ or have a neighbouring $K$-dyadic square that intersects $\tilde{F}_{\epsilon}$ and calling $\tilde{F}_{\epsilon, K}= \bigcup_{Q_i \in F_{\epsilon,K} } Q_i$ it holds that
		$$
			|D^af|(\tilde{F}_{\epsilon, K}) \leq 2\epsilon|D^a f|\big(Q(0,1)\big),
		$$
		\item it holds that
		$$
			|D^s f|\big(Q(0,1) \setminus\tilde{F}_{\epsilon, K} \big) \leq \epsilon |D^sf|\big(Q(0,1)\big)
		$$
		and if $|D^a f|\big(Q(0,1)\big) > 0$ then
		$$
			|D^s f|\big(Q(0,1) \setminus \tilde{F}_{\epsilon, K} \big) \leq \epsilon^2 |D^af|\big(Q(0,1)\big),
		$$
		\item there exists a subselection $E_{\epsilon, K}$ of squares of $F_{\epsilon, K}$ (whose union we denote as $\tilde{E}_{\epsilon, K}$) such that
		$$
			|D^s f|(Q(0,1) \setminus \tilde{E}_{\epsilon, K}) \leq 2\epsilon |D^s f|\big(Q(0,1)\big)
		$$
		and for any $Q_i \in E_{\epsilon, K}$ there exists a $w_i \in Q_i\cap S$ such that
		\begin{equation}\label{Piano}
			\int_{S\cap 2Q_i}|g(z)-g(w_i)|d|D^sf|(z) \leq \epsilon |D^sf|(4Q_i \cap S ),
		\end{equation}
		where $g\in L^{1}(|D^sf|, \er^2)$, $g|D^s f| = D^s f$ is the polar decomposition of $D^sf$ and $g(w_i) = u_i\otimes v_i$ for an appropriate $|u_i| = |v_i|=1$.
		\item Further, all squares $Q_i$ with $Q_i \cap \tilde{F}_{\epsilon} = \emptyset$, are separated into three disjoint categories $G_{\epsilon, \alpha, K}$, $T_{\epsilon, \alpha, K}$ and $W_{\epsilon, \alpha, K}$ (their unions denoted by $\tilde{G}_{\epsilon, \alpha, K}$, $\tilde{T}_{\epsilon, \alpha, K}$ and $\tilde{W}_{\epsilon, \alpha, K}$) such that
		\begin{equation}\label{Judy}
			|D^a f|\big(\tilde{W}_{\epsilon, \alpha, K} \big) \leq 8\epsilon \big[|D^a f|\big(Q(0,1)\big) + 1\big]
		\end{equation}
		and, for any $Q_i = Q(c_i, 2^{-K}) \in G_{\epsilon, \alpha, K}\cup T_{\epsilon, \alpha, K}$ there exists an $w_i \in Q(c_i, 2^{-K-2})$ and a set $Z_{i, \alpha}\subset Q(c_i, 2^{1-K})$ with $\L^2(Q(c_i, 2^{1-K}) \setminus Z_{i,\alpha})\leq 2^{-2K-9}$ such that
		\begin{equation}\label{Scorpions}
				\begin{aligned}
					\int_{Q(c_i, 2^{2-K})} |\nabla f(y) - \nabla f(w_i)| d\L^2(y) &\leq \epsilon \alpha^2 2^{-2K},\\
					\|f(\cdot) - f(w_i) - \nabla f(w_i)(\cdot - w_i)\|_{L^\infty(Z_{i,\alpha})} &< \alpha^4 2^{-K},\\
					\alpha_0\leq |\nabla f(w_i)|\leq \alpha_0^{-1}&,\\
			\end{aligned}
		\end{equation}
		and
		$$
			|D^sf|(Q_i) \leq \epsilon |D^a f|(Q_i).
		$$
		In the case that $Q_i \in G_{\epsilon, \alpha, K}$ it holds that $\alpha_0 < \det \nabla f(w_i)$ and in the case $Q_i \in T_{\epsilon, \alpha, K}$ it holds that $\det \nabla f(w_i) = 0$.
	\end{enumerate}
\end{thm}
\begin{proof}
	\step{1}{Prove (1) and (2) by applying Lemma~\ref{Isolationism}}{CEZS1}
	
	The set $\tilde{F}_{\epsilon}$ is the union of a finite number of disjoint squares. Then as $K\to \infty$ we clearly have $\L^2(\tilde{F}_{\epsilon, K} \setminus \tilde{F}_{\epsilon}) \to 0$. Therefore we find a $K_0$ such that for any $K\geq K_0$ we have $\L^2(\tilde{F}_{\epsilon, K} \setminus \tilde{F}_{\epsilon})<\delta$, where $\delta$ is so small that $|D^af|(A) \leq \epsilon |D^af|(Q(0,1))$ as soon as $\L^2(A)< \delta$. Since $\tilde{F}_{\epsilon}$ was chosen so that $|D^af|(\tilde{F}_{\epsilon}) \leq \epsilon |D^af|(Q(0,1))$, we satisfy point (1) of the claim. The set $\tilde{F}_{\epsilon, K} \supset \tilde{F}_{\epsilon}$ and so point (2) of our claim is immediate from Lemma~\ref{Isolationism}.
	
	\step{2}{Find a $K_1$ that allows us to prove (3)}{CEZS2}
	
	Let us call $g$ the function of the so-called polar decomposition of $D^sf$. Then $|g|=1$ $|D^sf|$-almost everywhere and $D^sf = g|D^sf|$. The function $g\in L^1(Q(0,1), |D^sf|, \er^2)$ and so $|D^sf|$ almost every point of $S$ is a Lebesgue point of $g$ with respect to $|D^sf|$. Recall that the singular part of the derivative of $f$ is supported on $S$, i.e. $D^sf = D^sf_{\rceil S}$ (see \cite[Proposition 3.92]{AFP}). As a result of this (see \cite[Theorem 1.33]{EG}) for $|D^sf|$-almost every $w \in S$ it holds that
	$$
		\frac{1}{|D^sf|(Q(w,r) \cap S )}{\int_{Q(w,r) \cap S}|g(z)-g(w)|d|D^sf|(z)} \xrightarrow{r\to 0^+} 0.
	$$
	Thus, for any given $\epsilon$ the $|D^sf|$ measure of the set of points $w \in S$ such that
	$$
		\int_{Q(w,r) \cap S}|g(z)-g(w)|d|D^sf|(z)> \epsilon |D^sf|(Q(w,r) \cap S )
	$$
	for some $0<r<2^{-K}$ tends to zero as $K \to \infty$. Thus we find a $K_1$ so that the $|D^sf|$ measure of this set is bounded by $\epsilon |D^sf|(Q(0,1))$. Call $\tilde{X}_{\epsilon, K}$ the set of $w\in S$ such that
	$$
		\int_{Q(w,r) \cap S}|g(z)-g(w)|d|D^sf|(z) \leq \epsilon |D^sf|(Q(w,r) \cap S ) \text{ for all } 0< r< 2^{-K_1}.
	$$
	For any $K\geq K_1 + 1$ we have for any $Q_i = Q(c_i,2^{-K}) \in F_{\epsilon, K}$ such that $Q_i \cap \tilde{X}_{\epsilon, K} \neq \emptyset$  and for any choice of $w_i \in Q_i \cap \tilde{X}_{\epsilon, K}$ that
	$$
		\begin{aligned}
			\int_{Q(c_i,2^{-K}) \cap S}|g(z)-g(w_i)|d|D^sf|(z)
			&\leq \int_{Q(w_i,2^{1-K}) \cap S}|g(z)-g(w_i)|d|D^sf|(z)\\
			&\leq \epsilon |D^sf|\big(Q(w_i,2^{1-K})\big)\\
			& \leq \epsilon |D^sf|\big(Q(c_i,2^{2-K})\big) .
		\end{aligned}
	$$
	We define the collection $E_{\epsilon, K}$ (for $K\geq \max\{K_0, K_1+1\}$) as those squares $Q_i \in F_{\epsilon, K}$ such that $Q_i\cap \tilde{X}_{\epsilon, K}\neq \emptyset$. By the choice of $K_1$, we have that
	$$
		|D^sf|\Big(\bigcup_{Q_i \in F_{\epsilon, K}\setminus E_{\epsilon, K}} Q_i\Big) =  |D^s f| (S \setminus \tilde{X}_{\epsilon,K})  < \epsilon |D^sf|(Q(0,1)).
	$$
	These two estimates together with (2) are point (3) of our claim.
	
	\step{3}{Choose an appropriate $\alpha_0>0$ and for every $0<\alpha<\alpha_0$ find an appropriate $K(\epsilon, \alpha)$}{CEZS3}
	
	For every $K \geq \max \{ K_0, K_1+1\}$ we call $W_{\epsilon, K}'$ the collection of $Q_i$, the $K$-dyadic squares $Q_i$ such that
	$$
		Q_i \notin F_{\epsilon, K}  \ \text{ and } \ |D^sf|(2Q_i) > \epsilon |D^af|(Q_i).
	$$
	Denote $\tilde{W}_{\epsilon, K}' = \bigcup_{Q_i\in W_{\epsilon, K}'}Q_i$. Notice that for any $Q_i \notin F_{\epsilon, K}$ we have $2Q_i \cap \tilde{F}_{\epsilon} = \emptyset$ by the definition of $F_{\epsilon, K}$. Then, by the second estimate of point (2), for every $K$ we have
	\begin{equation}\label{Aha}
		|D^af|(\tilde{W}_{\epsilon, K}')  \leq \epsilon^{-1}\sum_{Q_i \in W_{\epsilon, K}'} |D^sf|(2Q_i) \leq 4 \epsilon|D^af|(Q(0,1)).
	\end{equation}
	The constant $4$ is the overlap multiplicity bound for $\{2Q_i\}$, i.e. a bound for $\sum_i \chi_{2Q_i}$. The remaining squares $Q(0,1)\supset Q_i \notin F_{\epsilon, K} \cup W_{\epsilon, K}'$ satisfy the estimate $|D^sf|(2Q_i) \leq \epsilon |D^af|(Q_i)$.
	
	Call $\tilde{P}_{\alpha_0}$ the set where
	$$
	\begin{aligned}
	\tilde{P}_{\alpha_0} =& \big\{w\in Q(0,1)\setminus S; |\nabla f(w)|>\alpha_0^{-1} \big\}
	\cup\big\{w\in Q(0,1); 0<|\nabla f(w)|<\alpha_0 \big\}\\
	&\cup  \big\{w\in Q(0,1); 0< \det\nabla f(w) <\alpha_0 \big\}
	\end{aligned}
	$$ Our first observation is that as $\alpha_0 \to 0$ we have
	$$
		\begin{aligned}
			\L^2\Big(\big\{w\in Q(0,1)\setminus S; |\nabla f(w)|>\alpha_0^{-1} \big\}\Big) &\to 0\\
			\L^2\Big(\big\{w\in Q(0,1); 0<|\nabla f(w)|<\alpha_0 \big\}\Big) &\to 0\\
			\L^2\Big( \big\{w\in Q(0,1); 0< \det\nabla f(w) <\alpha_0 \big\} \Big) &\to 0,
		\end{aligned}
	$$
	because as we send $\alpha_0 \to 0$ the sets (which are nested) tend to the empty set. Recall the choice of the parameter $\delta$, chosen such that $|D^af|(A) \leq \epsilon |D^af|(Q(0,1))$ for any $A$ such that $\L^2(A)< \delta$. We find an $0<\alpha_0<\epsilon$ such that $\L^2(\tilde{P}_{\alpha_0}) <\tfrac{\delta}{16}$.

	From \cite[Theorem 3.83]{AFP} we have that $\L^2$-almost every point of $Q(0,1)$ is a point of approximate differentiability of $f$. We define $\tilde{Y}_{\epsilon, \alpha, K}$ as the set of points $w\in Q(0,1)$ such that
	\begin{equation}\label{Ruler}
	\begin{aligned}
	&\qquad  \frac{1}{\L^2(Q(w, 8r))} \int_{Q(w, 8r)} |\nabla f(z) - \nabla f(w)| d\L^2(y) > 2^{-8}\epsilon \alpha^2 \quad \text{or}\\
	&\L^2\Big(\Big\{y\in Q(w, 4r): |f(z) - f(w) - \nabla f(w)(z - w) | > \alpha^4 r\Big\}\Big) \geq 2^{-9} r^2
	\end{aligned}
	\end{equation}
	for some $0< r< 2^{-K}$. By \cite[Theorem 3.83]{AFP} the $\L^2$ measure of $\tilde{Y}_{\epsilon, \alpha, K}$ tends to zero as $K\to \infty$. Therefore we find a $K_2(\alpha)$ sufficiently large such that for any $0<\alpha <\alpha_0$ and any $K\geq K_2(\alpha)$ we have $\L^2(\tilde{Y}_{\epsilon, \alpha, K})<\delta/16$. From now on we require that $K \geq K_2$. Also we require $2^{-K_2}< \epsilon\alpha_0$.
	
	\step{4}{Designate the squares $G_{\epsilon, \alpha, K}, T_{\epsilon, \alpha, K},W_{\epsilon, \alpha, K}$ and prove (4)}{CEZS4}
	
	Notice that the bound on the measure of $\tilde{P}_{\alpha_0}$ and $\tilde{Y}_{\epsilon, \alpha, K}$ implies that the union of all $K$-dyadic squares $Q_i = Q(c_i, 2^{-K})$ such that $Q(c_i, 2^{-K-2}) \subset \tilde{P}_{\alpha_0} \cup \tilde{Y}_{\epsilon, \alpha, K}$ has measure at most $2\delta$. Therefore
	\begin{equation}\label{Punch}
		|D^af|(\tilde{P}_{\alpha_0} \cup \tilde{Y}_{\epsilon, \alpha, K}) \leq 2\epsilon |D^af|(Q(0,1))
	\end{equation} 
	because of the choice of $\delta$. We call $W_{\epsilon, \alpha, K}$ the collection of those squares either 
	\begin{enumerate}
		\item[(i)]$Q_i \in W_{\epsilon, K}'$ or
		\item[(ii)]$Q(c_i, 2^{-K-2}) \subset \tilde{P}_{\alpha_0} \cup \tilde{Y}_{\epsilon, \alpha, K}$ or
		\item[(iii)] $|D^a f|(Q_i) < 2\alpha_02^{-2K+2} = 2\alpha_0\L^2(Q_i)$
	\end{enumerate}
	In case $(iii)$ we use the fact that $\L^2(Q(0,1)) = 4$ the choice of $\alpha_0<\epsilon$ to get
	$$
		|D^a f|\Big(\bigcup_{\{i : |D^a f|(Q_i) <\alpha_02^{-2K+3}\}}Q_i\Big) \leq \sum_{\{i : |D^a f|(Q_i) < 2\alpha_02^{-2K+2}\}}|D^a f|(Q_i) \leq 2\alpha_0 4 < 8\epsilon.
	$$
	This in combination with \eqref{Aha} (for case $(i)$) and \eqref{Punch} (for case $(ii)$) prove the estimate \eqref{Judy}.
	
	All the other squares $Q_i \notin F_{\epsilon, K} \cup W_{\epsilon, \alpha, K}$ have a point $w_i \in Q(c_i, 2^{-K-1}) \setminus (P_{\alpha_0} \cup Y_{\epsilon, \alpha, K})$ and therefore they satisfy the estimates
	$$
		|\nabla f(w_i)| \leq \alpha_0^{-1}
	$$
	and
	$$
		  \text{ either } \det\nabla f(w_i) \geq \alpha_0 \text{ or } \det\nabla f(w_i) = 0
	$$
	and
	$$
		\begin{aligned}
			\int_{Q(c_i, 2^{2-K})} |\nabla f(z) - \nabla f(w_i)| d\L^2(z) &\leq \int_{Q(w_i, 2^{3-K})} |\nabla f(z) - \nabla f(w_i)| d\L^2(z)\\
			& \leq \epsilon \alpha^2 2^{-2K}.
		\end{aligned}
	$$
	Further the fact that $|D^a f|(Q_i) \geq 2\alpha_0\L^2(Q_i)$ implies that $|\nabla f(w_i)| \geq 2\alpha_0 - \epsilon\alpha \alpha_0$ and since $\alpha<\epsilon<1$ we have
	$$
		\alpha_0 \leq  |\nabla f(w_i)| \leq \alpha_0^{-1}.
	$$
	Moreover, using $Q(c_i, 2^{1-K}) \subset Q(w_i, 2^{2-K})$ and $r =2^{-K}$ in \eqref{Ruler} we have that 
	$$
		\L^2\Big(\big\{z\in Q(c_i, 2^{1-K}): |f(z) - f(w_i) - \nabla f(w_i)(z - w_i) | > \alpha^4 2^{-K} \big\}\Big) < 2^{-2K - 9}.
	$$
	Then, $Z_{i, \alpha} \subset Q(c_i, 2^{1-K})$ being the set where $|f(\cdot) - f(w_i) - \nabla f(w_i)(\cdot - w_i) | \leq \alpha^4 2^{-K} $ satisfies $\L^2\big( Q_i \setminus Z_{i, \alpha} \big) < 2^{-2K-9}$. Thus we have proved point (4).
\end{proof}

	In \cite[Proposition 3.92]{AFP} the authors introduced the set
	$$
		\Theta_f = \{w\in Q(0,1); \liminf_{r\to 0} r^{-1}|Df|(B(w,r)) >0 \}.
	$$
	We adapt slightly this notion and in the following theorem we use the sets
	$$
		\Theta_f^{\beta} = \{w\in Q(0,1); \liminf_{r\to 0} r^{-1}|Df|(B(w,r)) > \tfrac{1}{10}\beta \}
	$$
	for $\beta > 0$.

	In the Theorem`\ref{KulovyBlesk}, for each pair of neighbouring vertexes $V, \tilde{V}$ of some square $Q_i$ chosen in Theorem~\ref{CEZ}, we find sets $H_{V, \tilde{V}}$ such that when we create a quadrilateral by shifting $V$ and $\tilde{V}$ to a pair of points in $H_{V, \tilde{V}}$ then the behaviour of $f_{\rceil[V, \tilde{V}]}$ corresponds to the behaviour of $f$ inside $Q_i$. In fact we can can create a good non-straight grid for $f$ by joining neighbouring shifted vertices with segments. Further we get the useful estimates \eqref{NothingEst} and \eqref{SomethingEst}. 

\begin{thm}\label{KulovyBlesk}
	Let $f\in BV(Q(0,1))$ be an NCBV map, let $\epsilon, \beta>0,$ and let $\alpha_0<\epsilon$ be the number given by Theorem~\ref{CEZ} and let $0<\alpha<\min\{\alpha_0, 2^{-9}\}$. Let $K = K(\epsilon, \alpha)\in \en$ and $\{Q_i\}_{i = 1}^{2^{2K}}$ be the $K$-dyadic squares chosen in Theorem~\ref{CEZ}. There exists a constant $C>0$ such that the following holds. Let $V$ be a vertex of a square $Q_i$ and let $\tilde{V}$ be one of its neighbouring vertices. There exists a set $H_{V, \tilde{V}} \subset Q(V, 2^{-K-2}) \times Q(\tilde{V}, 2^{-K-2})$  with $\L^4(H_{V, \tilde{V}}) \geq \tfrac{4}{5}\L^4(Q(V, 2^{-K-2}) \times Q(\tilde{V}, 2^{-K-2}))$ with the following properties:
	\begin{enumerate}
		\item for every pair $(X, \tilde{X}) \in H_{V, \tilde{V}}$ both $X$ and $\tilde{X}$ are Lebesgue points of $f$ and we assume that
		$$
			f(X) = \lim_{r\to 0} \oint_{B(X, r)}f(z) dz \quad \text{ and } \quad f(\tilde{X}) = \lim_{r\to 0} \oint_{B(\tilde{X}, r)}f(z) dz,
		$$
		\item for any pair $(X, \tilde{X}) \in H_{V, \tilde{V}}$ the map $f_{\rceil L_{X, \tilde{X}}}$ is continuous at $X$ and $\tilde{X}$, where $L_{X, \tilde{X}}$ is the line passing through $X$ and  $\tilde{X}$,
		\item for any pair $(X, \tilde{X}) \in H_{V, \tilde{V}}$ the segment $[X\tilde{X}]$ intersects the set $\Theta_f^{\beta}$ only at Lebesgue points of the function $(f^+(x)-f^{-}(x))\otimes v(x)$ with respect to $\H^1_{\rceil \Theta_f^{\beta}}$ and further $\langle X-\tilde{X}, v(x)\rangle \neq 0$ at every such point of intersection,
		\item for any pair $(X, \tilde{X}) \in H_{V, \tilde{V}}$ it holds that $|Df|([X\tilde{X}]) = 0$,
		\item for any pair $(X, \tilde{X}) \in H_{V, \tilde{V}}$ the estimate holds
		\begin{equation}\label{NothingEst}
			|D_{\tau}f_{\rceil [X\tilde{X}]}|([X\tilde{X}]) \leq C2^{K}|Df|(2Q_i) .
		\end{equation}
		Further, if $V$ and $\tilde {V}$ are both vertices of some $Q_j \in G_{\epsilon, \alpha, K}\cup T_{\epsilon, \alpha, K}$ (the collection of squares defined in Theorem~\ref{CEZ}) then
		\begin{equation}\label{SomethingEst}
			|D_{\tau}[f(\cdot) - \nabla f(w_j)(\cdot)]_{\rceil [X\tilde{X}]}| ([X\tilde{X}]) \leq C \epsilon  |D^af|\big(Q(c_j,2^{1-K})\big) 2^{K}  
		\end{equation}
		where $w_j$ is the point chosen in Theorem~\ref{CEZ}. Moreover
		\begin{equation}\label{InZ}
			X,\tilde{X} \in Z_{j, \alpha},
		\end{equation}
		where $Z_{j,\alpha}$ is the set from Theorem~\ref{CEZ}.
	\end{enumerate}
\end{thm}
\begin{proof}
	To satisfy point (1) it suffices to consider a `good' representative of $f$ and eliminate the set of non-Lebesgue points from $Q(V,2^{-K-2})$ and $Q(\tilde{V},2^{-K-2})$, which is a set of zero measure. Call $N_1$ the set of all $(X,\tilde{X}) \in Q({V},2^{-K-2})\times Q(\tilde{V},2^{-K-2})$ such that either $X$ is not a Lebesgue point of $f$ or $\tilde{X}$ is not a Lebesgue point of $f$. Clearly $\L^4(N_1) = 0$.
	
	During the course of this proof we use the mapping
	$$
	\Psi: (X, \tilde{X}) \to \Bigg(\frac{X- \tilde{X}}{|X- \tilde{X}|}, \,
	 \langle X, \tfrac{X- \tilde{X}}{|X- \tilde{X}|}\rangle, \,
	   \langle \tilde{X}, \tfrac{X- \tilde{X}}{|X- \tilde{X}|} \rangle,  \,
	  \bigg\langle X, \bigg(\begin{matrix}
	  0& \  -1\\
	  1& \  0\\
	  \end{matrix}\bigg)\tfrac{X- \tilde{X}}{|X- \tilde{X}|}\bigg\rangle \Bigg).
	$$
	
	In the image of $\Psi$ we have the measure
	$$\mu := \H^1_{\rceil \{|X|=1\}}\times \L^3.$$
	The first space, $\H^1_{\rceil \{|X|=1\}}$, is bi-Lipschitz equivalent with $\L^1_{\rceil (0,2\pi)}$ and so the measure $\mu$ is locally bi-Lipschitz equivalent with $\L^4$ on $\er^4$ which is $\H^4$ on $\er^4$. In this sense $\mu$ is equivalent with $\H^4$ on $\{(x_1,x_2,x_3,x_4,x_5)\in \er^5: x_1^2+x_2^2 = 1  \}$. We have the validity of the area formula (for example see \cite[Theorem 3.8]{EG}) for $\Psi$, i.e.
	$$
		\int_{E}\J_4\Psi(X, \tilde{X}) \ d\L^4(X, \tilde{X}) = \int_{\er^5} \H^0\big(\Psi^{-1}(Z)\big) \ d\H^4(Z),
	$$
	where $\J_4\Psi(X, \tilde{X}) = \sqrt{\sum_{\lambda \in \Lambda(5,4)}[\det \nabla\Psi_{\lambda}(X, \tilde{X})]^2}$ is calculated by the well-known Cauchy-Binet formula. In fact it is easy to see that $\Psi$ is injective and
	$$
		\int_{E}\J_4\Psi(X, \tilde{X}) \ d\L^4(X, \tilde{X})  = \H^4(\Psi(E)) \approx \mu(\Psi(E)).
	$$
	When the arguments satisfy $|X - \tilde{X}| \approx 2^{-K}$ then $\Psi$ is in fact $2^{K}$ bi-Lipschitz. It is not difficult to calculate under these circumstances that $\J_4 \Psi \approx 2^{K}$. In particular
	\begin{equation}\label{NCond}
	\mu(\Psi(N)) = 0 \text{ exactly when } \L^4(N) = 0
	\end{equation}
	and in the following we use this property repeatedly.
	
	From \cite[Theorem 3.107]{AFP} we have that for any choice of $|u|=1$, $\L^2$ almost every choice of $X \in Q(V,2^{-K-2}),$ and $\H^1$ almost every choice of $\tilde{X} \in Q(\tilde{V},2^{-K-2})$ with $\tilde{X} \in X + \er u$ there exists a partial derivative in the direction $u$ at $X$ and at $\tilde{X}$ and therefore $f_{\rceil [X\tilde{X}]}$ is continuous at $X$ (and $\tilde{X}$). Call $N_2$ the set of all $(X,\tilde{X}) \in Q({V},2^{-K-2})\times Q(\tilde{V},2^{-K-2})$ such that either there is no partial derivative in the direction $\frac{X-\tilde{X}}{|X-\tilde{X}|}$ at $X$ or there is no partial derivative in the direction $\frac{X-\tilde{X}}{|X-\tilde{X}|}$ at $\tilde{X}$. We have from the above that $\mu(\Psi(N_2)) = 0$. Then \eqref{NCond} implies that $\L^4(N_2) = 0$.
	
	As a step towards proving (3) we show that the pairs of $X, \tilde{X}$ whose corresponding segments intersect $\Theta_{f}^{\beta}$ at non-Lebesgue points of the map in (3) has zero $\L^4$ measure. Although this is a standard result of structure theory we give some details here. We denote the so-called jump set of $f$ as $J_f$. The set $\Theta_f^0 \supset \Theta_f^{\beta}$ is a superset of $J_f$ (see \cite[Proposition 3.92]{AFP}) and $\H^1(\Theta_f^0 \setminus J_f) = 0$ implying that $D^jf = (f^+-f^{-})\otimes u \H^1_{\rceil \Theta_f^0}$ (see \cite[Lemma 3.76, Theorem 3.77]{AFP}). Then, since $|f^+(x)-f^{-}(x)|> \tfrac{1}{10}\beta$ for $\H^1$ almost every $x \in\Theta_f^{\beta}$, we have that $\Theta_f^{\beta}$ is both 1-rectifiable and $\H^1(\Theta_f^{\beta})< \infty$ (see Theorem~\ref{FV}). Also we have that $f^+-f^{-} \in L^1(\Theta_f^\beta, \H^1)$. Therefore $\H^1$ almost every point of $\Theta_{f}^{\beta}$ is a Lebesgue point of $f^+-f^{-}$ with respect to $\H^1_{\rceil \Theta_{f}^{0}}$. That is, calling $T_1$ the set of non-Lebesgue points of $f^+-f^{-}$ with respect to $\H^1_{ \rceil\Theta_{f}^{0}}$, we have $\H^1(T_1) = 0$. This in turn implies that for any direction $|u|=1$ and its corresponding projection $\pi_u(\cdot ) := \cdot  - u \langle \cdot, u\rangle$ we have that $\H^1(\pi_u(T_1)) = 0$. Let us call $N_3$ the set of $(X, \tilde{X})$ such that $[X\tilde{X}]$ intersects $T_1$. Since $\H^1\Big(\pi_{\tfrac{X- \tilde{X}}{|X- \tilde{X}|}}(T_1)\Big) = 0$ for every possible value of $\frac{X- \tilde{X}}{|X- \tilde{X}|}$, the Fubini theorem gives that $\mu(\Psi(N_3)) = 0$ and, by \eqref{NCond}, $\L^4(N_3) = 0$.
	
	Now we show that $\L^4$ almost every choice of $X$ and $\tilde{X}$ does not meet $\Theta_f^{\beta}$ tangentially, which is a claim of point (3). The set of directions $|u| =1$ such that
	$$
		\H^1\big(\{w \in \Theta_{f}^{\beta} : v(w) \bot u \}\big) > 0
	$$
	is at most countable and so has zero measure since $\H^1(\Theta^\beta_f)<\infty.$ We call this set of directions $T_2$. Therefore, for almost every direction $|u|=1$ we have that
	$$
		\H^1\Big(\pi_u\big(Q(V,2^{-K-2})\big) \cap \pi_u\big(\{w\in \Theta_{f}^{\beta} : v(w) \bot u \}\big)\Big) = 0
	$$
	and the same estimate holds after replacing $V$ with $\tilde{V}$. Choose any direction $|u|=1$ with $u\notin T_2$ such that
	$\H^1\big(\pi_u(Q(V,2^{-K-2})) \cap \pi_u(Q(\tilde{V},2^{-K-2}))\big)>0$. We have that the set of points $\hat{X}$ such that $\hat{X} \in \pi_u\big(Q(V,2^{-K-2})\big) \cap \pi_u\big(Q(\tilde{V},2^{-K-2})\big)$ and $\hat{X} \in \pi_u(\{w \in \Theta_{f}^{\beta} : v(w) \bot u\})$ has $\H^1$ measure equal zero. This holds for any vector $|u| =1$, $u\notin T_2$. Call $N_4$ the set of pairs $(X,\tilde{X})$ such that $\frac{X- \tilde{X}}{|X- \tilde{X}|} \in T_2$ or 
	 $$
	 	\pi_{\tfrac{X- \tilde{X}}{|X- \tilde{X}|}}(X) \in \pi_{\tfrac{X- \tilde{X}}{|X- \tilde{X}|}}\Big(\{x \in \Theta_{f}^{\beta} : v(x) \bot \tfrac{X- \tilde{X}}{|X- \tilde{X}|} \}\Big).
	 $$
	 The Fubini theorem and \eqref{NCond} guarantee that $\L^4(N_4) = 0$.
	 
	 Let us fix a direction $|u| = 1$ then almost every line $L$ parallel to $u$ has $|Df|(L\cap Q(0,1)) = 0$. Therefore using the bi-Lipschitz quality of $\Psi$ we have that the set of pairs $(X,\tilde{X})$ (call it $N_5$) such that  $|Df|([X\tilde{X}]) >0$ satisfies $\L^4(N_5) = 0$.
	 
	 Let $|u| =1$ be any vector such that $P_{u, V,\tilde{V}, K} := \pi_u(Q(V,2^{-K-2})) \cap \pi_u(Q(\tilde{V},2^{-K-2}))\neq \emptyset$. For any $p \in P_{u, V,\tilde{V}, K}$ let us also denote $M_{u,V,K, p} := \pi^{-1}_u(p)\cap Q(V,2^{-K-2})$ and $M_{u,\tilde{V} K,p} := \pi^{-1}_u(p)\cap Q(\tilde{V},2^{-K-2})$. Then, (because each $[X,\tilde{X}]\subset 2Q_i$), we have by the Fubini theorem, \cite[Theorem 3.107]{AFP} and Lemma~\ref{Stupido}
	 $$
	 	\begin{aligned}
	 		\int_{P_{u, V,\tilde{V}, K}} \int_{M_{u,{V},K,U} \times M_{u,\tilde{V},K,U}}& |D_{\tau}f_{\rceil[X\tilde{X}]}| ([X\tilde{X}]) \, d\H^1\times\H^1(X,\tilde{X}) \,d\H^1(U) \\
	 		&\leq C2^{-2K}|\langle Df, u\rangle|(2Q_i) \\
	 		&\leq C2^{-2K}|Df|(2Q_i).
	 	\end{aligned}
	 $$
	 Integrating this with respect to $u$ and then using the change of variables formula with $\Psi$ (note that $\J_4 \Psi \approx 2^K$) we get
	 $$
	 	\int_{Q(V,2^{-K-2}) \times Q(\tilde{V},2^{-K-2})} |D_{\tau}f_{\rceil[X\tilde{X}]}|([X\tilde{X}]) d\L^4(X,\tilde{X})  \leq C2^{-3K}|Df|(2Q_i).
	 $$
	 Call $N := N_1\cup N_2\cup N_3\cup N_4\cup N_5$, then $\L^4(N) = 0$. Using the Chebyshev inequality we find a constant $\lambda >0$ (dependent only on $f$) and a set $H_{V,\tilde{V}} \subset Q(V,2^{-K-2}) \times Q(\tilde{V},2^{-K-2}) \setminus N$ such that $\L^4(H_{V,\tilde{V}}) \geq \tfrac{4}{5} \L^4(Q(V,2^{-K-2}) \times Q(\tilde{V},2^{-K-2}))$ and
	 $$
	 	|D_{\tau}f_{\rceil[X\tilde{X}]}|( [X\tilde{X}])  \leq C\lambda 2^{K}|Df|(2Q_i)
	 $$
	 for any $(X,\tilde{X}) \in H_{V,\tilde{V}}$, this is \eqref{NothingEst}.
	 
	 Now let us prove \eqref{SomethingEst} and \eqref{InZ}. We assume that $Q_j=Q(c_j, 2^{-K}) \in G_{\epsilon, \alpha, K}\cup T_{\epsilon, \alpha, K}$. Specifically by Theorem \ref{CEZ} there exists $w_j\in Q(c_j, 2^{-K-2})$
	 \begin{equation}\label{Senile}
	 	\int_{Q(c_j, 2^{2-K})} |\nabla f(z) - \nabla f(w_j)| d\L^2(z) \leq \epsilon \alpha^2 2^{-2K},
	 \end{equation}
	 and
	 \begin{equation}\label{SleepyJoe}
	 |D^sf|(2Q_j) \leq \epsilon |D^a  f|(Q_j)
	 \end{equation}
	 and finally that there exists a set $Z_{j, \alpha} \subset Q(c_j, 2^{1-K})$ with $\L^2\big(Q(c_j, 2^{1-K}) \setminus Z_{j,\alpha}\big)\leq 2^{-2K-9}$ such that
	 $$
	 		\|f(\cdot) - f(w_j) - \nabla f(w_j)(\cdot - w_j)\|_{L^\infty(Z_{j,\alpha})} < \alpha^4 2^{-K}
	 $$
	 and
	 $$
	 \alpha_0\leq |\nabla f(w_j)|\leq \alpha^{-1}_0, \ \alpha_0< \det \nabla f(w_j) .
	 $$
	 We integrate over all lines parallel to $|u| =1$ a vector such that $P_{u, V,\tilde{V}, K} := \pi_u(Q(V,2^{-K-2})) \cap \pi_u(Q(\tilde{V},2^{-K-2}))\neq \emptyset$ and over $M_{u,V,K,U} = \pi^{-1}_u(U)\cap Q(V,2^{-K-2})$ and $M_{u,\tilde{V},K,U} = \pi^{-1}_u(U)\cap Q(\tilde{V},2^{-K-2})$. First we decompose into the singular and absolutely continuous part
	 $$
		 \begin{aligned}
	 		&\int_{P_{u, V,\tilde{V}, K}}\int_{M_{u,{V},K,U}}\int_{M_{u,\tilde{V},K,U}}|D_{\tau}[f(\cdot) - \nabla f(w_j)(\cdot)]_{\rceil [X\tilde{X}]}|([X\tilde{X}]) \, d\tilde{X} \, dX \,d\H^1(U) \\
			 &\quad \leq  \int_{P_{u, V,\tilde{V}, K}}\int_{M_{u,{V},K,U}}\int_{M_{u,\tilde{V},K,U}}  \int_{[X\tilde{X}]}|\nabla f(z) - \nabla f(w_j)|\,d\H^1(z) \, d\tilde{X} \, dX \,d\H^1(U)\\
	 		& \qquad  +\int_{P_{u, V,\tilde{V}, K}}\int_{M_{u,{V},K,U}}\int_{M_{u,\tilde{V},K,U}} \langle |D^sf|, u\rangle ([X\tilde{X}]) d\tilde{X} \, dX \,d\H^1(U).\\
	 	\end{aligned}
	 $$
	 Now we use Fubini and the fact that $\{X\in Q(V,2^{-K-2}) : \pi_u(X)\in P_{u, V,\tilde{V}, K} \}, \{\tilde{X}\in Q(\tilde{V},2^{-K-2}) : \pi_u(\tilde{X})\in P_{u, V,\tilde{V}, K} \}\subset 2Q_i$. Also we use the estimate that every slice of $Q(V,2^{-K-2})$ has diameter bounded by $2^{-K}$ to get
	 $$
	 	\begin{aligned}
	 		&\int_{P_{u, V,\tilde{V}, K}}\int_{M_{u,{V},K,U}}\int_{M_{u,\tilde{V},K,U}}  \int_{[X\tilde{X}]}|\nabla f(z) - \nabla f(w_j)|\,d\H^1(z) \, d\tilde{X} \, dX \,d\H^1(U)\\
	 		& \quad  +\int_{P_{u, V,\tilde{V}, K}}\int_{M_{u,{V},K,U}}\int_{M_{u,\tilde{V},K,U}} \langle |D^sf|, u\rangle ([X\tilde{X}]) d\tilde{X} \, dX \,d\H^1(U)\\
	 		&\qquad \leq 2^{-2K} \int_{2Q_i}|\nabla f - \nabla f(w_i)| + 2^{-2K} |D^sf|(2Q_i)
	 	\end{aligned}
	 $$
	 We use \eqref{Senile} and \eqref{SleepyJoe} and then $\alpha< \alpha_0\leq |\nabla f(w_i)| $ to get
	 $$
	 	\begin{aligned}
	 		2^{-2K} \int_{2Q_i}|\nabla f - \nabla f(w_i)| +& 2^{-2K} |D^sf|(2Q_i)\\
	 		&\leq   \epsilon \alpha^2 2^{-4K} + C2^{-2K} \epsilon |D^a f|(Q_i)\\
			&\leq  C\epsilon \alpha^2 2^{-4K} + C2^{-4K} \epsilon |\nabla f(w_i)| + C\epsilon \alpha^2 2^{-4K}\\
			&\leq C2^{-4K} \epsilon |\nabla f(w_i)|.\\
		\end{aligned}
	$$
	Integrating the above equations over all $|u|=1$ and using the change of variables formula with $\Psi$ (recall that $\J_4 \Psi \approx 2^K$) and denoting $A = (Q(V,2^{-K-2})\cap Z_{j, \alpha}) \times (Q(\tilde{V},2^{-K-2})\cap Z_{j,\alpha})$ we get
	$$
		\int_{A} |D_{\tau}[f(\cdot) - \nabla f(w_j)(\cdot)]_{\rceil[X\tilde{X}]}|( [X\tilde{X}]) d\L^4(X,\tilde{X}) \leq C\epsilon 2^{-3K}|D^af|(2Q_i).
	$$
	Because $\L^2\big(Q(c_j, 2^{1-K}) \setminus Z_{j,\alpha}\big)\leq  2^{-2K-9}$, we have that
	$$
	\begin{aligned}
	\L^2(Q(V, 2^{-K-2}) \cap Z_{j,\alpha})
	&\geq \L^2(Q(V, 2^{-K-2})) - \L^2\big(Q(c_j, 2^{1-K}) \setminus Z_{j,\alpha}\big)\\
	& \geq 2^{-2K-4} - 2^{-2K-9} = \tfrac{31}{32}2^{-2K-4}\\
	&= \tfrac{31}{32} \L^2\big(Q(V, 2^{-K-2})\big).
	\end{aligned}
	$$
	This implies that
	$$
		\L^4\Big(\big(Q(V, 2^{-K-2}) \cap Z_{j,\alpha}\big) \times \big(Q(\tilde{V}, 2^{-K-2}) \cap Z_{j,\alpha}\big)  \Big) \geq \tfrac{9}{10}\L^4\big(Q(V, 2^{-K-2}) \times Q(\tilde{V}, 2^{-K-2})\big) .
	$$
	Again we have, up to increasing the value of $\lambda$, that
	$$
		|D_{\tau}[f(\cdot) - \nabla f(w_j)(\cdot)]_{\rceil[X\tilde{X}]}|( [X\tilde{X}]) \leq C\lambda\epsilon 2^{K}|D^af|(2Q_i)
	$$
	for all $(X,\tilde{X}) \in H_{V,\tilde{V}} \subset (Z_{j, \alpha}\times Z_{j, \alpha})$ while simultaneously $\L^4(H_{V, \tilde{V}}) \geq \tfrac{4}{5}\L^4(Q(V, 2^{-K-2}) \times Q(\tilde{V}, 2^{-K-2}))$, thus proving \eqref{SomethingEst} and concluding our proof.
\end{proof}

	In the following proposition we use the following notation. Let $\P\subset \er^2$ be an injective continuous image of a circle with $\H^1(\P)<\infty$. Denote the closure of the bounded component of $\er^2 \setminus \P$ by $\tilde{\P}$. Let $p_1,p_2 \in  \tilde{\P}$, we define
	\begin{equation}
		d_{\mathcal P}(p_1,p_2)=\inf\Big\{l(\gamma)\colon \gamma \textrm{ is a path joining } p_1,\, p_2; \gamma\subset  \overline{\tilde{\P}}\Big\},
	\end{equation}
	where by a path joining $p_1$ and $p_2$ in  $\tilde{\P}$ we mean a continuous curve $\gamma:[0,1]\to \R^2$ such that $\gamma(0)=p_1$ and $\gamma(1)=p_2$.
	
	The following proposition is the utilization of the $NCBV^+$ condition on a good non-straight grid $\Gamma$ chosen using the previous theorem. It gives us a map $\phi$ defined on $\Gamma$. The utility of $\phi$ is that we are able to find a homeomorphic extension of $\phi$ with estimates that allow us to prove area-strict convergence.

	\begin{prop}\label{DefiningPhi}
	For every $\epsilon >0$ let $\alpha_0>0$ as in Theorem \ref{CEZ}. For every $0<\alpha<\alpha_0$ there exists a good non-straight grid for $f$ called $\Gamma$ and a function $\phi$ defined on $\Gamma$ such that
		\begin{enumerate}
			\item $\Gamma$ is admissible for $f$,
			\item every component of $Q(0,1)\setminus \Gamma$ is a convex quadrilateral and contains exactly one point $c_i$, where $c_i$ is the centre of a square $Q_i$ the set $\{Q_i\}_{i = 1}^{2^{2K}}$ of $K$-dyadic squares in Theorem~\ref{CEZ} and Theorem~\ref{KulovyBlesk} (thanks to this we call the quadrilaterals $\Q_i$ the components of $Q(0,1)\setminus \Gamma$ containing $c_i$),
			\item if $Q_i \in E_{\epsilon, K}$ (the set from Theorem~\ref{CEZ}) and finding $u_i,v_i$ such that that the function $g$ of Theorem~\ref{CEZ} satisfies $g(w_i) = u_i\otimes v_i$ we have
			\begin{equation}\label{SingEst1}
					\int_{\pi_{v_i}(\Q_i)}d_{\phi(\partial \Q_i)}(\phi(X_*),\phi(X^*))\,d\H^1 \leq (1+\epsilon)|Df|(\Q_i) + C\epsilon 2^{-2K}
			\end{equation}
			and
			\begin{equation}\label{SingEst2}
					\int_{\pi_{v_i^{\bot}}(\Q_i)}d_{\phi(\partial \Q_i)}(\phi(Z^*),\phi(Z_*)) \, d\H^1(Z) \leq C\epsilon |D^sf|(4Q_i \cap S ) +C2^{-2K}\epsilon.
			\end{equation}
			where $\pi_{v_i}(x) = v_i^{\bot}\langle x,v_i^{\bot}\rangle$ and $\pi_{v_i^{\bot}}(x) = v_i\langle x,v_i\rangle$ and where for each $X\in \pi_{v_i}(\Q_i)$ the points $X_*, X^*$ are the two distinct points in $\partial \Q_i$ such that $\pi_{v_i(X_*)} =  \pi_{v_i(X^*)} =  X$ and similarly $\pi_{v_i^{\bot}}(Z_*) = \pi_{v_i^{\bot}}(Z^*) = Z \in \pi_{v_i^{\bot}}(\Q_i)$,
			\item for all $i=1,\dots 2^{2K}$ it holds that
			$$
				|D_{\tau}\phi|(\partial \Q_i) \leq C2^{K}|Df|(2Q_i),
			$$
			\item if $Q_i \in G_{\epsilon, \alpha, K}$ then $\phi(x,y) = f(x,y)$ at each $(x,y)$ vertex of $\Q_i$ and $\phi$ is linear on each side of $\partial \Q_i$
			\item if $Q_i \in T_{\epsilon, \alpha, K}$ then
			\begin{equation}\label{TropicalWood}
				\int_{\partial \Q_i} |\partial_{\tau}\phi - \nabla f(w_i)\tau| \,d\H^1 \leq C\epsilon 2^{K} |Df|(2 Q_i).
			\end{equation}
		\end{enumerate}
	\end{prop}
	\begin{proof}
		\step{1}{Choice of $\Gamma$}{CoG}
		
		Each vertex $V$ of each square $Q_i$ has at most four neighbouring vertexes, call them $\tilde{V_1}, \dots, \tilde{V_4}$. By the Fubini theorem we have the existence of an $X \in Q(V, 2^{-K-2})$ such that $\L^2(\{\tilde{X}: (X,\tilde{X})\in H_{V,\tilde{V}_j} \} )\geq \frac{4}{5}\L^2(Q(\tilde{V}_j, 2^{-K-2}) )$ holds for all $j=1,2,3,4$ simultaneously. It follows that it is possible to choose $X_V \in Q(V, 2^{-K-2})$ for each $V$ such that for every pair of neighbours $V$ and $\tilde{V}$ we have that $(X_{V}, X_{\tilde{V}})\in H_{V,\tilde{V}}$. The squares $Q_i$ can be described as the convex hull of their vertexes $V_1, \dots V_4$ and the corresponding quadrilateral $\Q_i$ is the convex hull of $X_{V_1}, \dots , X_{V_4}$. By definition it is not hard to check that $\{\Q_i\}$ are pairwise disjoint outside their mutual boundaries, they are convex quadrilaterals and $X_{V}$ lies in the boundary of $\Q_i$ exactly when $V$ is a vertex of $Q_i$. This is point (2). Theorem~\ref{KulovyBlesk} and $(X_{V}, X_{\tilde{V}})\in H_{V,\tilde{V}}$ guarantees point (1) of our claim. For every unit vector $v$ the set $\{(X,\tilde{X}) : \frac{X-\tilde{X}}{|X-\tilde{X}|} = v \}$ has Hausdorff dimension $3$ and therefore has $\L^4$ measure $0$. Therefore it is not restrictive to assume that the sides of $\Q_i$ are not parallel to $v_i$ or $v_i^{\bot}$.
		
		In order to prove points (3)-(6) we need to define the map $\phi$. Before we start the construction of $\phi$ itself we add extra lines to the grid $\Gamma$ to get an augmented grid $\tilde{\Gamma}$. Although we do not need to use $f$ on $\tilde{\Gamma} \setminus \Gamma$ we use the extra lines added to get a good parametrization of the geometric representative of $f$ on $\Gamma$. This is equivalent to the concept of guidelines from \cite{CKR}.

		\step{2}{Construction of $\tilde{\Gamma}$ by the choice of guidelines}{CotG}

		For each $Q_i \in E_{\alpha, K}$, by \eqref{Piano}, we have $|u_i| = |v_i|=1$ such that
		$$
			\int_{S\cap 2Q_i}|g(z)-u_i\otimes v_i|d|D^sf|(z) \leq \epsilon |D^sf|(4Q_i \cap S ).
		$$
		For almost every $X\in \pi_{v_i}(\Q_i)$ we have $f_{\rceil [X+\er v_i]\cap Q(0,1)}$ in $BV$ on $[X+\er v_i]\cap Q(0,1)$. The corresponding claim holds for almost every $Z\in \pi_{v_i^{\bot}}(\Q_i)$. It follows from the $BV$ on lines characterization and Lemma~\ref{Stupido} that
		$$
			|\langle Df, v_i\rangle | (\Q_i) = \int_{\pi_{v_i}(\Q_i)}|D_{\tau}f_{\rceil [X+\er v_i]\cap \Q_i}|([X+\er v_i]\cap Q(0,1)) \, d\H^1(X)
		$$
		and, by \eqref{Piano},
		$$
			|\langle Df, v_i^{\bot}\rangle | (\Q_i) = \int_{\pi_{v_i^{\bot}}(\Q_i)}|D_{\tau}f_{\rceil [Z+\er v_i^{\bot}]\cap \Q_i}|([Z+\er v_i^{\bot}]\cap Q(0,1))\, d\H^1(Z) \leq \epsilon|D^sf|(4Q_i).
		$$
		
		Since $\Gamma$ is admissible for $f$ we have $f_{\rceil \Gamma}$ is $BV$ on $\Gamma$. Therefore for each pair, $V_1, V_2$ of neighbouring vertices of $\Q_i$
		\begin{equation}\label{FiniteBigJumps}
		\text{there exists a finite set } J_{V_1,V_2} = \{X \in [V_1V_2]: |D_{\tau} f_{\rceil [V_1V_2]}|(\{X\}) > \epsilon 2^{-2K}\}.
		\end{equation}
		We call the cardinality of this finite set $\mathfrak{K}$.
		
		We now chop $\Q_i$ into slices parallel to $v_i$. Let us have a locally finite decomposition of $\er$ into pairwise disjoint intervals indexed by $m\in \mathbb{Z}$ called $I_{m}$. Then we define $S_{i, m} = \Q_i \cap \{(x,y)\in \er^2: \langle (x,y), v^{\bot}_i \rangle \in I_{m} \}$. Because the sides of $\Q_i$ are not parallel to $v_i$ it holds that for every $X \in [\pi_{v_i}(\Q_i)]^{\circ}$ there are exactly two distinct points $X_*,X^* \in \partial \Q_i$ such that $\pi_{v_i}(X_*) = \pi_{v_i}(X^*) = X$. For simplicity denote the points $X_*$ and $X^*$ such that $\langle X_*, v_i \rangle < \langle X^*, v_i \rangle$. In fact, the sides of $\partial Q$ are separated into two categories, either the points of the side are all $X^*$-type points (an upper side) or all the points of the side are $X_*$-type points (a lower side). Then we denote $S^{+}_{i,m}$ the set of points $X^* \in \partial \Q_i \cap S_{i,m}$ that lie on an ``upper'' side (with respect to $v_i$). Similarly we call $S^{-}_{i,m}$ the set of points $X^* \in \partial \Q_i \cap S_{i,m}$ that lie on a ``lower'' side (with respect to $v_i$). By choosing the intervals $I_m$ carefully we can achieve that either
		\begin{equation}\label{Alabama}
			|D_{\tau} f_{\rceil \Gamma}|(\partial \Q_i \cap S_{i,m}) < \epsilon 2^{-2K} \ \text{ or } \ \H^1(S_{i,m}^{\pm}) < \frac{\epsilon}{2^{2K}|D_{\tau} f_{\rceil \Gamma}|(\Gamma)\mathfrak{K}}
		\end{equation}
		for every $i,m$. We may assume also that $S_{i,m}^{\pm}$ is always a segment.
		
		The argument from the previous paragraph can be repeated in the perpendicular direction $v_i^{\bot}$ (and if necessary increase the number $\mathfrak{K}$). This gives us a finite number of segments $T_{i,m}^{\pm}$ covering $\partial \Q_i$ and either
		$$
			|D_{\tau} f_{\rceil \Gamma}|(\partial \Q_i \cap T_{i,m}) < \epsilon 2^{-2K} \ \text{ or } \ \H^1(T_{i,m}^{\pm}) < \frac{\epsilon}{2^{2K}|D_{\tau} f_{\rceil \Gamma}|(\Gamma)\mathfrak{K}}.
		$$
		As before we separate $\partial \Q_i$ into upper and lower sides (with respect to $v_i^{\bot}$) and points on lower sides we denote by $Z_*$ and points on upper sides we denote by $Z^*$. Note that even if $S_{i,m}^- = T_{i,m}^-$ then $S_{i,m}^+ \neq T_{i,m}^+$.
		
		By Remark~\ref{GGG}, we have for $\H^1$-almost every $X\in \pi_{v_i}(\Q_i)$ that
		\begin{equation}\label{Choice}
			\Gamma \cup ([X+ v_i\er]\cap \Q_i) \cup ([Z + v_i^{\bot}\er]\cap \Q_i)	\text{ is admissible for $f$, for $\H^1$-almost every $Z$.}
		\end{equation}
		Then for every $S_{i,m}$ such that $|D_{\tau} f_{\rceil \Gamma}|(\partial \Q_i \cap S_{i,m}) < \epsilon 2^{-2K}$ we find an $X_{i,m}$ satisfying \eqref{Choice} and such that
		\begin{equation}\label{EstimateHelper}
			\begin{aligned}
				\H^1\big(\pi_{v_i}(S_{i,m})\big) |D_{\tau}f_{\rceil [X_{i,m}+\er v_i]\cap \Q_i}|([X_{i,m}+\er v_i]\cap \Q_i)  \leq |D_{v_i}f|(S_{i,m}\cap \Q_i).
			\end{aligned}
		\end{equation}
		By $X_{i,m}^*$ we denote the point in $[X_{i,m}+\er v_i]\cap S_{i,m}^+$ and by $X_{i,m, *}$ we denote the point in $[X_{i,m}+\er v_i]\cap S_{i,m}^-$ (for clarification see Figure~\ref{Fig:ASet}).
		
		Similarly, for each $T_{i,m}$ we find an $Z_{i,m} \in T_{i,m}$ such that 
		\begin{equation}\label{Lego}
			\Gamma \cup \bigcup_{m, i} ([X_{i,m}+ v_i\er]\cap \Q_i) \cup ([Z_{i,m} + v_i^{\bot}\er]\cap \Q_i)
		\end{equation}
		is admissible for $f$ and
		\begin{equation}\label{EstimateHelperT}
			\begin{aligned}
				\H^1\big(\pi_{v_i^{\bot}}(T_{i,m})\big) |D_{\tau}f_{\rceil [Z_{i,m}+\er v_i^{\bot}]\cap \Q_i}|([Z_{i,m}+\er v_i^{\bot}]\cap \Q_i) \leq |D_{v_i^{\bot}}f|(T_{i,m}\cap \Q_i)
			\end{aligned}
		\end{equation}
		for every $T_{i,m}$ such that $|D_{\tau} f_{\rceil \Gamma}|(T_{i,m}) < \epsilon 2^{-2K}$. The choice of $X_{i,m}, Z_{i,m}$ will be exactly what we need to get the estimate \eqref{SingEst1} and \eqref{SingEst2}. To fix a temporary notation, we call $\hat{\Gamma}$ the good non-straight grid for $f$ which is constructed by repeating the process in \eqref{Lego} for every $\Q_i$ such that $Q_i\in E_{\epsilon, K}$.

		The Figure~\ref{Fig:ASet} can help orient the reader in the following construction. For $S_{i,m}$ such that  $|D_{\tau} f_{\rceil \Gamma}|(\partial \Q_i \cap S_{i, m}) \geq \epsilon 2^{-2K}$ we find pairs of points $X_{i,m}^+, X_{i,m}^-\in \pi_{v_i}(S_{i,m})$ satisfying \eqref{Choice} such that the sets
		$$
			S_{i,m}^l = \{(x,y)\in S_{i,m}: \langle (x,y), v_i^{\bot} \rangle \leq \langle X_{i,m}^-, v_i^{\bot} \rangle \}
		$$
		and
		$$
			S_{i,m}^r = \{(x,y)\in S_{i,m}: \langle (x,y), v_i^{\bot} \rangle \geq \langle X_{i,m}^+, v_i^{\bot} \rangle \}
		$$
		satisfy
		\begin{equation}\label{BadSegemnts}
			\begin{aligned}
				|D_{\tau} f_{\rceil \Gamma}|(\partial\Q_i \cap \pi_{v_i}^{-1} (S_{i,m}^l)) < \epsilon 2^{-2K} \text{ and } 
				|D_{\tau} f_{\rceil \Gamma}|(\partial\Q_i \cap \pi_{v_i}^{-1} (S_{i,m}^r)) < \epsilon 2^{-2K} .
			\end{aligned}
		\end{equation}
		Similarly, for $T_{i,m}$ such that $|D_{\tau} f_{\rceil \Gamma}|(\partial \Q_i \cap T_{i, m}) \geq \epsilon 2^{-2K}$ we find pairs of points $Z_{i,m}^+, Z_{i,m}^-\in \pi_{v_i}(T_{i,m})$ such that
		$$
			\tilde{\Gamma} : = \hat{\Gamma} \cup \bigcup_m ([X^{\pm}_{i,m}+ v_i\er]\cap \Q_i) \cup ([Z^{\pm}_{i, m} + v_i^{\bot}\er]\cap \Q_i)
		$$
		is admissible for $f$ and such that the sets
		$$
			T_{i,m}^l = \{(x,y)\in \er^2: \langle (x,y), v_i \rangle \leq \langle Z_{i,m}^-, v_i \rangle \}
		$$
		and
		$$
			T_{i,m}^r = \{(x,y)\in \er^2: \langle (x,y), v_i \rangle \geq \langle Z_{i,m}^+, v_i \rangle \}
		$$
		satisfy
		\begin{equation}\label{MoreBadSegemnts}
			\begin{aligned}
				|D_{\tau} f_{\rceil \Gamma}|(\partial\Q_i \cap \pi_{v_i^{\bot}}^{-1} (T_{i,m}^l)) < \epsilon 2^{-2K} \text{ and } 
				|D_{\tau} f_{\rceil \Gamma}|(\partial\Q_i \cap \pi_{v_i^{\bot}}^{-1} (T_{i,m}^r)) < \epsilon 2^{-2K} .
			\end{aligned}
		\end{equation}
		As argued above we have that almost every choice of $X_{i,m}^{\pm}$ permits almost any choice of $Z_{i,m'}^{\pm}$ and so there is no obstacle in choosing $X_{i,m}^{\pm}$ and $Z_{i,m}^{\pm}$ for every $m$ and for each $\Q_i$ so that the resulting set $\tilde{\Gamma}$ is a good non-straight grid for $f$.
		 
		\step{3}{Injective approximations of the geometric representative of $f$}{Station}
		
		By $\hat{\phi}$ we denote the geometric representative of $f$ on $\tilde{\Gamma}$ as defined in \eqref{defh} and its following paragraphs. Then Lemma~\ref{ArrivalGrid} provides a good arrival grid, $\mathcal{G}$, associated with $\tilde{\Gamma}$ and $\hat{\phi}$ with side length $\kappa = \frac{\epsilon}{2^K+1}$ in the sense of Definition \ref{good arrival grid}.
		
		By the definition of the good arrival grid $\mathcal{G}$, the set $P:=\hat{\phi}^{-1}(\mathcal G)\cap \tilde{\Gamma}$ is finite and does not contain any vertices of $\tilde{\Gamma}$. Nor does $\hat{\phi}(\tilde{\Gamma})$ intersect any vertex of $\G$ (the points denoted as $(w_n, z_m)$ in Definition~\ref{good arrival grid}). 
		Further, for every point $(x,y)\in P$, it holds that the derivative of $\hat{\phi}$ at $(x,y)$ tangential to $\tilde{\Gamma}$ (we denote it as $\partial_{\tau}\hat{\phi}(x,y)$) has non-zero component perpendicular to the side of $\mathcal{G}$ containing $\hat{\phi}(x,y)$ (the existence of $\partial_{\tau}\hat{\phi}$ at all points of $P$ is a requirement of the good arrival grid, see Definition \ref{good arrival grid}). Therefore there exists a smallest perpendicular component whose size is $v>0$.

		For each point $a\in P$ we have some $d_{a}>0$ such that when $(x,y)\in \tilde{\Gamma}$ and $|(x,y)-a|< d_{a}$ then
		$$
		\hat{\phi}(x,y)-\hat{\phi}(a) - \partial_{\tau}\hat{\phi}(a)[(x,y)-a]< \frac{v}{3}|(x,y) - a|.
		$$
		Since $P$ is finite we define $d := \min_{a\in P}d_a>0$. As a consequence, the images through $\hat{\phi}$ of the endpoints of the segments $B(a, d)\cap\tilde{\Gamma}$ have distance at least $\frac{vd}{2}$ from $\hat{\phi}(a)$. By making $d$ smaller if necessary we can assume that each $B(a, d)\cap\tilde{\Gamma}$ is a segment and each pair of these (finitely many) segments is disjoint neither does any of the segments contain any vertex of $\tilde{\Gamma}$. The choice of the number $d$ has been made so that the following holds; let $c$ be an endpoint of the segment of $B(a, d)\cap\tilde{\Gamma}$ then $|\hat{\phi}(c) - \hat{\phi}(a)| \geq \frac{vd}{2}$ and this holds for all $a \in P$.
		
		On the other hand, being $\hat{\phi} \big(\tilde{\Gamma} \setminus \bigcup_{a\in P} B(a, d) \big)$ a closed set, it follows that there exists a $\sigma_0>0$ such that
		\begin{equation}\label{Marathon}
		\dist\bigg(\hat{\phi} \Big(\tilde{\Gamma} \setminus \bigcup_{a\in P}B\big(a,d\big)\Big), \mathcal{G}\bigg) \geq 3\sigma_0.
		\end{equation}
		We have that $\tilde{\Gamma}$ is a good non-straight grid and so the $NCBV^+$ condition enjoyed by $f$ garantees the existence of injective uniform approximations of $\hat{\phi}$. Equation \eqref{Marathon} immediately implies that for any $0<\sigma\leq \sigma_0$ and any $\tilde{\phi}_{\sigma}$, continuous injective approximation of $\hat{\phi}$ with $\|\tilde{\phi}_{\sigma} - \hat{\phi}\|_{\infty,\tilde{\Gamma}} \leq \sigma$, it holds that $\tilde{\phi}_{\sigma}^{-1}(\mathcal{G})\subset \bigcup_{a\in P} B\big(a,d\big) \cap \tilde{\Gamma}$.

		Let 
		$$
			\rho'=\min\{|\mathbf{a}-\mathbf{b}|\colon \mathbf{a}\in \hat{\phi}(P),\, \mathbf{b} \textrm{ vertex of }\mathcal{G}\}     
		$$
		and
		$$
			\rho''=\min\{|\mathbf{a}-\mathbf{b}|\colon\,  \mathbf{a}\neq \mathbf{b}, \ \mathbf{a},\mathbf{b}\in \hat{\phi}(P)\}
		$$
		Finally, we set 
		$$
			\rho=\min\{\rho',\,\rho', \sigma_0, \tfrac{1}{100}\}.
		$$
		Notice that $\rho$ is positive due to the properties of good arrival grid.
		Let 
		\begin{equation}\label{SigmaDef}
		0<\sigma\leq \frac{\epsilon^2\rho}{ 12(2^K+1)}.
		\end{equation}
		Then by applying the $NCBV^+$ condition to $\hat{\phi}$ we get a continuous injective $\tilde{\phi}_{\sigma}$ with $\|\tilde{\phi}_{\sigma} - \hat{\phi}\|_{\infty, \Gamma}<\sigma$.
		
		We adjust the map $\tilde{\phi}_{\sigma}$ as follows. For each $a\in P=\hat{\phi}^{-1}(\mathcal G)\cap \tilde{\Gamma}$ we find the first and last point (i.e. the points furthest away from $a$) on the segment $B(a, d)\cap\tilde{\Gamma}$ (call them $a^-$ and $a^+$ respectively) such that $\tilde{\phi}_{\sigma}(a^{\pm}) \in B(\hat{\phi}(a), 2\sigma)$. Notice that \eqref{Marathon} and the choice of $\sigma$ imply that
		\[ \hat{\phi} \Big(\tilde{\Gamma} \setminus \bigcup_{a \in P}B\big(a,d\big)\Big) \cap \Big( \bigcup_{a \in P} B(\hat{\phi}(a), 2\sigma) \Big) = \emptyset,  \] 
		hence we have that $\tilde{\phi}_\sigma$ intersects $B(\hat{\phi}(a), 2\sigma)$ only on the segments $B(a, d)\cap\tilde{\Gamma}$, $a\in P$.
		We define
		$$
		\tilde{\tilde{\phi}}_{\sigma}(t) = \frac{|t - a^+|}{|a^+ - a^-|}\tilde{\phi}_{\sigma}(a^-) +\frac{|t - a^-|}{|a^+ - a^-|} \tilde{\phi}_{\sigma}(a^+) \qquad \mbox{for all $t \in [a^-a^+]$ and $a \in P$}
		$$
		then we set
		$$
		\tilde{\tilde{\phi}}_{\sigma}(t) ={\tilde{\phi}}_{\sigma}(t) \qquad \text{ for } t\in \tilde{\Gamma} \setminus \bigcup_{a\in P} [a^- a^+].
		$$
		By construction, it follows that $\tilde{\tilde{\phi}}_{\sigma}(t)$ is again continuous, injective and $\|\tilde{\tilde{\phi}}_{\sigma} -\hat{\phi}\|_{L^\infty(\tilde{\Gamma})}\leq 7\sigma$. Since $\tilde{\tilde{\phi}}_{\sigma}(a^-)$ and $\tilde{\tilde{\phi}}_{\sigma}(a^+)$ must be separated by $\mathcal{G}$ there is exactly one point $\tilde{a}$ in each $[a^-a^+]$ which is mapped onto $\mathcal{G} \cap B(\hat{\phi}(a), 2\sigma)$.
		
		At this stage we use $\tilde{\tilde{\phi}}_{\sigma}$ and the arrival grid $\mathcal{G}$ to define a piecewise linear map from $\tilde{\Gamma}$ to $\er^2$. We will call this map $\phi$. 
		We start by specifying the image, $\phi(\tilde{\Gamma})$. 
		For each segment of $\tilde{\Gamma}$ we have a finite number of points $\tilde{a}$ such that $\tilde{\tilde{\phi}}_{\sigma}(\tilde{a}) \in \mathcal{G}$. 
		Whenever we have a pair of adjacent points $\tilde{a}_1, \tilde{a}_2$ lying on a common segment of $\tilde{\Gamma}$ such that $\tilde{\tilde{\phi}}_{\sigma}(\tilde{a}_1)$ and $\tilde{\tilde{\phi}}_{\sigma}(\tilde{a}_2)$ lie on two distinct sides of a rectangle in $\mathcal{G}$ we define the segment $S_{a_1,a_2} = [\tilde{\tilde{\phi}}_{\sigma}(\tilde{a}_1)\tilde{\tilde{\phi}}_{\sigma}(\tilde{a}_2)]$ where $a_1$ and $a_2$ are the unique points in $P$ for which $\tilde{a}_i \in B(a_i, d)$.
		
		Let us now consider a pair of adjacent $\tilde{a}_1$ and $\tilde{a}_2$ for which $\tilde{\tilde{\phi}}_{\sigma}(\tilde{a}_1),\tilde{\tilde{\phi}}_{\sigma}(\tilde{a}_2)$ lie on the same side of a rectangle in $\mathcal{G}$. 
		Firstly notice that for any such pair $\tilde{a}_1$ and $\tilde{a}_2$ there exists an $0<\xi_{a_1,a_2}$ so small that the generalized segments (see Definition \ref{generalized segments}) with $\xi = \xi_{a_1,a_2}$ intersect only those previously defined straight segments $S_{a_3,a_4}$ for which $\tilde{\tilde{\phi}}_{\sigma}([\tilde{a}_1\tilde{a}_2])$ was already intersecting $\tilde{\tilde{\phi}}_{\sigma}([\tilde{a}_3\tilde{a}_4])$. We define
		$$
			\xi = \tfrac{1}{2}\min\big\{\xi_{a_1,a_2}: a_1,a_2 \in P \text{ adjacent and } f(a_1), f(a_2) \text{ lie on a common side of }\mathcal{G}\big\}.
		$$
		We assume that $\xi <\kappa < \epsilon$ and define $S_{a_1,a_2}$ as the generalised segment from $\tilde{\tilde{\phi}}_{\sigma}(\tilde{a}_1)$ to $\tilde{\tilde{\phi}}_{\sigma}(\tilde{a}_2)$ with the chosen $\xi$.	
		
		It is very easy to check that any pair $S_{a_1,a_2}$ and $S_{a_3,a_4}$ as defined above, where $a_1,a_2$ and $a_3,a_4$ are pairs of adjacent points of $P$ lying on a common segment of $\tilde{\Gamma}$, intersect each other if and only if $\tilde{\tilde{\phi}}_{\sigma}([\tilde{a}_1\tilde{a}_2])$ intersects $\tilde{\tilde{\phi}}_{\sigma}([\tilde{a}_3\tilde{a}_4])$. Further, any two (distinct) paths can have at most one intersection. 
		
		Now we are in a position to define the map $\phi$ on $\tilde{\Gamma}$. We define $\phi(a) = \tilde{\tilde{\phi}}_{\sigma}(\tilde{a})$ for all $a\in P$ and for all the corresponding $\tilde{a}$. 
		Further, for every $\{X_{i,j}\} = \gamma_i([0,1])\cap\gamma_j([0,1])$, intersection point of the grid $\tilde{\Gamma}$, there exists exactly two pairs of adjacent $a_1,a_2\in P$ (both lying on $\gamma_i$) and $a_3,a_4 \in P$ (both lying on $\gamma_j$) closest to $X_{i,j}$ on $\gamma_i$ and $\gamma_j$ respectively. That is there exists a $t_1,t_2, t_3,t_4 \in [0,1]$ such that $\gamma_i(t_1) = a_1, \gamma_i(t_2) = a_2$, $\gamma_j(t_3) = a_3, \gamma_j(t_3) = a_3$ further $\gamma_{i}((t_1,t_2))\cap P = \emptyset$ and $\gamma_{j}((t_3,t_4))\cap P = \emptyset$.
		
		Then, by construction, there exists exactly one point of intersection call it $\mathbf{X}_{i,j}$ in the set $S_{a_1,a_2}\cap S_{a_3,a_4}$ and we define $\phi(X_{i,j}) = \mathbf{X}_{i,j}$. Thus we have separated the grid $\tilde{\Gamma}$ into simple segments lying between adjacent intersections with $\mathcal{G}$, intersecting segments of $\tilde{\Gamma}$ or a combination of the two. In each case there is a clear correspondence between the endpoints of remaining segments in $\tilde{\Gamma}$ and (parts of the possibly generalized) segments defined in the previous paragraph. We define $\phi$ by parametrizing these segments (or possibly paths consisting of 2 segments) at constant speed from the corresponding segments in $\tilde{\Gamma}$.
		
		Thus we obtain a continuous injective piecewise linear mapping $\phi:\tilde{\Gamma}\to Q(0,1)$ satisfying $\|\hat{\phi} - \phi\|_{L^\infty(\tilde{\Gamma})}\leq 4\kappa \leq \eps 4(2^K+1)^{-1}$. The last step we make in order to define $\phi$ is to redefine it as linear on each side of $\partial\Q_i \in G_{\epsilon, \alpha, K}$ keeping the same values at its vertices. By the definition of $G_{\epsilon, \alpha, K}$ (especially \eqref{Scorpions}) it is not hard to see that this modification keeps $\phi$ continuous, injective and piecewise linear. We refer to \cite{C}, Proof of Theorem 4.1, Step 3, the detailed argument.
		 Thus we have achieved point (5) of the claim.

		 \step{4}{Estimates}{Satisfaction}

		We claim that for any pair of points $a,b \in [ab]\subset \tilde{\Gamma}$ we have that
		\begin{equation}\label{Simple}
			|D_{\tau}\phi|([ab]) \leq (1+\xi)(1+\epsilon) \left(|D_{\tau}\hat{\phi}|([ab]) + 4\kappa \right).
		\end{equation}
		Indeed it holds that $|\hat{\phi}(a) - \phi(a)|\leq 3\sigma \leq \tfrac{1}{2}\epsilon\rho$ for each $a\in P$. Immediately from the definition of $\rho$ we obtain that for any pair $a, a' \in P$ adjacent on a segment of $\tilde{\Gamma}$ we have that either $\hat{\phi}(a) = \hat{\phi}(a')$ or $|\hat{\phi}(a)- \hat{\phi}(a')| \geq \rho.$  In the first case we know that the length of the curve given by $\hat{\phi}$ on $[aa']$ is at least $100\sigma$ by \eqref{Marathon}, $\epsilon<\tfrac{1}{100}$, and $\sigma\leq \epsilon^2 \rho \leq \epsilon \sigma_0$. On the other hand we have $|\phi(a)- \phi(a')| < 6\sigma$ and so the length of $S_{aa'}$ is at most $(1+\xi)6\sigma \leq 12\sigma < 100\sigma$. Therefore, on such segments we in fact have that $|D_{\tau}\phi|([aa']) < |D_{\tau}\hat{\phi}|([aa'])$.
		
		In the second case we have that $|\hat{\phi}(a)- \hat{\phi}(a')| \geq \rho$ and, by \eqref{SigmaDef}, that
		$$
			|\phi(a) - \hat{\phi}(a)| < 3\sigma\leq \epsilon\rho
		$$
		and the same holds also for $a'$. We can estimate by the triangle inequality and \eqref{SigmaDef} that
		$$
			\begin{aligned}
				|\phi(a) - \phi(a')|
				&\leq |\phi(a) - \hat{\phi}(a)|+|\hat{\phi}(a) - \hat{\phi}(a')|+|\hat{\phi}(a') - \phi(a')|\\
				& \leq |\hat{\phi}(a) - \hat{\phi}(a')| + 6\sigma\\
				& \leq (1+\epsilon)|\hat{\phi}(a) - \hat{\phi}(a')|.
			\end{aligned}
		$$
		Now, because the length of the generalized segment between $\mathbf{a}$ and $\mathbf{b}$ with parameter $\xi$ has length bounded by $(1+\xi)|\mathbf{a} - \mathbf{b}|$ (see Proposition~\ref{EverythingIDo}), we get that
		\begin{equation}\label{BasicA}
			|D_{\tau}\phi|([aa']) \leq (1+\xi)(1+\epsilon)|D_{\tau}\hat{\phi}|([aa'])
		\end{equation}
		for any $a,a'\in P$ adjacent on a segment in $\tilde{\Gamma}$. Summing over subsegments we see immediately that the same holds for any $a,a'\in P$ lying on a segment of $\tilde{\Gamma}$ but not necessarily adjacent.
		
		The argument for a general pair $a,b$ both lying on a single segment of $\tilde{\Gamma}$ is as follows. We use the estimate \eqref{BasicA} on the maximal segment $[a_1a_2]\subset [a,b]$ for $a_1, a_2 \in P$. Now it remains to estimate the length of the image of $[aa_1]$ and $[a_2b]$; or if $[ab]\cap P = \emptyset$ then we have to estimate the length of the image of $[ab]$ knowing that $[ab]\cap P = \emptyset$. Let us deal with the former case first; the latter case readily follows from the first. We can deal with each segment separately so let us estimate the length of $\phi([aa_1])$. The image of $[aa_1]$ is a generalized segment contained in some rectangle, i.e. $\{(x,y) \in \er^2: w_n<x<w_{n+1}, z_m < y<z_{m+1}\}$ of $\G$. The diameter of each rectangle of $\G$ is bounded by $2\frac{\epsilon}{2^K+1}$. Then also the diameter of $\phi([aa_1])$ is also bounded by $2\frac{\epsilon}{2^K+1}$. On the other hand by Proposition~\ref{EverythingIDo} the length of $\phi([aa_1])$ is bounded by $(1+\xi)$ times its diameter. Therefore, the length of $\phi([aa_1])$ is bounded by $2(1+\xi)\frac{\epsilon}{2^K+1}$. The same estimate holds for the length of $\phi([a_2b])$ and indeed for $\phi([ab])$ if $[ab]\cap P = \emptyset$. Assuming then that $[ab]\cap P \neq \emptyset$ we combine \eqref{BasicA} with
		$$
			\H^1(\phi([aa_1])), \ \H^1(\phi([a_2b])) \leq (1+\xi)2\frac{\epsilon}{2^K+1}
		$$
		we get \eqref{Simple}, since $\kappa = \frac{\epsilon}{2^K+1}$. In the case that $[ab]\cap P = \emptyset$ then $\H^1(\phi([ab])) \leq (1+\xi)2\frac{\epsilon}{2^K+1}$ and so \eqref{Simple} holds. 
		 
		 The estimate in (4) follows immediately from $(X_V,X_{\tilde{V}})\in H_{V,\tilde{V}}$ for each adjacent $V, \tilde{V}$ and \eqref{NothingEst}. Similarly we get the estimate (6) immediately from \eqref{SomethingEst}.
		  
		 Let us now prove the estimate \eqref{SingEst1}. As in \eqref{Alabama}, each $S_{i,m}$  has either
		 $$
		 	|D_{\tau} f_{\rceil \partial \Q_i}|(S_{i,m}^+), \  |D_{\tau} f_{\rceil \partial \Q_i}|(S_{i,m}^-) < \epsilon 2^{-2K}
		 $$
		 or
		 $$
		 	\max\{|D_{\tau} f_{\rceil \partial \Q_i}|(S_{i,m}^+) , |D_{\tau} f_{\rceil \partial \Q_i}|(S_{i,m}^-) \} \geq \epsilon 2^{-2K}, \text{ but } \H^1\big(\pi_{v_i}(S_{i,m}) \big) < \frac{\epsilon}{2^{2K}|D_{\tau} f_{\rceil \Gamma}|(\Gamma)}.
		 $$
	
		 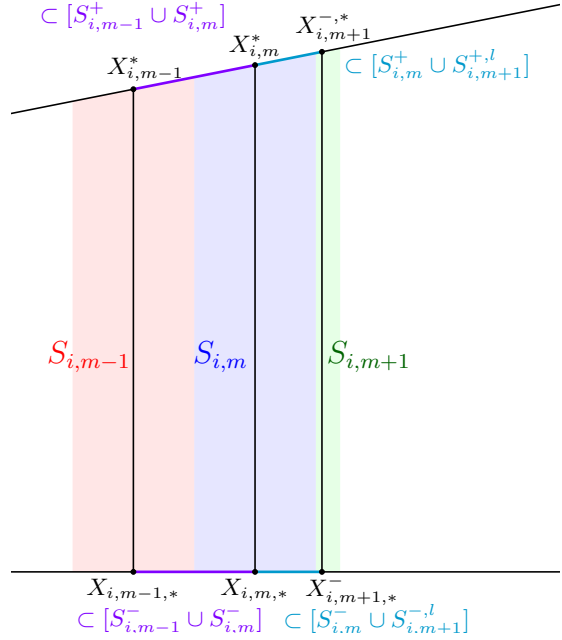
\begin{figure}[h]
		 	\begin{tikzpicture}[line cap=round,line join=round,>=triangle 45,x=1.6cm,y=1.6cm]
		 	\clip(1,3.5) rectangle (5.5,8.8);
		 	\fill[line width=0.6pt,color=ffqqqq,fill=ffqqqq,fill opacity=0.10000000149011612] (1.5,7.9) -- (2.5,8.1) -- (2.5,4.) -- (1.5,4.) -- cycle;
		 	\fill[line width=0.6pt,color=qqffqq,fill=qqffqq,fill opacity=0.10000000149011612] (3.5,8.3) -- (3.7,8.34) -- (3.7,4.) -- (3.5,4.) -- cycle;
		 	\fill[line width=0.6pt,color=qqqqff,fill=qqqqff,fill opacity=0.10000000149011612] (2.5,8.1) -- (3.5,8.3) -- (3.5,4.) -- (2.5,4.) -- cycle;
		 	\draw [line width=0.6pt] (3.,8.2)-- (3.,4.);
		 	\draw [line width=0.6pt] (2.,8.)-- (2.,4.);
		 	\draw [line width=0.6pt] (3.55,8.31)-- (3.55,4.);
		 	\draw [line width=1.0pt,color=xfqqff] (2.,8.)-- (3.,8.2);
		 	\draw [line width=0.6pt] (0.,7.6)-- (2.,8.);
		 	\draw [line width=0.6pt] (3.55,8.31)-- (12.,10.);
		 	\draw [line width=0.6pt] (2.,4.)-- (0.,4.);
		 	\draw [line width=0.6pt] (3.55,4.)-- (14.,4.);
		 	\draw [line width=1.0pt,color=xfqqff] (2.,4.)-- (3.,4.);
		 	\draw [line width=1.0pt,color=qqzzcc] (3.,8.2)-- (3.55,8.31);
		 	\draw [line width=1.0pt,color=qqzzcc] (3.,4.)-- (3.55,4.);
		 	\draw [color=ffqqqq](1.2,6) node[anchor=north west] {$S_{i,m-1}$};
		 	\draw [color=qqqqff](2.4,6) node[anchor=north west] {$S_{i,m}$};
		 	\draw [color=qqwuqq](3.5,6) node[anchor=north west] {$S_{i,m+1}$};
		 	\begin{scriptsize}
		 	\draw [fill=black] (2.,8.) circle (0.9pt);
		 	\draw[color=black] (2.1,8.17) node {$X_{i,m-1}^*$};
		 	\draw [fill=black] (2.,4.) circle (0.9pt);
		 	\draw[color=black] (2.,3.85) node {$X_{i,m-1, *}$};
		 	\draw [fill=black] (3.,8.2) circle (0.9pt);
		 	\draw[color=black] (3.,8.35) node {$X_{i,m}^*$};
		 	\draw [fill=black] (3.,4.) circle (0.9pt);
		 	\draw[color=black] (3.,3.85) node {$X_{i,m,*}$};
		 	\draw [fill=black] (3.55,4.) circle (0.9pt);
		 	\draw[color=black] (3.8,3.85) node {$X_{i,m+1, *}^{-}$};
		 	\draw [fill=black] (3.55,8.31) circle (0.9pt);
		 	\draw[color=black] (3.65,8.52) node {$X_{i,m+1}^{-, *}$};
		 	\draw[color=xfqqff] (2.0,8.6) node {$\subset [S_{i,m-1}^+\cup S_{i,m}^+]$};
		 	\draw[color=xfqqff] (2.3,3.6) node {$\subset [S_{i,m-1}^-\cup S_{i,m}^-]$};
		 	\draw[color=qqzzcc] (4.5,8.2) node {$\subset [S_{i,m}^+\cup S_{i,m+1}^{+,l}]$};
		 	\draw[color=qqzzcc] (4,3.6) node {$\subset [S_{i,m}^-\cup S_{i,m+1}^{-,l}]$};
		 	\end{scriptsize}
		 	\end{tikzpicture}
		 	\caption{A depiction of the layout of the sets $S_{i,m}$, $S_{i,m}^{\pm}$ and the points $X_{i,m}^*$, $X_{i,m,*}$. The set $S_{i,m-1}$ is category-1 and $S_{i,m+1}$ is category-2. The variation along the purple and teal segments is small by the choice of $X_{i,m-1}, X_{i,m}, X_{i,m+1}^-$.}\label{Fig:ASet}
		 \end{figure}
		 We start with the first case. Figure~\ref{Fig:ASet} should help identify the following sets. We define the set
		 $$
		 	\begin{aligned}
		 		A_{i,m}^+&:= \big[S_{i,m-1}^+\cup S_{i,m}^+\cup S_{i,m+1}^+\big] \cap \big\{X\in \er^2: \langle Y_{i,m-1}, v_i^{\bot} \rangle \leq \langle X, v_i^{\bot} \rangle \leq \langle Y_{i,m+1}, v_i^{\bot} \rangle \big\},\\
		 		A_{i,m}^-&:= \big[S_{i,m-1}^-\cup S_{i,m}^-\cup S_{i,m+1}^-\big] \cap \big\{X\in \er^2: \langle Y_{i,m-1}, v_i^{\bot} \rangle \leq \langle X, v_i^{\bot} \rangle \leq \langle Y_{i,m+1}, v_i^{\bot} \rangle \big\},
		 	\end{aligned}
		 $$
		 where $Y_{i,m-1} = X_{i,m-1}$ if $S_{i,m-1}$ is in the first category of \eqref{Alabama} and $Y_{i,m-1} = X_{i,m-1}^+$ if $S_{i,m-1}$ is in the second category of \eqref{Alabama} and $Y_{i,m+1} = X_{i,m+1}$ if $S_{i,m+1}$ is in the first category of \eqref{Alabama} and $Y_{i,m+1} = X_{i,m+1}^-$ if $S_{i,m+1}$ is in the second category of \eqref{Alabama}. In all of the above cases we have
		 $$
		 	|D_{\tau} f_{\rceil \partial \Q_i}|(A_{i,m}^+\cap S_{i,m+1}^+),\  |D_{\tau} f_{\rceil \partial \Q_i}|(A_{i,m}^+\cap S_{i,m-1}^+) \leq \epsilon 2^{-2K}.
		 $$
		 By the choice of $S_{i,m}$ being first category in \eqref{Alabama} we have $|D_{\tau} f_{\rceil \partial \Q_i}|(A_{i,m}^+) \leq 3\epsilon 2^{-2K}$. For each $X^*\in S_{i,m}^+$ we can estimate (using~\eqref{Simple})
		 $$
		 	|D_{\tau} \phi|([X^*X^*_{i,m}]) \leq |D_{\tau} f_{\rceil \partial \Q_i}|(A_{i,m}^+) + 4\kappa \leq 3\epsilon 2^{-2K} + 4\epsilon2^{-K}.
		 $$
		 The above estimate is done on the bottom side in the same way, specifically
		 $$
		 	|D_{\tau} \phi|([X_*X_{i,m,*}]) \leq |D_{\tau} f_{\rceil \partial \Q_i}|(A_{i,m}^-) + 4\kappa \leq 3\epsilon 2^{-2K} + 4\epsilon2^{-K}.
		 $$
		 for all $X_* \in S_{i,m}^-$.
		 
		 For each $X_*\in S_{i,m}^-$ and $X^*\in S_{i,m}^+$ with $\pi_{v_i}(X_*)  = \pi_{v_i}(X^*)$. We define the path $p_{X^*,X_*}\subset \tilde{\Gamma}\cap \overline{\Q_i}$ as
		 $$
		 	p(X^*, X_*) : = [X^* X_{i,m}^*]\cup[X_{i,m}^*X_{i,m,*}]\cup [X_{i,m,*}, X_*].
		 $$
		 We estimate using \eqref{Simple} and \eqref{EstimateHelper} that
		 $$
		 	\begin{aligned}
		 		|D_{\tau}\phi|([X_{i,m}^*X_{i,m,*}]) &\leq (1+\epsilon)(1+\xi)|D_{\tau}f_{\rceil [X_{i,m}+\er v_i]\cap \Q_i}|([X_{i,m}+\er v_i]\cap \Q_i) +4\kappa \\
		 		&\leq (1+\epsilon)(1+\xi)\frac{ |D_{v_i}f|(S_{i,m}\cap \Q_i) }{ \H^1\big(\pi_{v_i}(S_{i,m})\big)} + 4\kappa.
		 	\end{aligned}
		 $$
		 Combining the above estimates we get
		 $$
		 	|D_{\tau} \phi|(p(X^*, X_*)) \leq (1+\epsilon)(1+\xi) \frac{ |D_{v_i}f|(S_{i,m}\cap \Q_i) }{ \H^1\big(\pi_{v_i}(S_{i,m})\big)} +6\epsilon 2^{-2K}+ 12\kappa.
		 $$
		 This means that (see \eqref{dDef})
		 \begin{equation}\label{Services}
		 	d_{\phi(\partial \Q_i)}(\phi(X^*),\phi(X_*)) \leq (1+\epsilon)(1+\xi)\frac{ |D_{v_i}f|(S_{i,m}\cap \Q_i)}{ \H^1\big(\pi_{v_i}(S_{i,m})\big)} +6\epsilon 2^{-2K}+ 12\kappa
		 \end{equation}
		 for every opposing pair $X_*$, $X^*$ in $S_{i,m}$ of first category.
		 
		 For opposing pairs $X_*, X^*$ in $S_{i,m}$ for $S_{i,m}$ in the second category of \eqref{Alabama} we use the estimate
		 \begin{equation}\label{TheBad}
		 	\begin{aligned}
		 		d_{\phi(\partial \Q_i)}(\phi(X^*),\phi(X_*)) &\leq  (1+\epsilon)(1+\xi) |D_{\tau}f_{\rceil \Gamma}|(\partial\Q_i) +4\kappa\\
		 		& \leq (1+\epsilon)(1+\xi) |D_{\tau}f_{\rceil \Gamma}|(\Gamma) +4\kappa.
		 	\end{aligned}
		 \end{equation}
		 
		 For opposing pairs in the perpendicular direction, i.e. $Z_*\in T_{i,m}^{-}, Z^* \in T_{i,m}^{+}$ we estimate in the same way using \eqref{EstimateHelperT} in place of \eqref{EstimateHelper}. For $Z^* \in T_{i,m}^{+}, Z_*\in T_{i,m}^{-}$ we get
		 \begin{equation}\label{TheUgly}
		 		d_{\phi(\partial \Q_i)}(\phi(Z^*),\phi(Z_*)) \leq \begin{cases}
		 		\frac{  (1+\epsilon)(1+\xi)|D_{v_i^{\bot}}f|(T_{i,m}\cap \Q_i) }{ \H^1\big(\pi_{v_i^{\bot}}(T_{i,m})\big)} +6  \epsilon 2^{-2K}+ 12\kappa\quad &T_{i,m}^{\pm} \text{ is category 1}\\
		 		 (1+\epsilon)(1+\xi)|D_{\tau}f_{\rceil \Gamma}|(\Gamma) +4\kappa \quad &T_{i,m}^{\pm} \text{ is category 2}.
		 		\end{cases}
		 \end{equation}
		 
		 Find $\theta_i \in [0,2\pi)$ such that $v_i = (\cos\theta_i, \sin\theta_i)$. We denote the set of indexes $m$ such that $S_{i,m}$ is category 1 by $C_1$ and category 2 by $C_2$. We denote the set of indexes $m$ such that $T_{i,m}$ is category 1 by $C_3$ and category 2 by $C_4$. We use the admissibility of $\tilde{\Gamma}$ (especially $|Df|(\tilde{\Gamma}) = 0$) to prove \eqref{SingEst1}
		$$
		 	\begin{aligned}
		 		\int_{\pi_{v_i}(\Q_i)}&d_{\phi(\partial \Q_i)}(\phi(X^*),\phi(X_*)) \, d\H^1(X)\\
		 		& \leq (1+\epsilon)(1+\xi)\Big[\sum_{m\in C_1}  |D_{v_i}f|(S_{i,m}\cap \Q_i) \\
		 		&\quad+  \sum_{m\in C_2}|D_{\tau}f_{\rceil{\Gamma}}|(\Gamma)   \frac{\epsilon}{2^{2K}|D_{\tau} f_{\rceil \Gamma}|(\Gamma) \mathfrak{K}}  \Big]\\
		 		&\quad +[6\epsilon 2^{-2K}+ 12\kappa]\mathcal{H}^1(\pi_{v_i}(\Q_i))\\
				&\leq (1+\epsilon)(1+\xi)| Df|(\Q_i) +C2^{-2K}\epsilon.
		 	\end{aligned}
		$$
		Recall our notation $g|D^sf| = D^sf$ and by the Theorem~\ref{ARO} we have $g = u(X)\otimes v(X)$ for $|D^sf|$-almost every $X$. We have $|\langle v(X), v_i^{\bot}\rangle| \leq C|v(X) - v_i| +C |\langle v_i, v_i^{\bot}\rangle|$. Using this and \eqref{Piano}  we calculate
		$$
			\begin{aligned}
				\int_{\pi_{v_i^{\bot}}(\Q_i)}&d_{\phi(\partial \Q_i)}(\phi(Z^*),\phi(Z_*)) \, d\H^1(Z)\\
				& \leq (1+\epsilon)(1+\xi)\Big[\sum_{m\in C_3}  |D_{v_i^{\bot}}f|(T_{i,m}\cap \Q_i) +\sum_{m\in C_4} |D_{\tau}f_{\rceil{\Gamma}}|(\Gamma)\frac{\epsilon}{2^{2K}|D_{\tau} f_{\rceil \Gamma}|(\Gamma)\mathfrak{K}}\Big]\\ 
				&\quad+[6  \epsilon 2^{-2K}+ 12\kappa]\mathcal{H}^1(\pi_{v_i}(\Q_i))\\
				&\leq (1+\epsilon)(1+\xi)|\langle Df, v_i^{\bot}\rangle|(\Q_i) +C2^{-K}\epsilon\\
				&\leq (1+\epsilon)(1+\xi)\epsilon |D^sf|(4Q_i \cap S ) +C2^{-2K}\epsilon
			\end{aligned}
		$$
		Proving \eqref{SingEst2}.
	\end{proof}
	
	Now it suffices to extend the mapping $\phi$ defined in Proposition~\ref{DefiningPhi} to get a $BV$ homeomorphism. This is the content of the following theorem.
	
	\begin{thm}\label{TheMeat}
		Let $f\in BV(Q(0,1); Q(0,1))$, let $f(x,y) = (x,y)$ for $(x,y)\in \partial Q(0,1)$ and let $f$ satisfy the $NCBV^+$ condition. Then there exists a sequence $f_k\in BV(Q(0,1), Q(0,1))$, $f_k = \id$ on $\partial Q(0,1)$ converging to $f$ area-strictly.
	\end{thm}
	\begin{proof}
		Given $\epsilon >0$ we get a good non-straight grid for $f$ called $\Gamma$ and a mapping $\phi$ from Proposition~\ref{DefiningPhi}.
		
		\step{1}{Extend $\phi$ to get a homeomorphism}{MS1}
		
		Let us start by extending $\phi$ on the `good' quadrilaterals, i.e. on $\Q_i$ such that $Q_i \in G_{\epsilon,\alpha,K}$. We have that $\phi$ is linear on each side of $\Q_i$. Further by the choice of $X,X'$, the vertices of $\Q_i$, (specifically \eqref{InZ}) we have
		\begin{equation}\label{Noice}
			|f(X) - f(X') - \nabla f(w_i)(X - X')|< \alpha^4 2^{-K}.
		\end{equation}
		Recall that we have $\alpha\leq |\nabla f(w_i)|\leq \alpha^{-1}, \ \alpha< \det \nabla f(w_i)$, which implies that $|\nabla f(w_i)v| \geq \alpha^2$ for all $|v|=1$. Therefore, since $|X-X'| \geq 2^{-K-1}$, we have $|\nabla f(w_i)(X - X')| \geq 2^{-K-1}\alpha^2$. By Proposition~\ref{DefiningPhi}, step~\ref{CoG} we have that $\Q_i$ is convex and so also is $\nabla f(w_i)\Q_i$. 
		Therefore \eqref{Noice} guarantees that $\phi(\partial \Q_i)$ is a convex quadrilateral. Then we can choose any pair of opposing corners of $\Q_i$, call them $X$ and $X'$. We define $\phi$ as linear continuous on the segment $[XX']$. Then $\Q_i$ is composed of two triangles and $\phi$ is continuous and linear on each side of each triangle and therefore $\phi$ extends to an affine map on each triangle, which we call $g$. By the convexity of $\phi(\partial \Q_i)$ we get the injectivity of the map $g$ on each $\Q_i$.
		
		On the quadrilaterals $\Q_i$ such that $Q_i\in T_{\epsilon, \alpha, K}$ we apply Theorem~\ref{NullExt}. For each quadrilateral $\Q_i$ such that $Q_i\in E_{\epsilon, K}$ we find a $\theta_i$ such that $v_i = (\cos\theta_i, \sin\theta_i)$ and we apply Theorem~\ref{Componentwise Extension} with the parameter $\epsilon$ from Theorem~\ref{Componentwise Extension} $\epsilon2^{-2K}$. On the other quadrilaterals, i.e. $\Q_i$ such that $Q_i\in W_{\epsilon, \alpha, K} \cup [F_{\epsilon, K} \setminus E_{\epsilon, K}]$ we apply Theorem~\ref{HPExt}.
		
		The map $g$ is injective on $Q(0,1)$ because it is injective on each $\Q_i$, because $g = \phi$ on $\Gamma$ and because $\phi$ is injective on $\Gamma$.
		
		\step{2}{Convergence estimates}{MS2}
		
		It is not hard to observe the $L^1$ convergence. Since $\|g\|_{\infty}<2$ and the set $\L^2(\tilde{F}_{\epsilon, K}\cup \tilde{P}_{\alpha_0})<\delta < \epsilon$ it remains to consider the set $Q(0,1) \setminus [\tilde{F}_{\epsilon, K} \cup \tilde{P}_{\alpha_0}]$. On each of the $\Q_i$  we use \eqref{Scorpions} and the Poincar\'e inequality to get
		$$
			\|f-g\|_{L^1(\Q_i)} \leq C2^{-3K} \alpha_0^{-1}
		$$
		and since $2^{-K} < \epsilon \alpha_0$ we get that $\|f-g\|_{1}<C\epsilon$.
		
		Let us call $\Omega = Q(0,1)\setminus \bigcup_{Q_i\in E_{\epsilon,K}}\Q_i$. We want to estimate $\|Dg\chi_{\Omega} - D^af\|_{L^1(Q(0,1))}$. For $\L^2$-almost all points of $\Q_i$ with $Q_i\in G_{\epsilon, \alpha, K}$ we have $|Dg - \nabla f(w_i)| \leq 2\alpha^4$ by \eqref{Noice} and we chose $\alpha < \epsilon$. Therefore using \eqref{Scorpions} we have
		$$
			\int_{\Q_i} |Dg - D^af| \leq 3\epsilon^3 \L^2(\Q_i).
		$$
		Summing this over $Q_i\in G_{\epsilon,\alpha, K}$ we bound the sum by $\epsilon$.
		
		Now we calculate on $\Q_i$ for $Q_i\in T_{\epsilon,\alpha, K}$. Summing the estimate from Theorem~\ref{NullExt} using the estimate from \eqref{TropicalWood} we get 
		$$
			\sum_i \int_{\Q_i} |Dg - D^af| \leq C\epsilon |Df|(Q).
		$$
		
		By the definition of $W_{\epsilon,\alpha, K}$ and $F_{\epsilon, K} \setminus E_{\epsilon, K}$ we have that
		$$
			|Df|\left(\bigcup_{Q_i \in W_{\epsilon,\alpha, K} \cup [F_{\epsilon, K} \setminus E_{\epsilon, K}]} \Q_i\right) \leq C\epsilon.
		$$
		On each $\Q_i$ such that $Q_i\in W_{\epsilon, \alpha, K} \cup [F_{\epsilon, K} \setminus E_{\epsilon, K}]$ we have
		$$
			|Dg|\bigg(\bigcup_{Q_i \in W_{\epsilon,\alpha, K} \cup [F_{\epsilon, K} \setminus E_{\epsilon, K}]}\Q_i\bigg)\leq C|Df|\bigg(\bigcup_{Q_i \in W_{\epsilon,\alpha, K} \cup [F_{\epsilon, K} \setminus E_{\epsilon, K}]} \Q_i\bigg) \leq C\epsilon.
		$$
		Therefore, recalling that $\|D^a f\|_{L^1(Q(0,1)\setminus\Omega)} < \epsilon |D^af|(Q(0,1))$ by Lemma~\ref{Isolationism}, we conclude that
		$$
			\|(Dg)\chi_{\Omega} - D^af\|_{L^1(Q(0,1))} < C\epsilon.
		$$
		
		It remains to prove that $\|Dg\|_{L^1(Q(0,1)\setminus \Omega)} \leq |D^sf|(Q(0,1)) + C\epsilon$ in order to prove the area strict convergence. In the following we use the fact that
		$$
			|Dg|(\Q_i) \leq |\langle Dg, v_i\rangle|(\Q_i) + |\langle Dg, v_i^{\bot}\rangle|(\Q_i).
		$$
		On the set in question we defined $g$ using Theorem~\ref{Componentwise Extension} and thanks to \eqref{SingEst1}, \eqref{SingEst2} and Theorem~\ref{CEZ}, points 2) and 3) we have
		$$
			\begin{aligned}
				|Dg|(Q(0,1)\setminus \Omega) & \leq \sum_{Q_i\in E_{\epsilon,K}} |\langle Dg, v_i\rangle|(\Q_i) + |\langle Dg, v_i^{\bot}\rangle|(\Q_i)\\
				&\leq \sum_{Q_i\in E_{\epsilon,K}}(1+\epsilon)|Df|(\Q_i) + C\epsilon 2^{-2K} + C\epsilon |D^sf|(4Q_i \cap S ) \\
				&\leq (1+C\epsilon)|D^sf|(Q(0,1)) + C\epsilon
			\end{aligned}
		$$
		because $\sum_i \chi_{4Q_i} \leq 25$.
	\end{proof}

	\begin{proof}[Proof of Theorem~\ref{main}]
		The equivalence of points 1) and 2) is in Theorem~\ref{ItsAllLies}. The implication 2) implies 3) is Theorem~\ref{TheMeat}. Trivially 3) implies 4). The equivalence of 4) and 1) was proved in \cite{CKR}.
	\end{proof}

\end{document}